\documentclass[11pt]{amsart}

\usepackage{amsmath}
\usepackage{amssymb}
\usepackage{amscd}
\usepackage{color}
\usepackage{hyperref}
\usepackage{mathrsfs}
\usepackage{eucal}
\usepackage{bm}
\usepackage{upgreek}

\usepackage{tikz}

\usepackage{xy}
\xyoption{all}

\newcommand{\nc}{\newcommand}

\nc{\A}{{\mathbb{A}}}
\nc{\CC}{{\mathbb{C}}}
\nc{\HH}{\mathbf{H}_*}
\nc{\LL}{{\mathbb{L}}}
\nc{\RR}{{\mathbb{R}}}
\renewcommand{\P}{{\mathbb{P}}}
\nc{\OO}{{\mathbb{O}}}
\renewcommand{\SS}{{\mathbb{S}}}
\nc{\QQ}{{\mathbb{Q}}}
\nc{\ZZ}{{\mathbb{Z}}}

\nc{\cA}{{\mathcal{A}}}
\nc{\cB}{{\mathcal{B}}}
\nc{\cC}{{\mathcal{C}}}
\nc{\cD}{{\mathcal{D}}}
\nc{\cE}{{\mathcal{E}}}
\nc{\cF}{{\mathcal{F}}}
\nc{\cG}{{\mathcal{G}}}
\nc{\cH}{{\mathcal{H}}}
\nc{\cI}{{\mathcal{I}}}
\nc{\cJ}{{\mathcal{J}}}
\nc{\cK}{{\mathcal{K}}}
\nc{\cL}{{\mathcal{L}}}
\nc{\cM}{{\mathcal{M}}}
\nc{\cN}{{\mathcal{N}}}
\nc{\cO}{{\mathcal{O}}}
\nc{\cP}{{\mathcal{P}}}
\nc{\cQ}{{\mathcal{Q}}}
\nc{\cR}{{\mathcal{R}}}
\nc{\cS}{{\mathcal{S}}}
\nc{\cT}{{\mathcal{T}}}
\nc{\cU}{{\mathcal{U}}}
\nc{\cV}{{\mathcal{V}}}
\nc{\cW}{{\mathcal{W}}}
\nc{\cX}{{\mathcal{X}}}
\nc{\cY}{{\mathcal{Y}}}
\nc{\cZ}{{\mathcal{Z}}}

\nc{\rc}{{\mathrm{c}}}
\nc{\rd}{{\mathrm{d}}}

\nc{\rA}{{\mathrm{A}}}
\nc{\rC}{{\mathrm{C}}}
\nc{\rD}{{\mathrm{D}}}
\nc{\rE}{{\mathrm{E}}}
\nc{\rF}{{\mathrm{F}}}
\nc{\rH}{{\mathrm{H}}}
\nc{\rK}{{\mathrm{K}}}
\nc{\rL}{{\mathrm{L}}}
\nc{\rM}{{\mathrm{M}}}
\nc{\rQ}{{\mathrm{Q}}}
\nc{\rU}{{\mathrm{U}}}

\nc{\rW}{{\mathbf{W}}}
\nc{\rWnu}{{\mathbf{W\!}_{\mathrm{nu}}}}

\nc{\bA}{{\mathbf{A}}}
\nc{\bB}{{\mathbf{B}}}
\nc{\bC}{{\mathbf{C}}}
\nc{\bD}{{\mathbf{D}}}
\nc{\bE}{{\mathbf{E}}}
\nc{\bF}{{\mathbf{F}}}
\nc{\bG}{{\mathbf{G}}}
\nc{\bH}{{\mathbf{H}}}
\nc{\bI}{{\mathbf{I}}}
\nc{\bJ}{{\mathbf{J}}}
\nc{\bK}{{\mathbf{K}}}
\nc{\bL}{{\mathbf{L}}}
\nc{\bM}{{\mathbf{M}}}
\nc{\bN}{{\mathbf{N}}}
\nc{\bO}{{\mathbf{O}}}
\nc{\bP}{{\mathbf{P}}}
\nc{\bQ}{{\mathbf{Q}}}
\nc{\bR}{{\mathbf{R}}}
\nc{\bS}{{\mathbf{S}}}
\nc{\bT}{{\mathbf{T}}}
\nc{\bU}{{\mathbf{U}}}
\nc{\bV}{{\mathbf{V}}}
\nc{\bW}{{\mathbf{W}}}
\nc{\bX}{{\mathbf{X}}}
\nc{\bY}{{\mathbf{Y}}}
\nc{\bZ}{{\mathbf{Z}}}

\nc{\ba}{{\mathbf{a}}}
\nc{\bb}{{\mathbf{b}}}
\nc{\bc}{{\mathbf{c}}}
\nc{\bd}{{\mathbf{d}}}
\nc{\be}{{\mathbf{e}}}
\nc{\bg}{{\mathbf{g}}}
\nc{\bh}{{\mathbf{h}}}
\nc{\bi}{{\mathbf{i}}}
\nc{\bj}{{\mathbf{j}}}
\nc{\bk}{{\mathbf{k}}}
\nc{\bl}{{\mathbf{l}}}
\nc{\bn}{{\mathbf{n}}}
\nc{\bo}{{\mathbf{o}}}
\nc{\bp}{{\mathbf{p}}}
\nc{\bq}{{\mathbf{q}}}
\nc{\br}{{\mathbf{r}}}
\nc{\bs}{{\mathbf{s}}}
\nc{\bu}{{\mathbf{u}}}
\nc{\bv}{{\mathbf{v}}}
\nc{\bw}{{\mathbf{w}}}
\nc{\bx}{{\mathbf{x}}}
\nc{\by}{{\mathbf{y}}}
\nc{\bz}{{\mathbf{z}}}

\nc{\fA}{{\mathfrak{A}}}
\nc{\fB}{{\mathfrak{B}}}
\nc{\fC}{{\mathfrak{C}}}
\nc{\fD}{{\mathfrak{D}}}
\nc{\fE}{{\mathfrak{E}}}
\nc{\fF}{{\mathfrak{F}}}
\nc{\fG}{{\mathfrak{G}}}
\nc{\fH}{{\mathfrak{H}}}
\nc{\fI}{{\mathfrak{I}}}
\nc{\fJ}{{\mathfrak{J}}}
\nc{\fK}{{\mathfrak{K}}}
\nc{\fL}{{\mathfrak{L}}}
\nc{\fM}{{\mathfrak{M}}}
\nc{\fN}{{\mathfrak{N}}}
\nc{\fO}{{\mathfrak{O}}}
\nc{\fP}{{\mathfrak{P}}}
\nc{\fQ}{{\mathfrak{Q}}}
\nc{\fR}{{\mathfrak{R}}}
\nc{\fS}{{\mathfrak{S}}}
\nc{\fT}{{\mathfrak{T}}}
\nc{\fU}{{\mathfrak{U}}}
\nc{\fV}{{\mathfrak{V}}}
\nc{\fW}{{\mathfrak{W}}}
\nc{\fX}{{\mathfrak{X}}}
\nc{\fY}{{\mathfrak{Y}}}
\nc{\fZ}{{\mathfrak{Z}}}

\nc{\fa}{{\mathfrak{a}}}
\nc{\fb}{{\mathfrak{b}}}
\nc{\fc}{{\mathfrak{c}}}
\nc{\fd}{{\mathfrak{d}}}
\nc{\fe}{{\mathfrak{e}}}
\nc{\ff}{{\mathfrak{f}}}
\nc{\fg}{{\mathfrak{g}}}
\nc{\fh}{{\mathfrak{h}}}
\nc{\fj}{{\mathfrak{j}}}
\nc{\fk}{{\mathfrak{k}}}
\nc{\fl}{{\mathfrak{l}}}
\nc{\fm}{{\mathfrak{m}}}
\nc{\fn}{{\mathfrak{n}}}
\nc{\fo}{{\mathfrak{o}}}
\nc{\fp}{{\mathfrak{p}}}
\nc{\fq}{{\mathfrak{q}}}
\nc{\fr}{{\mathfrak{r}}}
\nc{\fs}{{\mathfrak{s}}}
\nc{\ft}{{\mathfrak{t}}}
\nc{\fu}{{\mathfrak{u}}}
\nc{\fv}{{\mathfrak{v}}}
\nc{\fw}{{\mathfrak{w}}}
\nc{\fx}{{\mathfrak{x}}}
\nc{\fy}{{\mathfrak{y}}}
\nc{\fz}{{\mathfrak{z}}}

\nc{\sA}{{\mathsf{A}}}
\nc{\sB}{{\mathsf{B}}}
\nc{\sC}{{\mathsf{C}}}
\nc{\sD}{{\mathsf{D}}}
\nc{\sE}{{\mathsf{E}}}
\nc{\sF}{{\mathsf{F}}}
\nc{\sG}{{\mathsf{G}}}
\nc{\sH}{{\mathsf{H}}}
\nc{\sI}{{\mathsf{I}}}
\nc{\sJ}{{\mathsf{J}}}
\nc{\sK}{{\mathsf{K}}}
\nc{\sL}{{\mathsf{L}}}
\nc{\sM}{{\mathsf{M}}}
\nc{\sN}{{\mathsf{N}}}
\nc{\sO}{{\mathsf{O}}}
\nc{\sP}{{\mathsf{P}}}
\nc{\sQ}{{\mathsf{Q}}}
\nc{\sR}{{\mathsf{R}}}
\nc{\sS}{{\mathsf{S}}}
\nc{\sT}{{\mathsf{T}}}
\nc{\sU}{{\mathsf{U}}}
\nc{\sV}{{\mathsf{V}}}
\nc{\sW}{{\mathsf{W}}}
\nc{\sX}{{\mathsf{X}}}
\nc{\sY}{{\mathsf{Y}}}
\nc{\sZ}{{\mathsf{Z}}}

\nc{\sa}{{\mathsf{a}}}
\nc{\sd}{{\mathsf{d}}}
\nc{\se}{{\mathsf{e}}}
\nc{\sg}{{\mathsf{g}}}
\nc{\sh}{{\mathsf{h}}}
\nc{\si}{{\mathsf{i}}}
\nc{\sj}{{\mathsf{j}}}
\nc{\sk}{{\mathsf{k}}}
\nc{\sm}{{\mathsf{m}}}
\nc{\sn}{{\mathsf{n}}}
\nc{\so}{{\mathsf{o}}}
\nc{\sq}{{\mathsf{q}}}
\nc{\sr}{{\mathsf{r}}}
\nc{\st}{{\mathsf{t}}}
\nc{\su}{{\mathsf{u}}}
\nc{\sv}{{\mathsf{v}}}
\nc{\sw}{{\mathsf{w}}}
\nc{\sx}{{\mathsf{x}}}
\nc{\sy}{{\mathsf{y}}}
\nc{\sz}{{\mathsf{z}}}

\nc{\oA}{{\overline{A}}}
\nc{\oB}{{\overline{B}}}
\nc{\oC}{{\overline{C}}}
\nc{\oD}{{\overline{D}}}
\nc{\oE}{{\overline{E}}}
\nc{\oF}{{\overline{F}}}
\nc{\oG}{{\overline{G}}}
\nc{\oH}{{\overline{H}}}
\nc{\oI}{{\overline{I}}}
\nc{\oJ}{{\overline{J}}}
\nc{\oK}{{\overline{K}}}
\nc{\oL}{{\overline{L}}}
\nc{\oM}{{\overline{M}}}
\nc{\oN}{{\overline{N}}}
\nc{\oO}{{\overline{O}}}
\nc{\oP}{{\overline{P}}}
\nc{\oQ}{{\overline{Q}}}
\nc{\oR}{{\overline{R}}}
\nc{\oS}{{\overline{S}}}
\nc{\oT}{{\overline{T}}}
\nc{\oU}{{\overline{U}}}
\nc{\oV}{{\overline{V}}}
\nc{\oW}{{\overline{W}}}
\nc{\oX}{{\overline{X}}}
\nc{\oY}{{\overline{Y}}}
\nc{\oZ}{{\overline{Z}}}

\nc{\oa}{{\overline{a}}}
\nc{\ob}{{\overline{b}}}
\nc{\oc}{{\overline{c}}}
\nc{\od}{{\overline{d}}}
\nc{\of}{{\overline{f}}}
\nc{\og}{{\overline{g}}}
\nc{\oh}{{\overline{h}}}
\nc{\oi}{{\overline{i}}}
\nc{\oj}{{\overline{j}}}
\nc{\ok}{{\overline{k}}}
\nc{\ol}{{\overline{l}}}
\nc{\om}{{\overline{m}}}
\nc{\on}{{\overline{n}}}
\nc{\oo}{{\overline{o}}}
\nc{\op}{{\overline{p}}}
\nc{\oq}{{\overline{q}}}
\nc{\os}{{\overline{s}}}
\nc{\ot}{{\overline{t}}}
\nc{\ou}{{\overline{u}}}
\nc{\ov}{{\overline{v}}}
\nc{\ow}{{\overline{w}}}
\nc{\ox}{{\overline{x}}}
\nc{\oy}{{\overline{y}}}
\nc{\oz}{{\overline{z}}}

\nc{\tA}{{\tilde{A}}}
\nc{\tB}{{\tilde{B}}}
\nc{\tC}{{\tilde{C}}}
\nc{\tD}{{\tilde{D}}}
\nc{\tE}{{\tilde{E}}}
\nc{\tF}{{\tilde{F}}}
\nc{\tG}{{\tilde{G}}}
\nc{\tH}{{\tilde{H}}}
\nc{\tI}{{\tilde{I}}}
\nc{\tJ}{{\tilde{J}}}
\nc{\tK}{{\tilde{K}}}
\nc{\tL}{{\tilde{L}}}
\nc{\tM}{{\tilde{M}}}
\nc{\tN}{{\tilde{N}}}
\nc{\tO}{{\tilde{O}}}
\nc{\tP}{{\tilde{P}}}
\nc{\tQ}{{\tilde{Q}}}
\nc{\tR}{{\tilde{R}}}
\nc{\tS}{{\tilde{S}}}
\nc{\tT}{{\tilde{T}}}
\nc{\tU}{{\tilde{U}}}
\nc{\tV}{{\tilde{V}}}
\nc{\tW}{{\tilde{W}}}
\nc{\tX}{{\tilde{X}}}
\nc{\tY}{{\tilde{Y}}}
\nc{\tZ}{{\tilde{Z}}}

\nc{\ta}{{\tilde{a}}}
\nc{\tb}{{\tilde{b}}}
\nc{\tc}{{\tilde{c}}}
\nc{\td}{{\tilde{d}}}
\nc{\te}{{\tilde{e}}}
\nc{\tf}{{\tilde{f}}}
\nc{\tg}{{\tilde{g}}}
\nc{\ti}{{\tilde{i}}}
\nc{\tj}{{\tilde{j}}}
\nc{\tk}{{\tilde{k}}}
\nc{\tl}{{\tilde{l}}}
\nc{\tm}{{\tilde{m}}}
\nc{\tn}{{\tilde{n}}}
\nc{\tq}{{\tilde{q}}}
\nc{\tr}{{\tilde{r}}}
\nc{\ts}{{\tilde{s}}}
\nc{\tu}{{\tilde{u}}}
\nc{\tv}{{\tilde{v}}}
\nc{\tw}{{\tilde{w}}}
\nc{\tx}{{\tilde{x}}}
\nc{\ty}{{\tilde{y}}}
\nc{\tz}{{\tilde{z}}}

\nc{\hA}{{\hat{A}}}
\nc{\hB}{{\hat{B}}}
\nc{\hC}{{\hat{C}}}
\nc{\hD}{{\hat{D}}}
\nc{\hE}{{\hat{E}}}
\nc{\hF}{{\hat{F}}}
\nc{\hG}{{\hat{G}}}
\nc{\hH}{{\hat{H}}}
\nc{\hI}{{\hat{I}}}
\nc{\hJ}{{\hat{J}}}
\nc{\hK}{{\hat{K}}}
\nc{\hL}{{\hat{L}}}
\nc{\hM}{{\hat{M}}}
\nc{\hN}{{\hat{N}}}
\nc{\hO}{{\hat{O}}}
\nc{\hP}{{\hat{P}}}
\nc{\hQ}{{\hat{Q}}}
\nc{\hR}{{\hat{R}}}
\nc{\hS}{{\hat{S}}}
\nc{\hT}{{\hat{T}}}
\nc{\hU}{{\hat{U}}}
\nc{\hV}{{\hat{V}}}
\nc{\hW}{{\hat{W}}}
\nc{\hX}{{\hat{X}}}
\nc{\hY}{{\hat{Y}}}
\nc{\hZ}{{\hat{Z}}}

\nc{\ha}{{\hat{a}}}
\nc{\hb}{{\hat{b}}}
\nc{\hc}{{\hat{c}}}
\nc{\hd}{{\hat{d}}}
\nc{\he}{{\hat{e}}}
\nc{\hf}{{\hat{f}}}
\nc{\hg}{{\hat{g}}}
\nc{\hh}{{\hat{h}}}
\nc{\hi}{{\hat{i}}}
\nc{\hj}{{\hat{j}}}
\nc{\hk}{{\hat{k}}}
\nc{\hn}{{\hat{n}}}
\nc{\ho}{{\hat{o}}}
\nc{\hp}{{\hat{p}}}
\nc{\hq}{{\hat{q}}}
\nc{\hr}{{\hat{r}}}
\nc{\hs}{{\hat{s}}}
\nc{\hu}{{\hat{u}}}
\nc{\hv}{{\hat{v}}}
\nc{\hx}{{\hat{x}}}
\nc{\hy}{{\hat{y}}}
\nc{\hz}{{\hat{z}}}

\nc{\eps}{\varepsilon}
\nc{\teps}{\tilde\varepsilon}
\nc{\beps}{{\boldsymbol{\varepsilon}}}
\nc{\tcE}{\tilde\cE}
\nc{\hcE}{\hat\cE}
\nc{\lam}{\lambda}
\nc{\lan}{\big\langle}
\nc{\ran}{\big\rangle}
\DeclareMathOperator{\HE}{HE}
\nc{\kk}{{\mathsf{k}}}

\nc{\opp}{{\mathrm{opp}}}

\nc{\Db}{\mathbf{D}}

\def\bw#1#2{\textstyle{\bigwedge\hskip-0.9mm^{#1}}\hskip0.2mm{#2}}

\DeclareMathOperator{\Hom}{\mathrm{Hom}}
\DeclareMathOperator{\Ext}{\mathrm{Ext}}
\DeclareMathOperator{\cExt}{\mathcal{E}\!\mathit{xt}}
\DeclareMathOperator{\End}{\mathrm{End}}
\DeclareMathOperator{\cEnd}{\mathcal{E}\mathit{nd}}

\DeclareMathOperator{\cRHom}{\mathrm{R}\mathcal{H}\mathit{om}}

\DeclareMathOperator{\Spec}{\mathrm{Spec}}

\DeclareMathOperator{\Bl}{\mathrm{Bl}}

\DeclareMathOperator{\Pic}{\mathrm{Pic}}
\DeclareMathOperator{\Br}{\mathrm{Br}}
\DeclareMathOperator{\Cl}{\mathrm{Cliff}}

\DeclareMathOperator{\Sym}{\mathrm{Sym}}
\DeclareMathOperator{\Ker}{\mathrm{Ker}}
\DeclareMathOperator{\Coker}{\mathrm{Coker}}
\DeclareMathOperator{\Ima}{\mathrm{Im}}
\DeclareMathOperator{\Cone}{\mathrm{Cone}}

\DeclareMathOperator{\Tr}{\mathrm{Tr}}
\DeclareMathOperator{\diag}{\mathrm{diag}}

\DeclareMathOperator{\Grr}{\mathrm{K}_0(\mathrm{Var/\kk})}

\DeclareMathOperator{\LGr}{\mathrm{LGr}}

\DeclareMathOperator{\id}{\mathrm{id}}
\DeclareMathOperator{\rk}{\mathrm{rk}}

\DeclareMathOperator{\rw}{\mathrm{w}}
\DeclareMathOperator{\hw}{\mathrm{hw}}
\DeclareMathOperator{\codim}{\mathrm{codim}}
\DeclareMathOperator{\supp}{\mathrm{supp}}

\DeclareMathOperator{\gr}{\mathrm{gr}}
\DeclareMathOperator{\hl}{\ell}

\newcommand{\bt}{{\mathrm{bot}}}
\newcommand{\tp}{{\mathrm{top}}}
\newcommand{\la}{{\mathrm{L}}}
\newcommand{\ua}{{\mathrm{U}}}
\newcommand{\mn}{{\mathrm{min}}}
\newcommand{\un}{{\mathrm{uni}}}

\newcommand{\Sha}{\mathbf{Sh}}
\newcommand{\moplus}{\mathop{\textstyle\bigoplus}\limits}

\numberwithin{equation}{section}

\theoremstyle{plain}

\newtheorem{theorem}{Theorem}[section]

\newtheorem{question}[theorem]{Question}
\newtheorem{lemma}[theorem]{Lemma}
\newtheorem{proposition}[theorem]{Proposition}
\newtheorem{corollary}[theorem]{Corollary}

\theoremstyle{definition}

\newtheorem{definition}[theorem]{Definition}

\newtheorem{example}[theorem]{Example}

\theoremstyle{remark}

\newtheorem{remark}[theorem]{Remark}

\title{Quadric bundles and hyperbolic equivalence}

\author{Alexander Kuznetsov}
\address{{\sloppy
\parbox{0.9\textwidth}{
Algebraic Geometry Section, Steklov Mathematical Institute of Russian Academy of Sciences,\\
8 Gubkin str., Moscow 119991 Russia}
\bigskip}}
\email{akuznet@mi-ras.ru}
\date{}
\thanks{This work is supported by the Russian Science Foundation under grant 19-11-00164,
\url{https://rscf.ru/project/19-11-00164/}.}

\begin{document}

\begin{abstract}
We introduce the notion of hyperbolic equivalence for quadric bundles and quadratic forms on vector bundles 
and show that hyperbolic equivalent quadric bundles share many important properties:
they have the same Brauer data;
moreover, if they have the same dimension over the base, they are birational over the base 
and have equal classes in the Grothendieck ring of varieties.

Furthermore, when the base is a projective space 
we show that two quadratic forms are hyperbolic equivalent if and only if 
their cokernel sheaves are isomorphic up to twist, their fibers over a fixed point of the base are Witt equivalent, 
and, in some cases, certain quadratic forms on intermediate cohomology groups of the underlying vector bundles are Witt equivalent.
For this we show that any quadratic form over~$\P^n$ is hyperbolic equivalent to a quadratic form 
whose underlying vector bundle has many cohomology vanishings; 
this class of bundles, called VLC bundles in the paper, is interesting by itself.
\end{abstract}

\maketitle

\section{Introduction}

Let $Q \to X$ be a \emph{quadric bundle}, that is a proper morphism 
which can be presented as a composition~\mbox{$Q \hookrightarrow \P_X(\cE) \to X$},
where $\P_X(\cE) \to X$ is the projectivization of a vector bundle~$\cE$ and $Q \hookrightarrow \P_X(\cE)$ 
is a divisorial embedding of relative degree~2 over~$X$.
A quadric bundle is determined by a \emph{quadratic form}~\mbox{$q \colon \Sym^2\cE \to \cL^\vee$} 
with values in a line bundle~$\cL^\vee$,
or, equivalently, by a \emph{self-dual morphism}
\begin{equation}
\label{eq:morphism-general}
q \colon \cE \otimes \cL \to \cE^\vee.
\end{equation}
Conversely, the quadratic form~$q$ is determined by~$Q$ up to rescaling and a twist transformation
\begin{equation*}
\cE \mapsto \cE \otimes \cM,
\qquad 
\cL \mapsto \cL \otimes \cM^{-2}
\end{equation*}
where~$\cM$ is a line bundle on~$X$.

Furthermore, with a quadric bundle one associates the coherent sheaf
\begin{equation}
\label{def:cc-q}
\cC(Q) = \cC(q) := \Coker \Big( q \colon \cE \otimes \cL \to \cE^\vee \Big)
\end{equation}
on~$X$, which we call its \emph{cokernel sheaf}
and which is determined by~$Q$ up to a line bundle twist.
We will usually assume that~$X$ is integral and the general fiber of~$Q \to X$ is non-degenerate,
or equivalently, that~$q$ is an isomorphism at the general point of~$X$, so that $\Ker(q) = 0$ and~$\cC(q)$ is a torsion sheaf on~$X$.
Then the sheaf~$\cC(q)$
is endowed with a ``shifted'' self-dual isomorphism 
\begin{equation}
\label{eq:cokernel-form}
\bar{q} \colon \cC(q) \xrightarrow{\ \simeq\ } \cExt^1(\cC(q), \cL) = \cC(q)^\vee \otimes \cL[1],
\end{equation}
where~$\cC(q)^\vee$ is the \emph{derived} dual of~$\cC(q)$ and~$[1]$ is the shift in the derived category
(see~\S\ref{subsec:sym-sheaves} for a discussion of sheaves enjoying this property).

The main question addressed in this paper is: what properties of quadric bundles are determined by their cokernel sheaves
(we restate this question below in a more precise form as Question~\ref{main-question})?
A prioiri it is hard to expect that the cokernel sheaf determines a lot; 
for instance because it is supported only on the discriminant divisor of~$Q/X$.
However, the main result of this paper, is that in the case where~$X$ is a projective space 
and some mild numerical conditions discussed below are satisfied,
the cokernel sheaf determines the quadric bundle up to a natural equivalence relation, which we call \emph{hyperbolic equivalence},
and which itself preserves the most important geometric properties of quadric bundles.

Hyperbolic equivalence is generated by operations of \emph{hyperbolic reduction} and \emph{hyperbolic extension}.
The simplest instance of a hyperbolic reduction (over the trivial base) is the operation 
that takes a quadric~\mbox{$Q \subset \P^r$} and a smooth point $p \in Q$ and associates to it 
the fundamental locus of the linear projection $\Bl_{p}(Q) \to \P^{r-1}$, 
which is a quadric $Q_- \subset \P^{r-2} \subset \P^{r-1}$ of dimension by~2 less than~$Q$.
From the above geometric perspective it is clear that the hyperbolic reduction procedure is invertible:
the inverse operation, which we call a hyperbolic extension, 
takes a quadric $Q \subset \P^{r}$ and a hyperplane embedding~\mbox{$\P^{r} \hookrightarrow \P^{r+1}$} 
and associates to it the quadric $Q_+ \subset \P^{r+2}$ obtained by blowing up $Q \subset \P^{r+1}$ 
and then contracting the strict transform of $\P^{r} \subset \P^{r+1}$.

The operations of hyperbolic reduction and extension can be defined in relative setting, 
i.e., for quadric bundles $Q \subset \P_X(\cE) \to X$ over any base~$X$,
and, moreover, can be lifted to operations on quadratic forms.
For the reduction a smooth point is replaced by a section~\mbox{$X \to Q$} that does not pass through singular points of fibers,
or more generally, by a \emph{regular isotropic subbundle}~$\cF\ \subset \cE$,
and for the extension a hyperplane embedding is replaced 
by an embedding $\cE \hookrightarrow \cE'$ of vector bundles of arbitrary corank.
We define these operations for quadratic forms and quadric bundles in~\S\ref{subsec:hr-he} and~\S\ref{subsec:he} 
and say that quadratic forms~$(\cE,q)$ and~$(\cE',q')$ or quadric bundles~$Q$ and~$Q'$ over~$X$ 
are \emph{hyperbolic equivalent} if they can be connected by a chain of hyperbolic reductions and extensions.

While the construction of hyperbolic reduction is quite straightforward in the general case,
this is far from true for hyperbolic extension.
In fact, when we start with an extension~\mbox{$0 \to \cE \to \cE' \to \cG \to 0$} of vector bundles,
where the bundle~$\cG$ has rank \emph{greater than~$1$},
this operation does not have a simple geometric description (as in the rank~1 case);
moreover, the set~$\HE(\cE,q,\eps)$ of all hyperbolic extensions of~$(\cE,q)$ 
with respect to an extension class~$\eps \in \Ext^1(\cG,\cE)$
is empty unless a certain obstuction class~$q(\eps,\eps) \in \Ext^2(\bw2\cG,\cL^\vee)$ vanishes,
and when the obstruction is zero, $\HE(\cE,q,\eps)$ is a principal homogeneous space 
under the natural action of the group~$\Ext^1(\bw2\cG,\cL^\vee)$.
This can be seen even in the simplest case where the extension is split, i.e., $\cE' = \cE \oplus \cG$ ---
in this case the obstruction vanishes and the corresponding hyperbolic extensions 
have the form~$\cE_+ = \cE \oplus \cG_+$, where~$\cG_+$ is an arbitrary extension of~$\cG$ by~$\cL^\vee \otimes \cG^\vee$
with the class in the subspace~$\Ext^1(\bw2\cG,\cL^\vee) \subset \Ext^1(\cG, \cL^\vee \otimes \cG^\vee)$.
For a discussion of a slightly more complicated situation see Remark~\ref{rem:q-eps-0}.
In general the situation is similar but even more complicated.
The construction of hyperbolic extension explained in~\S\ref{subsec:he} (see Theorem~\ref{thm:he})
is the first main result of this paper.

As we mentioned above hyperbolic equivalence does not change the basic invariants of a quadratic form.
In~\S\ref{subsection:he} we prove the following 
(for the definition of the Clifford algebra~$\Cl_0(\cE,q)$ we refer to~\cite{K08}).

\begin{proposition}
\label{prop:he-invariants}
Let~$(\cE,q)$ and~$(\cE',q')$ be hyperbolic equivalent generically non-degenerate quadratic forms over~$X$
and let~$Q \to X$ and~$Q' \to X$ be the corresponding hyperbolic equivalent quadric bundles,
where~$X$ is a scheme over a field~$\kk$ of characteristic not equal to~$2$.
Then 
\begin{enumerate}\setcounter{enumi}{-1}
\item 
\label{item:dim-mod-2}
One has $\dim(Q/X) \equiv \dim(Q'/X) \bmod 2$.
\item 
\label{item:he-cokernel}
The cokernel sheaves~$\cC(Q) = \cC(q)$ and~$\cC(Q') = \cC(q')$ are isomorphic up to twist by a line bundle on~$X$
and their isomorphism is compatible with the shifted quadratic forms~\eqref{eq:cokernel-form}.
\item 
\label{item:he-discriminant}
The discriminant divisors $\operatorname{Disc}_{Q/X} \subset X$ and $\operatorname{Disc}_{Q'/X} \subset X$ of~$Q$ and $Q'$ coincide.
\item 
\label{item:he-clifford}
The even parts of Clifford algebras $\Cl_0(\cE, q)$ and $\Cl_0(\cE', q')$ on~$X$ are Morita equivalent.
\item 
\label{item:l-equivalence}
If~$\dim(Q/X) = \dim(Q'/X)$ then~$[Q] = [Q']$ in the Grothendieck ring of varieties~$\Grr$.
\item 
\label{item:he-witt}
If the base scheme~$X$ is integral the classes of general fibers~$q_{\rK(X)}$ and~$q'_{\rK(X)}$ 
in the Witt group of quadratic forms over the field of rational functions $\rK(X)$ on~$X$ are equal.
If, moreover, $\dim(Q/X) = \dim(Q'/X)$ then $Q_{\rK(X)} \cong Q'_{\rK(X)}$ and~$Q$ is birational to~$Q'$ over~$X$.
\end{enumerate}
\end{proposition}

In the rest of the paper we explore if the converse of Proposition~\ref{prop:he-invariants}\eqref{item:he-cokernel} is true.
More precisely, we discuss the following

\begin{question}
\label{main-question}
Does the cokernel sheaf endowed with its shifted quadratic form~\eqref{eq:cokernel-form} 
determine the hyperbolic equivalence class of quadratic forms?
\end{question}

At this point it makes sense to explain the relation of hyperbolic equivalence to Witt groups.
Recall that the Witt group~$\rW(K)$ of a field~$K$ is defined as 
the quotient of the monoid of isomorphism classes of \emph{non-degenerate} quadratic forms~$(V,q)$, 
where~$V$ is a $K$-vector space and~$q \in \Sym^2V^\vee$ is a non-degenerate quadratic form,
by the class of the hyperbolic plane~$\left(K^{\oplus 2}, \left(\begin{smallmatrix} 0 & 1 \\ 1 & 0 \end{smallmatrix}\right)\right)$.
Similarly, the Witt group~$\rW(X)$ of a scheme~$X$ is defined~\cite{Knebusch} as 
the quotient of the monoid of isomorphism classes of \emph{unimodular}, i.e., everywhere non-degenerate quadratic forms~$(\cE,q)$,
where~$\cE$ is a vector bundle on~$X$ and~\mbox{$q \in \Hom(\cO_X,\Sym^2\cE^\vee)$} is everywhere non-degenerate,
by the classes of \emph{metabolic forms}~$\left(\cF \oplus \cF^\vee, \left(\begin{smallmatrix} 0 & 1 \\ 1 & q' \end{smallmatrix}\right)\right)$.
As it is explained in the survey~\cite{Balmer}, modifying the standard duality operation on the category of vector bundles on~$X$
one can define the Witt group~$\rW(X,\cL)$ 
that classifies classes of line bundle valued non-degenerate quadratic forms $q \colon \Sym^2\cE \to \cL^\vee$.
Moreover, a trick described in~\cite{BFF} allows one to define the Witt group~$\rWnu(X,\cL)$ of \emph{non-unimodular quadratic forms}
(i.e., forms that are allowed to be degenerate)
as the usual Witt group of the category of morphisms of vector bundle.
Thus, quadratic forms~\eqref{eq:morphism-general} define elements of~$\rWnu(X,\cL)$.

It is well known that hyperbolic reduction (as defined above) does not change the class of a quadratic form~$(\cE,q)$
in the Witt group~$\rWnu(X,\cL)$ (see, e.g., \cite[\S1.1.5]{Balmer}, where it is called \emph{sublagrangian reduction}).
On the other hand, Witt equivalence may change the cokernel sheaf of a quadratic form,
e.g., for any morphism $\varphi \colon \cE_1 \to \cE_2$ of vector bundles the class 
of the quadratic form 
$\left(\cE_1 \oplus \cE_2^\vee, \left(\begin{smallmatrix} 0 & \varphi \\ \varphi^\vee & 0 \end{smallmatrix}\right)\right)$
in the Witt group~$\rWnu(X,\cO_X)$ is zero, 
but the corresponding cokernel sheaf $\cC \cong \Coker(\varphi) \oplus \Coker(\varphi^\vee)$ is non-trivial unless~$\varphi$ is an isomorphism.
Therefore, Question~\ref{main-question} does not reduce to a question about Witt groups.

To answer Question~\ref{main-question} (in the case~$X = \P^n$) we define the following 
two basic hyperbolic equivalence invariants of quadratic forms
that take values in the non-unimodular Witt group~$\rWnu(\kk)$ of the base field~$\kk$.
Here and everywhere below we assume that the characteristic of~$\kk$ is not equal to~2.

To define the first invariant, assume $X$ is a $\kk$-scheme with a $\kk$-point $x \in X(\kk)$.
We fix a trivialization of~$\cL_x$ and define
\begin{equation}
\label{eq:wx}
\rw_x(\cE,q) := [(\cE_x,q_x)] \in \rWnu(\kk)
\end{equation}
to be the class of the quadratic form~$q_x$ obtained 
as the composition~$\Sym^2\cE_x \xrightarrow{\ q\ } \cL_x^\vee \cong \kk$,
where the second arrow is given by the trivialization of~$\cL_x$
(we could also define~$\rw_x(\cE,q)$ to be the class of the quotient of~$(\cE_x,q_x)$ by the kernel; 
then it would take values in~$\rW(\kk)$).
The class~$\rw_x(\cE,q)$ depends on the choice of trivialization, but this is not a problem for our purposes.
If the scheme~$X$ has no~$\kk$-points, we could take~$x$ to be a~$\kk'$-point for any field extension~$\kk'/\kk$ 
and define~$\rw_x(\cE,q) \in \rWnu(\kk')$ in the same way. 

For the second invariant, assume $X$ is smooth, connected and proper $\kk$-scheme, $n = \dim(X)$ is even, 
and~$\cL \otimes \omega_X \cong \cM^2$ for a line bundle~$\cM$ on~$X$, where~$\omega_X$ is the canonical line bundle of~$X$.
Then we define the bilinear form
\begin{equation}
\label{def:q-cl-prime}
H^{n/2}(X, \cE \otimes \cM) \otimes H^{n/2}(X, \cE \otimes \cM) 
\xrightarrow{\ q\ }
H^{n}(X, \cL^\vee \otimes \cM \otimes \cM) \cong H^n(X,\omega_X) = \kk
\end{equation}
on the cohomology group $H^{n/2}(X, \cE \otimes \cM)$ which we denote~$H^{n/2}(q)$ or~$H^{n/2}(Q)$.
This form, of course, depends on the choice of the line bundle~$\cM$ (if $\Pic(X)$ has 2-torsion, there may be several choices),
but we suppress this in the notation.
The bilinear form~$H^{n/2}(q)$ is symmetric if $n/2$ is even (and skew-symmetric otherwise) and possibly degenerate.
Anyway, if~$n$ is divisible by~4, we denote its class in the non-unimodular Witt group by
\begin{equation}
\label{eq:hw}
\hw(\cE,q) := [H^{n/2}(X, \cE \otimes \cM), H^{n/2}(q)] \in \rWnu(\kk)
\end{equation} 
(again, we could define~$\hw(\cE,q)$ to be the class of the quotient of~$H^{n/2}(q)$ by its kernel; 
then it would take values in~$\rW(\kk)$).
As before, the class~$\hw(\cE,q)$ depends on the choice of isomorphism~$\cL \otimes \omega_X \cong \cM^2$, 
but this is still not a problem.

Note that when~$\kk$ is alegbraically closed, $\rW(\kk) \cong \ZZ/2$
and so, if the corresponding forms are non-degenerate, the invariants $\rw_x(\cE,q)$ and~$\hw(\cE,q)$ take values in~$\ZZ/2$,
and do not depend on extra choices.
In this case $\rw_x(\cE,q)$ is just the parity of the rank of~$\cE$ and~$\hw(\cE,q)$ is the parity of the rank of~$H^{n/2}(q)$.

The second main result of this paper is the affirmative answer to Question~\ref{main-question} in the case~$X = \P^n$.
Recall that $\Pic(\P^n) = \ZZ$, hence any line bundle~$\cL$ has the form~$\cL = \cO(-m)$ for some~$m \in \ZZ$.
We need to define the following two ``standard'' types of unimodular quadratic forms with values in~$\cO(m)$:
\begin{align}
\label{eq:trivial-linear}
(\cE,q) &\cong 
\moplus_{\hphantom{+n}i \equiv m \bmod 2\hphantom{+1}} W^i \otimes \cO((m + i)/2) ,
&&\text{or} 
\\
\label{eq:trivial-omega}
(\cE,q) &\cong 
\moplus_{i \equiv m + n + 1 \bmod 2} W^i \otimes \Omega^{n/2}((m + n + 1 + i)/2),
&&\text{if $n$ is even,} 
\end{align}
where~$q$ is the sum of tensor products of the natural pairings 
\begin{align*}
\cO((m - i)/2) \otimes \cO((m + i)/2) &\xrightarrow{\ \hphantom{\wedge}\ } \cO(m),\\
\Omega^{n/2}((m + n + 1 - i)/2) \otimes \Omega^{n/2}((m + n + 1 + i)/2) &\xrightarrow{\ \wedge\ } \Omega^n(m + n + 1 ) \cong \cO(m)
\end{align*}
(the second is given by wedge product, hence it is symmetric if~$n/2$ is even and skew-symmetric if~$n/2$ is odd) 
and of non-degenerate bilinear forms~$q_{W^i} \colon W^{-i} \otimes W^{i} \to \kk$
which for~$i = 0$ are symmetric in the case~\eqref{eq:trivial-linear} and~\eqref{eq:trivial-omega} with~$n/2$ even
and skew-symmetric in the case~\eqref{eq:trivial-omega} with~$n/2$ odd.

Recall that the cokernel sheaf~$\cC(q)$ of a quadratic form~$(\cE,q)$ is endowed 
with the shifted self-duality isomorphism~$\bar{q}$, see~\eqref{eq:cokernel-form}.
In conditions~\eqref{item:wx} and~\eqref{item:hw} of the theorem 
we use the same trivialization of~$\cO(-m)_x$ and the same isomorphism~$\cO(-m) \otimes \omega_{\P^n} \cong \cM^2$
for~$(\cE_1,q_1)$ and~$(\cE_2,q_2)$.

\begin{theorem}
\label{thm:he-intro}
Let~$\kk$ be a field of characteristic not equal to~$2$ and let $X = \P^n$ be a projective space over~$\kk$.
Let $\cE_1(-m) \xrightarrow{\ q_1\ } \cE_1^\vee$ and $\cE_2(-m) \xrightarrow{\ q_2\ } \cE_2^\vee$ 
be generically non-degenerate self-dual morphisms over~$\P^n$.
Assume there is an isomorphism of sheaves~$\cC(q_1) \cong \cC(q_2)$ compatible with the quadratic forms~$\bar{q}_1$ and~$\bar{q}_2$.
Then~$(\cE_1,q_1)$ is hyperbolic equivalent to the direct sum of~$(\cE_2,q_2)$ 
and one of the standard quadratic forms~\eqref{eq:trivial-linear} or~\eqref{eq:trivial-omega}, 
where~$W^i = 0$ for~$i \ne 0$ and~$q_{W^0}$ is anisotropic.

If, moreover, the following conditions hold true:
\begin{enumerate}
\item 
\label{item:wx}
if $m$ is even then $\rw_x(\cE_1,q_1) = \rw_x(\cE_2,q_2) \in \rWnu(\kk)$ for some $\kk$-point~$x \in \P^n$;
\item 
\label{item:hw}
if $m$ is odd and $n$ is divisible by~$4$ then $\hw(\cE_1,q_1) = \hw(\cE_2,q_2) \in \rWnu(\kk)$;
\end{enumerate}
then~$(\cE_1,q_1)$ is hyperbolic equivalent to~$(\cE_2,q_2)$.
\end{theorem}

If~$\kk$ is algebraically closed and~$x$ is chosen away from the support of~$\cC(q_i)$, condition~\eqref{item:wx} in the theorem
just amounts to~$\cE_1$ and~$\cE_2$ having ranks of the same parity.
Similarly, condition~\eqref{item:hw} amounts to the forms~$\hw(\cE_i,q_i)$ having ranks of the same parity.

Note also that adding a standard summand of type~\eqref{eq:trivial-linear} with~$W^i = 0$ for~$i \ne 0$ and~$\dim(W^0) = 1$
corresponds geometrically to replacing a quadric bundle~\mbox{$Q \subset \P_{\P^n}(\cE) \to \P^n$}
by the quadric bundle $\tilde{Q} \to \P^n$, where $\tilde{Q} \to \P_{\P^n}(\cE)$ is the double covering branched along~$Q$
(note that this operation changes the parity of the rank of~$\cE$).
The geometric meaning of adding a trivial summand of type~\eqref{eq:trivial-omega} is not so obvious.

\begin{remark}
The condition of compatibility of an isomorphism~$\cC(q_1) \cong \cC(q_2)$ with the shifted quadratic forms~$\bar{q}_1$ and~$\bar{q}_2$ 
may seem subtle, but in many applications it is easy to verify.
For instance, if the sheaves~$\cC(q_i)$ are \emph{simple}, i.e., $\End(\cC(q_i)) \cong \kk$,
then a non-degenerate shifted quadratic form on~$\cC(q_i)$ is unique up to scalar,
so if~$\kk$ is quadratically closed then any isomorphism of~$\cC(q_i)$ after appropriate rescaling
is compatible with the shifted quadratic forms.
\end{remark}

To prove Theorem~\ref{thm:he-intro} we develop in~\S\ref{sec:hacm} the theory of what we call \emph{VHC morphisms} 
(here VHC stands for \emph{vanishing of half cohomology}).
These are morphisms of vector bundles $\cE_\rL \to \cE_\rU$ on~$\P^n$ such that
\begin{align*}
H^p(\P^n,\cE_\rL(t)) &= 0 &&\text{for $1 \le p \le \lfloor n/2 \rfloor$} && \text{and all $t \in \ZZ$, and}\\
H^p(\P^n,\cE_\rU(t)) &= 0 &&\text{for $\lceil n/2 \rceil \le p \le n - 1$} && \text{and all $t \in \ZZ$}
\end{align*}
(we say then that~$\cE_\la$ is VLC as its \emph{lower} intermediate cohomology vanishes, 
and~$\cE_\ua$ is VUC as its \emph{upper} intermediate cohomology vanishes).
The main results of this section are Theorem~\ref{prop:hacm-uniqueness},
in which we prove the uniqueness (under appropriate assumptions) of VHC resolutions,
and Corollary~\ref{cor:hacm-resolution}, proving the existence of VHC resolutions for any sheaf of projective dimension~1.

In~\S\ref{sec:hacm-pn} we apply this technique to the case of resolutions of \emph{symmetric sheaves} (see Definition~\ref{def:sym-sheaf}).
Any cokernel sheaf~$\cC(q)$ is symmetric, and conversely, if $X = \P^n$ then under a mild technical assumption any symmetric sheaf
is isomorphic to $\cC(q)$ for some self-dual morphism $q \colon \cE(-m) \to \cE^\vee$ 
(see~\cite{CC97} or Theorem~\ref{thm:hacm-existence} and Remark~\ref{remark:epw} in~\S\ref{sec:hacm-pn}).

Our main technical result here is the Modification Theorem (Theorem~\ref{thm:he-hacm})
in which we show that any self-dual morphism over~$\P^n$ is hyperbolic equivalent 
to the sum of a self-dual VHC morphism and a standard unimodular self-dual morphism 
of type~\eqref{eq:trivial-linear} or~\eqref{eq:trivial-omega}.
This implies Theorem~\ref{thm:he-intro}, see~\S\ref{subsec:proofs} for the proof.

Combining Theorem~\ref{thm:he-intro} with Proposition~\ref{prop:he-invariants} we obtain the following corollary,
which for simplicity we state over an algebraically closed ground field.

\begin{corollary}
\label{corollary:invariants}
Let~$\kk$ be an algebraically closed field of characteristic not equal to~$2$.
Let~$Q \to \P^n$ and~$Q' \to \P^n$ be generically smooth quadric bundles such that 
there is an isomorphism of the cokernel sheaves~$\cC(Q) \cong \cC(Q')$ compatible with their shifted quadratic forms.
If $n$ is divisible by~$4$ and~$m$ is odd assume also that~\mbox{$\rk(H^{n/2}(Q)) \equiv \rk(H^{n/2}(Q')) \bmod 2$},
where the quadratic forms~$H^{n/2}(Q)$ and~$H^{n/2}(Q')$ are defined by~\eqref{def:q-cl-prime}.
Then
\begin{enumerate}
\item 
\label{item:cor-brauer-even}
If $\dim(Q/\P^n)$ and $\dim(Q'/\P^n)$ are even then the corresponding discriminant double covers \mbox{$S \to \P^n$} and~\mbox{$S' \to \P^n$} 
are isomorphic over~$\P^n$,
and the Brauer classes \mbox{$\upbeta_S \in \Br(S_{\le 1})$} and~\mbox{$\upbeta'_{S} \in \Br(S'_{\le 1})$} 
on the corank~$\le 1$ loci inside~$S$ and~$S'$ are equal.
\item 
\label{item:cor-brauer-odd}
If $\dim(Q/\P^n)$ and $\dim(Q'/\P^n)$ are odd then the corresponding discriminant root stacks \mbox{$S \to \P^n$} and~\mbox{$S' \to \P^n$} 
are isomorphic over~$\P^n$,
and the Brauer classes \mbox{$\upbeta_S \in \Br(S_{\le 1})$} and~\mbox{$\upbeta'_{S} \in \Br(S'_{\le 1})$} 
on the corank~$\le 1$ loci inside~$S$ and~$S'$ are equal.
\item 
\label{item:cor-l-equivalence}
If~$\dim(Q/\P^n) = \dim(Q'/\P^n)$ then~$[Q] = [Q']$ in the Grothendieck ring of varieties~$\Grr$.
\item 
\label{item:cor-birational-equivalence}
If~$\dim(Q/\P^n) = \dim(Q'/\P^n)$ then there is a birational isomorphism $Q \sim Q'$ over~$\P^n$.
\end{enumerate}
\end{corollary}

To finish the Introduction it should be said that this paper was inspired by the recent paper~\cite{BKK},
where similar questions were discussed.
In particular, assertions~\eqref{item:cor-brauer-even}
and~\eqref{item:cor-birational-equivalence} of Corollary~\ref{corollary:invariants} in case~$n = 2$ have been proved there.
We refer to~\cite{BKK} for various geometric applications of these results.

On the other hand, we want to stress that the approach of the present paper is completely different: the results of~\cite{BKK} are based 
on an explicit computation of the Brauer class of a quadric bundle using the technique developed in~\cite{IOOV}.
It is unclear whether these methods can be effectively generalized to higher dimensions.

It also makes sense to mention that the technique of hyperbolic extensions and VHC resolutions 
developed in this paper can be used for other questions
related to quadric bundles over arbitrary schemes and vector bundles on projective spaces.

\smallskip 

{\bf Convention:}
Throughout the paper we work over an arbitrary field~$\kk$ of characteristic not equal to~2.

\smallskip 

{\bf Acknowledgements:} 
This paper owns its very existence to~\cite{BKK}, so I am very grateful to its authors for inspiration and useful discussions.
I would also like to thank Alexey Ananyevskiy for a suggestion that allowed me to improve significantly 
the results of Proposition~\ref{prop:he-invariants}\eqref{item:l-equivalence} 
and Corollary~\ref{corollary:invariants}\eqref{item:cor-l-equivalence}
and the anonymous referee for many useful comments about the first version of the paper.

\section{Quadric bundles and hyperbolic equivalence}

Recall from the Introduction the definition of a quadric bundle, 
of its associated quadratic form and self-dual morphism~\eqref{eq:morphism-general} 
(which we assume to be generically non-degenerate),
of the cokernel sheaf~\eqref{def:cc-q} and of its shifted self-duality~\eqref{eq:cokernel-form}.
Conversely, we denote by 
\begin{equation*}
Q(\cE,q) \subset \P_X(\cE)
\end{equation*}
the quadric bundle associated with a quadratic form~$(\cE,q)$ or a morphism~\eqref{eq:morphism-general}.

\subsection{Hyperbolic reduction}
\label{subsec:hr-he}

We start with the notion of hyperbolic reduction, which is well known, see~\cite{ABB,KS18}.
For the reader's convenience we remind the definition in a slightly different form.

Let~\eqref{eq:morphism-general} be a self-dual morphism of vector bundles on a scheme~$X$.
We will say that a vector subbundle~\mbox{$\phi \colon \cF \hookrightarrow \cE$} is {\sf regular isotropic}, if the composition
\begin{equation*}
\cE \xrightarrow{\ q\ } \cE^\vee \otimes \cL^\vee \xrightarrow{\ \phi^\vee\ } \cF^\vee \otimes \cL^\vee
\end{equation*}
is surjective and vanishes on the subbundle~$\cF \subset \cE$, 
i.e., $\cF$ is contained in the subbundle
\begin{equation}
\label{def:cf-perp}
\cF^\perp := \Ker ( \cE \twoheadrightarrow \cF^\vee \otimes \cL^\vee ) \subset \cE.
\end{equation}
If $\cF$ is regular isotropic, the restriction of~$q$ to~$\cF^\perp$ contains~$\cF$ in the kernel, 
hence induces a quadratic form on~$\cF^\perp/\cF$.
We summarize these observations in the following

\begin{lemma}
\label{lemma:h-reduction}
Let~\eqref{eq:morphism-general} be a self-dual morphism of vector bundles on a scheme~$X$.
Let~$\phi \colon \cF \hookrightarrow \cE$ be a regular isotropic subbundle. 
Denote 
\begin{equation*}
\cE_- := \cF^\perp / \cF.
\end{equation*}
The restriction of~$q$ to~$\cF^\perp$ induces
a self-dual morphism $q_- \colon \cE_- \otimes \cL \to \cE_-^\vee$ 
such that there is an isomorphism~$\cC(q_-) \cong \cC(q)$ of the cokernel bundles 
compatible with their shifted self-dualities~$\bar{q}$ and~$\bar{q}_-$.
\end{lemma}

\begin{proof}
The result follows from the argument of~\cite[Lemma~2.4]{KS18}. 
Indeed, it is explained in \emph{loc.\ cit.}\/ that the cokernel sheaf~$\cC(q_-)$ is isomorphic to the cohomology of the bicomplex
(cf.~\cite[(2)]{KS18})
\begin{equation}
\label{diagram:cc-qminus}
\vcenter{\xymatrix@C=5em{
\cF \otimes \cL \ar[r]^\phi \ar[d]_\id & 
\cE \otimes \cL \ar[r]^{\phi^\vee \circ q} \ar[d]^q &
\cF^\vee \ar[d]^\id
\\
\cF \otimes \cL \ar[r]^{q \circ \phi} & 
\cE^\vee \ar[r]^{\phi^\vee} &
\cF^\vee.
}}
\end{equation}
Its left and right columns are acyclic, while the middle one coincides with~\eqref{eq:morphism-general}, 
hence~\mbox{$\cC(q_-) \cong \cC(q)$}. 
Furthermore, using the self-duality of~$q$, we see that the dual of~\eqref{diagram:cc-qminus} twisted by~$\cL$ is isomorphic to~\eqref{diagram:cc-qminus},
and moreover, this isomorphism is compatible with the isomorphism of the dual of~\eqref{eq:morphism-general} 
twisted by~$\cL$ with~\eqref{eq:morphism-general}.
This means that the isomorphism of the cokernel sheaves~\mbox{$\cC(q_-) \cong \cC(q)$} is compatible with their shifted self-dualities.
\end{proof}

The operation 
\begin{equation*}
(\cE,q) \mapsto (\cE_-,q_-)
\qquad\text{or}\qquad 
Q(\cE,q) \mapsto Q(\cE_-,q_-)
\end{equation*}
defined in Lemma~\ref{lemma:h-reduction} is called {\sf hyperbolic reduction} 
of a quadratic form (resp.\ of a quadric bundle) with respect to the subbundle~$\cF$.
As explained in~\cite[Proposition~2.5]{KS18}, this operation can be interpreted geometrically 
in terms of the linear projection of $Q \subset \P_X(\cE)$
from the linear subbundle~\mbox{$\P_X(\cF) \subset Q \subset \P_X(\cE)$}.

The next simple lemma motivates the terminology.

\begin{lemma}
\label{lemma:witt}
Assume $X$ is integral and~$\rK(X)$ is the field of rational functions on~$X$.
If $Q/X$ is a generically non-degenerate quadric bundle and $Q_-/X$ is its hyperbolic reduction, 
then the quadratic forms~$q_{\rK(X)}$ and~$({q_-})_{\rK(X)}$ 
corresponding to their general fibers are equal in the Witt group~$\rW(\rK(X))$ of~$\rK(X)$.
\end{lemma}
\begin{proof}
Hyperbolic reduction commutes with base change, so the question reduces to the case where the base is the spectrum of~$\rK(X)$,
i.e., to the case of hyperbolic reduction of a quadric $Q_{\rK(X)} \subset \P(E_{\rK(X)})$ 
with respect to a linear subspace $F_{\rK(X)} \subset E_{\rK(X)}$.
In this case~$q_-$ is the induced quadratic form on~$F_{\rK(X)}^\perp/F_{\rK(X)}$ 
(the orthogonal is taken with respect to the quadratic form~$q$).
It is easy to see that the quadratic form~$q$ is isomorphic to the orthogonal sum $q_- \perp q_0$ of~$q_-$ 
with the hyperbolic form $q_0 = \left(\begin{smallmatrix} 0 & 1_{\dim(F)} \\ 1_{\dim(F)} & 0 \end{smallmatrix}\right)$,
hence~\mbox{$q = q_-$} in the Witt group~$\rW(\rK(X))$.
\end{proof}

The following obvious lemma shows that hyperbolic reduction is transitive.

\begin{lemma}
\label{lemma:hr-transitive}
Let~$(\cE_-,q_-)$ be the hyperbolic reduction of~$(\cE,q)$ 
with respect to a regular isotropic subbundle~\mbox{$\cF \hookrightarrow \cE$}
and let~$(\cE_{--},q_{--})$ be the hyperbolic reduction of~$(\cE_-,q_-)$ 
with respect to a regular isotropic subbundle~\mbox{$\cF_- \hookrightarrow \cE_-$}.
Then~$(\cE_{--},q_{--})$ is a hyperbolic reduction of~$(\cE,q)$.
\end{lemma}

\begin{proof}
Let~$\tilde\cF \subset \cF^\perp$ be the preimage of~$\cF_- \subset \cE_-$ 
under the map~$\cF^\perp \twoheadrightarrow \cF^\perp/\cF = \cE_-$,
so that there is an exact sequence~$0 \to \cF \to \tilde\cF \to \cF_- \to 0$ and an embedding~$\tilde\cF \hookrightarrow \cE$.
Then~$\tilde\cF$ is regular isotropic and the hyperbolic reduction of~$(\cE,q)$ 
with respect to~$\tilde\cF$ is isomorphic to~$(\cE_{--},q_{--})$.
\end{proof}

In the next subsection we will describe a construction inverse to hyperbolic reduction,
and in the rest of this subsection we introduce the input data for that construction.

Assume $\cF \subset \cE$ is a regular isotropic subbundle with respect to a quadratic form~$q$ 
and let~$(\cE_-,q_-)$ be the hyperbolic reduction of~$(\cE,q)$ with respect to~$\cF$.
Consider the length~3 filtration
\begin{equation}
\label{eq:ce-filtration}
0 \hookrightarrow \cF \hookrightarrow \cF^\perp \hookrightarrow \cE.
\end{equation}
Its associated graded is $\gr^\bullet(\cE) = \cF \oplus \cE_- \oplus (\cF^\vee \otimes \cL^\vee)$.
In particular, we have two exact sequences
\begin{align}
\label{eq:cf-cfperp-cem}
0 \to \cF \to \cF^\perp \to \cE_- \to 0,\\
\label{eq:cem-cecf-cfvee}
0 \to \cE_- \to \cE/\cF \to \cF^\vee \otimes \cL^\vee \to 0.
\end{align}
The next lemma describes a relation between their extension classes.

\begin{lemma}
\label{lemma:eps-q-eps}
Let $\eps \in \Ext^1(\cF^\vee \otimes \cL^\vee, \cE_-)$ be the extension class of~\eqref{eq:cem-cecf-cfvee}.
Then the extension class of~\eqref{eq:cf-cfperp-cem} is equal to~$q_-(\eps)$, 
the Yoneda product of~$\eps$ with the map $q_- \colon \cE_- \to \cE_-^\vee \otimes \cL^\vee$, so that
\begin{equation*}
q_-(\eps) \in \Ext^1(\cF^\vee \otimes \cL^\vee, \cE_-^\vee \otimes \cL^\vee) \cong \Ext^1(\cE_-,\cF).
\end{equation*}
Moreover, the Yoneda product $q_-(\eps,\eps) := q_-(\eps) \circ \eps \in \Ext^2(\cF^\vee \otimes \cL^\vee, \cF)$ vanishes.
\end{lemma}
\begin{proof}
Tensoring diagram~\eqref{diagram:cc-qminus} by~$\cL^\vee$ and taking quotients by~$\cF$ 
we obtain a morphism of exact sequences
\begin{equation*}
\xymatrix{
0 \ar[r] &
\cE_- \ar[r] \ar[d]_{q_-} &
\cE/\cF \ar[r] \ar[d]_q &
\cF^\vee \otimes \cL^\vee \ar[r] \ar@{=}[d] &
0
\\
0 \ar[r] &
\cE_-^\vee \otimes \cL^\vee \ar[r] &
(\cE^\vee \otimes \cL^\vee)/\cF \ar[r] &
\cF^\vee \otimes \cL^\vee \ar[r] &
0
}
\end{equation*}
This is a pushout diagram and the extension class of the top row is~$\eps$, hence the extension class of the bottom row is~$q_-(\eps)$.
It remains to note that the bottom row is the twisted dual of~\eqref{eq:cf-cfperp-cem}.

Since the sequences~\eqref{eq:cf-cfperp-cem} and~\eqref{eq:cem-cecf-cfvee} come from a length~3 filtration of~$\cE$,
the Yoneda product of their extension classes vanishes.
\end{proof}

We axiomatize the property of the class~$\eps$ observed in Lemma~\ref{lemma:eps-q-eps} as follows 
(recall that for $s \in \ZZ$ we denote by~$[s]$ the shift by~$s$ in the derived category).

\begin{definition}
\label{def:q-eps-eps}
Let~\eqref{eq:morphism-general} be a self-dual morphism,
let~$\cG$ be a vector bundle on~$X$,
and let $\eps \in \Ext^1(\cG,\cE)$ be an extension class.
We define the classes $q(\eps) \in \Ext^1(\cE, \cG^\vee \otimes \cL^\vee)$ and
$q(\eps,\eps) \in \Ext^2(\cG,\cG^\vee \otimes \cL^\vee)$ as the Yoneda products 
\begin{equation*}
q(\eps) \colon 
\cE \xrightarrow{\ q\ } 
\cE^\vee \otimes \cL^\vee \xrightarrow{\ \eps\ } 
\cG^\vee \otimes \cL^\vee[1]
\qquad\text{and}\qquad 
q(\eps,\eps) \colon 
\cG \xrightarrow{\ \eps\ } 
\cE[1] \xrightarrow{\ q(\eps)\ } 
\cG^\vee \otimes \cL^\vee[2].
\end{equation*}
We say that~$\eps$ is {\sf $q$-isotropic} if $q(\eps,\eps) = 0$.
\end{definition}

Using this terminology we can reformulate Lemma~\ref{lemma:eps-q-eps} by saying that the class of~\eqref{eq:cem-cecf-cfvee} is $q_-$-isotropic.

\begin{remark}
It is easy to see 
that $q(\eps,\eps) \in \Ext^2(\bw2\cG,\cL^\vee) \subset \Ext^2(\cG \otimes \cG,\cL^\vee) = \Ext^2(\cG,\cG^\vee \otimes \cL^\vee)$.
Indeed, the morphism $q(\eps,\eps) = \eps \circ q \circ \eps$
is symmetric because $q$ is, hence it defines a morphism $\Sym^2(\cG[-1]) \to \cL^\vee$, 
and it remains to note that $\Sym^2(\cG[-1]) \cong \bw2\cG[-2]$.
\end{remark}

\subsection{Hyperbolic extension}
\label{subsec:he}

The following definition is central for this section.

\begin{definition}
\label{def:he}
Given a self-dual morphism~\eqref{eq:morphism-general} and a $q$-isotropic extension class $\eps \in \Ext^1(\cG,\cE)$ 
we say that $(\cE_+,q_+)$ is a {\sf hyperbolic extension of $(\cE,q)$ with respect to~$\eps$}
if there is a regular isotropic embedding~$\cL^\vee \otimes \cG^\vee \hookrightarrow \cE_+$ 
such that the hyperbolic reduction of~$(\cE_+,q_+)$ with respect to~$\cL^\vee \otimes \cG^\vee$ 
is isomorphic to~$(\cE,q)$ and the induced extension $0 \to \cE \to \cE_+/(\cL^\vee \otimes \cG^\vee) \to \cG \to 0$ has class~$\eps$.
\end{definition}

We denote by~$\HE(\cE,q,\eps)$ the set of isomorphism classes of all hyperbolic extensions of~$(\cE,q)$ 
with respect to a $q$-isotropic extension class~$\eps$.
The main goal of this section is to show that~$\HE(\cE,q,\eps)$ is non-empty;
we will moreover see that this set may be quite big.

We start, however, with a simpler case, where the set~$\HE(\cE,q,\eps)$ consists of a single element.

\begin{proposition}
\label{prop:he-rank-1}
Let~\eqref{eq:morphism-general} be a self-dual morphism of vector bundles.
If~$\cG$ is a line bundle then for any extension class~\mbox{$\eps \in \Ext^1(\cG,\cE)$} 
there exists a unique \textup(up to isomorphism\textup) hyperbolic extension of~$(\cE,q)$ with respect to~$\eps$.
\end{proposition}
\begin{proof}
We start by proving the existence of a hyperbolic extension.
The construction described below is an algebraic version of the geometric construction sketched in the Introduction.

Let 
\begin{equation}
\label{eq:general-extension}
0 \to \cE \to \cE' \to \cG \to 0
\end{equation}
be an extension of class~$\eps$ and consider its symmetric square 
$0 \to \Sym^2\cE \to \Sym^2\cE' \to \cE' \otimes \cG \to 0$,
its tensor product with~$\cG^\vee$, and its pushout along the map $\Sym^2\cE \otimes \cG^\vee \xrightarrow{\ q\ } \cL^\vee \otimes \cG^\vee$:
\begin{equation}
\label{eq:cep-pushout}
\vcenter{\xymatrix@R=3ex{
0 \ar[r] & 
\Sym^2\cE \otimes \cG^\vee \ar[r] \ar[d]_q & 
\Sym^2\cE' \otimes \cG^\vee \ar[r] \ar[d]_\phi & 
\cE' \ar[r] \ar@{=}[d] & 
0
\\
0 \ar[r] & 
\cL^\vee \otimes \cG^\vee \ar[r] & 
\cE_+ \ar[r] & 
\cE' \ar[r] & 
0,
}}
\end{equation}
defining a vector bundle~$\cE_+$ and a morphism~$\phi$.
We will show that~$\cE_+$ comes with a natural quadratic form~$q_+$ 
such that the embedding $\cL^\vee \otimes \cG^\vee \hookrightarrow \cE_+$ in the bottom row of~\eqref{eq:cep-pushout} is regular isotropic
and the corresponding hyperbolic reduction is isomorphic to~$(\cE,q)$.
For this we consider a component of the symmetric square of~$\phi$:
\begin{equation}
\label{eq:sym4-2}
\Sym^2(\phi) \colon \Sym^4\cE' \otimes (\cG^\vee)^{\otimes 2} \to \Sym^2\cE_+.
\end{equation}
We will show that its cokernel is canonically isomorphic to $\cL^\vee$, 
and we will take the cokernel morphism~$\Sym^2\cE_+ \to \cL^\vee$ as the definition of the quadratic form~$q_+$.

Indeed, considering~\eqref{eq:general-extension} as a length~2 filtration on~$\cE'$
and taking its fourth symmetric power we obtain a length~5 filtration on~$\Sym^4\cE' \otimes (\cG^\vee)^{\otimes 2}$ with factors
\begin{equation}
\label{eq:filt-e}
\Sym^4\cE \otimes (\cG^\vee)^{\otimes 2},\
\Sym^3\cE \otimes \cG^\vee,\
\Sym^2\cE,\
\cE \otimes \cG,\
\cG^{\otimes 2}.
\end{equation}
Similarly, the combination of the bottom row of~\eqref{eq:cep-pushout} with~\eqref{eq:general-extension} 
provides $\cE_+$ with a length~3 filtration which induces a length~5 filtration on~$\Sym^2\cE_+$ with factors
\begin{equation}
\label{eq:filt-e+-2}
(\cL^\vee)^{\otimes 2} \otimes (\cG^\vee)^{\otimes 2},\ 
\cE \otimes \cL^\vee \otimes \cG^\vee,\ 
\cL^\vee \oplus \Sym^2\cE,\ 
\cE \otimes \cG,\ 
\cG^{\otimes 2}.
\end{equation}
It is easy to check that the morphism~\eqref{eq:sym4-2} is compatible with the filtrations, induces isomorphisms of the last two factors,
epimorphisms on the first two factors, and the morphism
\begin{equation*}
\Sym^2\cE \xrightarrow{\ (q, \id)\ } \cL^\vee \oplus \Sym^2\cE
\end{equation*}
on the middle factors.
Therefore, the cokernel of~\eqref{eq:sym4-2} is canonically isomorphic to $\Coker(q,\id) \cong \cL^\vee$.
This induces a canonical morphism $q_+ \colon \Sym^2\cE_+ \to \cL^\vee$ 
which vanishes on the first two factors of~\eqref{eq:filt-e+-2} and restricts to the morphism~$(-\id,q)$ on the middle factor.

Since the morphism $q_+$ vanishes on the first 
factor~$(\cL^\vee)^{\otimes 2} \otimes (\cG^\vee)^{\otimes 2} \cong (\cG^\vee \otimes \cL^\vee)^{\otimes 2}$ of~\eqref{eq:filt-e+-2},
the subbundle $\cG^\vee \otimes \cL^\vee \hookrightarrow \cE_+$ is $q_+$-isotropic.
Similarly, since the morphism $q_+$ vanishes on the second factor of~\eqref{eq:filt-e+-2} 
and nowhere vanishes on the summand~$\cL^\vee \cong (\cG^\vee \otimes \cL^\vee) \otimes \cG$ of the third factor,
the subbundle~$\cG^\vee \otimes \cL^\vee \hookrightarrow \cE_+$ is regular isotropic, 
the underlying vector bundle of the hyperbolic reduction of $(\cE_+,q_+)$ is isomorphic to~$\cE$,
and the induced extension of~$\cG$ by~$\cE$ coincides with~\eqref{eq:general-extension}.
Finally, since the restriction of~$q_+$ to the summand~$\Sym^2\cE$ of the middle factor of~\eqref{eq:filt-e+-2} equals~$q$, 
the induced quadratic form on~$\cE$ is equal to~$q$.
Thus, $(\cE_+,q_+)$ is a hyperbolic extension of~$(\cE,q)$ with respect to~$\eps$.

Now we prove that the constructed hyperbolic extension is unique.
For this it is enough to show that for any hyperbolic extension~$(\cE_+,q_+)$ of~$(\cE,q)$ with respect to~$\eps$ 
there is a diagram~\eqref{eq:cep-pushout}
such that $q_+$ is the cokernel of $\Sym^2(\phi)$.

First, consider the morphism
\begin{equation*}
\phi_+ \colon \Sym^2\cE_+ \otimes \cG^\vee \to \cE_+,
\qquad 
e_1e_2 \otimes f \mapsto q_+(e_1,f) e_2 + q_+(e_2,f) e_1 - q_+(e_1,e_2) f,
\end{equation*}
where $e_i$ are sections of~$\cE_+$ and $f$ is a section of~$\cG^\vee$ that we consider as a subbundle in $\cE_+ \otimes \cL$.
The symmetric square of the exact sequence $0 \to \cL^\vee \otimes \cG^\vee \to \cE_+ \to \cE' \to 0$ tensored with~$\cG^\vee$ takes the form
\begin{equation*}
0 \to \cE_+ \otimes \cL^\vee \otimes \cG^\vee \otimes \cG^\vee \to \Sym^2\cE_+ \otimes \cG^\vee \to \Sym^2\cE' \otimes \cG^\vee \to 0,
\end{equation*}
where the first map takes $e \otimes f_1 \otimes f_2$ to $ef_1 \otimes f_2$.
The composition of this map with~$\phi_+$ acts as
\begin{equation*}
e \otimes f_1 \otimes f_2 \mapsto 
\phi_+(ef_1 \otimes f_2) = 
q_+(e,f_2)f_1 + q_+(f_1,f_2)e - q_+(e,f_1)f_2.
\end{equation*}
The second summand is zero because $\cL^\vee \otimes \cG^\vee  \subset \cE_+$ is isotropic 
and the first summand cancels with the last because the rank of~$\cG$ is~1, hence~$f_1$ and~$f_2$ are proportional.
Therefore, the map $\phi_+$ factors through a map $\phi \colon \Sym^2\cE' \otimes \cG^\vee \to \cE_+$.
Moreover, it is easy to see that this map fits into the diagram~\eqref{eq:cep-pushout}.
Finally, it is straightforward (but tedious) to check that the composition
\begin{equation*}
\Sym^4\cE' \otimes (\cG^\vee)^{\otimes 2} \xrightarrow{\ \Sym^2(\phi)\ }
\Sym^2\cE_+ \xrightarrow{\ q_+\ } \cL^\vee 
\end{equation*}
vanishes, and since~$q_+$ is a hyperbolic extension of~$q$, 
it vanishes on the first two factors of~\eqref{eq:filt-e+-2}
and induces the morphism~$\cL^\vee \oplus \Sym^2\cE \to \cL^\vee$ of the third factor which is equal to~$q$ on~$\Sym^2\cE$,
hence equal to~$(-\id,q)$ on this third factor, and thus coincides with the canonical cokernel of $\Sym^2(\phi)$.
\end{proof}

Note that the general case (where the rank of~$\cG$ is greater than~1) does not immediately reduce to a rank~1 case,
because a general vector bundle does not admit a filtration by line bundles.
Besides, even if such a filtration exists, it is hard to trace what happens with the obstructions
and to see how the nontrivial space of extensions shows up.
So, in the proof of the theorem below we use the projective bundle trick.

\begin{theorem}
\label{thm:he}
For any self-dual morphism~\eqref{eq:morphism-general} and a~$q$-isotropic extension class~\mbox{$\eps \in \Ext^1(\cG,\cE)$} 
the set~$\HE(\cE,q,\eps)$ of hyperbolic extensions of~$(\cE,q)$ with respect to~$\eps$ is non-empty and is 
a principal homogeneous variety under an action of the group~$\Ext^1(\bw2\cG,\cL^\vee)$.
\end{theorem}

The action of the group~$\Ext^1(\bw2\cG,\cL^\vee)$ on the set~$\HE(\cE,q,\eps)$ will be constructed in course of the proof.

\begin{proof}
Consider the projectivization~$\pi \colon \P_X(\cG) \to X$ and the tautological line subbundle $\cO(-1) \hookrightarrow \pi^*\cG$.
Note that the quotient bundle~$\pi^*\cG/\cO(-1)$ can be identified with~$\cT_\pi(-1)$, 
where~$\cT_\pi$ is the relative tangent bundle for the morphism~$\pi$.
We denote by~$\gamma \in \Ext^1(\cT_\pi(-1),\cO(-1))$ the extension class of the tautological sequence
\begin{equation}
\label{eq:pi-tautological}
0 \to \cO(-1) \to \pi^*\cG \to \cT_\pi(-1) \to 0.
\end{equation}
Pulling back the class~$\pi^*\eps \in \Ext^1(\pi^*\cG, \pi^*\cE)$ along the embedding $\cO(-1) \hookrightarrow \pi^*\cG$ 
we obtain the extension
\begin{equation}
\label{eq:tce-prime}
0 \to \pi^*\cE \to \tilde\cE' \to \cO(-1) \to 0
\end{equation}
on~$\P_X(\cG)$; we denote its extension class by $\teps \in \Ext^1(\cO(-1), \pi^*\cE)$.

By Proposition~\ref{prop:he-rank-1} there is a unique hyperbolic extension of $(\pi^*\cE,\pi^*q)$ with respect to $\teps$
which is given by an extension of vector bundles
\begin{equation*}
0 \to \pi^*\cL^\vee \otimes \cO(1) \to \tilde\cE_+ \to \tilde\cE' \to 0
\end{equation*}
and a quadratic form $\tilde q_+ \colon \pi^*\cL \to \Sym^2\tilde\cE_+^\vee$.
We denote the extension class of the above sequence by
\begin{equation*}
\teps' \in \Ext^1(\tcE', \pi^*\cL^\vee \otimes \cO(1)).
\end{equation*}
Note that by Lemma~\ref{lemma:eps-q-eps} the restriction of~$\teps'$ to~$\pi^*\cE \subset \tcE'$ is~$\pi^*q(\teps)$;
in particular, $\tilde\cE_+$ has a length~3 filtration with
\begin{equation*}
\gr^\bullet(\tilde\cE_+) = (\pi^*\cL^\vee \otimes \cO(1)) \oplus \pi^*\cE \oplus \cO(-1)
\end{equation*}
and the extension classes linking its factors are~$(\pi^*q)(\teps)$ and~$\teps$, respectively.

It would be natural at this point to consider a hyperbolic extension of $(\tilde\cE_+,\tilde{q}_+)$ by $\cT_\pi(-1)$
(note that the rank of~$\cT_\pi(-1)$ is less than~$\cG$) and then show that the result descends to a self-dual morphism on~$X$.
However, it turns out to be more convenient to use a simpler construction 
by ``adding'' the (twisted) dual bundle $\pi^*\cL^\vee \otimes \Omega_\pi(1)$ to the kernel space of~$\tilde{q}$ 
and then applying another version of descent.

Consider the product of extension classes
(recall that~$\gamma$ is the extension class of~\eqref{eq:pi-tautological}):
\begin{equation*}
\tcE' \xrightarrow{\ \teps'\ } \pi^*\cL^\vee \otimes \cO(1)[1] \xrightarrow{\ \gamma\ } \pi^*\cL^\vee \otimes \Omega_\pi(1)[2]
\end{equation*}
(where $\Omega_\pi = \cT_\pi^\vee$ is the relative sheaf of K\"ahler differentials). 
We claim that $\gamma \circ \teps' = 0$.
Indeed, using~\eqref{eq:tce-prime} and taking into account 
isomorphisms~$\bR\pi_*(\Omega_\pi(1)) = 0$, $\bR\pi_*(\Omega_\pi(2)) \cong \bw2\cG^\vee$, 
we obtain 
\begin{equation*}
\Ext^p(\tcE',\pi^*\cL^\vee \otimes \Omega_\pi(1)) \cong
\Ext^p(\cO(-1),\pi^*\cL^\vee \otimes \Omega_\pi(1)) \cong
\Ext^p(\bw2\cG,\cL^\vee),
\end{equation*}
for all~$p \in \ZZ$,
and note that under this isomorphism the product~$\gamma \circ \teps' \in \Ext^2(\tcE',\pi^*\cL^\vee \otimes \Omega_\pi(1))$ 
coincides with the obstruction class~$q(\eps,\eps) \in \Ext^2(\bw2\cG,\cL^\vee)$, 
and hence vanishes as~$\eps$ is assumed to be~$q$-isotropic.

Consider the tensor product of the dual sequence of~\eqref{eq:pi-tautological} with~$\pi^*\cL^\vee$:
\begin{equation*}
0 \to \pi^*\cL^\vee \otimes \Omega_\pi(1) \to \pi^*\cL^\vee \otimes \pi^*\cG^\vee \to \pi^*\cL^\vee \otimes \cO(1) \to 0,
\end{equation*}
its extension class is also~$\gamma$.
The vanishing of the product $\gamma \circ \teps'$ implies that the class $\teps'$ lifts 
to a class in $\Ext^1(\tcE', \pi^*\cL^\vee \otimes \pi^*\cG^\vee)$, or, equivalently,
that the class $\gamma \in \Ext^1(\pi^*\cL^\vee \otimes \cO(1), \pi^*\cL^\vee \otimes \Omega_\pi(1))$
lifts to a class in $\Ext^1(\tcE_+, \pi^*\cL^\vee \otimes \Omega_\pi(1))$.
Moreover, $\Hom(\pi^*\cL^\vee \otimes \cO(1), \pi^*\cL^\vee \otimes \Omega_\pi(1)) = 0$,
hence we have an exact sequence
\begin{equation*}
0 \to
\Ext^1(\tcE', \pi^*\cL^\vee \otimes \Omega_\pi(1)) \to
\Ext^1(\tcE_+, \pi^*\cL^\vee \otimes \Omega_\pi(1)) \to
\Ext^1(\pi^*\cL^\vee \otimes \cO(1), \pi^*\cL^\vee \otimes \Omega_\pi(1)) 
\end{equation*}
which shows that such a lift of~$\gamma$ is unique up to the natural free action of the group
\begin{equation*}
\Ext^1(\tcE', \pi^*\cL^\vee \otimes \Omega_\pi(1)) \cong 
\Ext^1(\cO(-1), \pi^*\cL^\vee \otimes \Omega_\pi(1)) \cong 
\Ext^1(\bw2\cG,\cL^\vee).
\end{equation*}
In other words, the set of such lifts is a principal homogeneous space under an action of~$\Ext^1(\bw2\cG,\cL^\vee)$.

The lifted classes define a vector bundle $\hcE_+$ that fits into two exact sequences
\begin{equation}
\label{eq:hce-plus}
0 \to \pi^*\cL^\vee \otimes \Omega_\pi(1) \to \hcE_+ \to \tcE_+ \to 0
\quad\text{and}\quad
0 \to \pi^*\cL^\vee \otimes \pi^*\cG^\vee \to \hcE_+ \to \tcE' \to 0.
\end{equation}
We consider the quadratic form on $\hcE_+$ defined by the following composition
\begin{equation*}
\hat{q}_+ \colon \pi^*\cL \xrightarrow{\ \tilde{q}_+\ } \Sym^2\tcE_+^\vee \hookrightarrow \Sym^2\hcE_+^\vee,
\end{equation*}
where the latter embedding is induced by the surjection $\hcE_+ \twoheadrightarrow \tcE_+$ from~\eqref{eq:hce-plus}.
Note that by construction~$\hcE_+$ has a length~4 filtration with
\begin{equation*}
\gr^\bullet(\hcE_+) = (\pi^*\cL^\vee \otimes \Omega_\pi(1)) \oplus (\pi^*\cL^\vee \otimes \cO(1)) \oplus \pi^*\cE \oplus \cO(-1)
\end{equation*}
and the extension classes linking its adjacent factors are~$\gamma$, $(\pi^*q)(\teps)$, and~$\teps$, respectively.
Furthermore, the subbundle~$\pi^*\cL^\vee \otimes \Omega_\pi(1) \subset \hcE_+$ is contained in the kernel of the quadratic form~$\hat{q}_+$.
Now we explain how to descend the quadratic form~$(\hcE_+,\hat{q}_+)$ over~$\P_X(\cG)$ 
to a quadratic form~$(\cE_+,q_+)$ over~$X$.

Consider the subbundle~$\Ker(\hcE_+ \to \cO(-1)) \subset \hcE_+$ generated by the first three factors of the filtration.
Since the first two factors a linked by the class~$\gamma$ of the twisted dual of~\eqref{eq:pi-tautological}, 
this bundle is an extension of $\pi^*\cE$ by $\pi^*(\cL^\vee \otimes \cG^\vee)$.
Since the functor $\pi^*$ is fully faithful on the derived category of coherent sheaves, 
its extension class is a pullback, 
hence there exists a vector bundle~$\cE''$ on~$X$ and exact sequences
\begin{align}
\label{eq:tce-pp}
0 \to \pi^*\cL^\vee \otimes \pi^*\cG^\vee \to \pi^*\cE'' \to \pi^*\cE \to 0,\\
\label{eq:hce-om1}
0 \to \pi^*\cE'' \to \hcE_+ \to \cO(-1) \to 0.
\end{align}
Since $\Ext^1(\cO(-1),\pi^*\cE'') \cong \Ext^1(\cG,\cE'')$, there is an extension $0 \to \cE'' \to \cE_+ \to \cG \to 0$ on~$X$ 
and a pullback diagram
\begin{equation}
\label{eq:ce-plus}
\vcenter{\xymatrix{
0 \ar[r] &
\pi^*\cE'' \ar[r] \ar@{=}[d] & 
\hcE_+ \ar[r] \ar[d] & 
\cO(-1) \ar[r] \ar[d] & 
0
\\
0 \ar[r] &
\pi^*\cE'' \ar[r] & 
\pi^*\cE_+ \ar[r] & 
\pi^*\cG \ar[r] & 
0
}}
\end{equation}
where the right vertical arrow is the tautological embedding.
The embedding of bundles $\hcE_+ \hookrightarrow \pi^*\cE_+$ in the middle column 
is identical on the subbundle $\pi^*\cE''$
hence the induced morphism
\begin{equation*}
\rho \colon \P_{\P_X(\cG)}(\hcE_+) \to \P_X(\cE_+)
\end{equation*}
is the blowup with center $\P_X(\cE'') \subset \P_X(\cE_+)$, 
i.e., we have~$\P_{\P_X(\cG)}(\hcE_+) \cong \Bl_{\P_X(\cE'')}(\P_X(\cE_+))$.
Note that~$\P_X(\cE'') \subset \P_X(\cE_+)$ is a locally complete intersection, hence
\begin{equation*}
\bR\rho_*\cO_{\P_{\P_X(\cG)}(\hcE_+)} \cong 
\bR\rho_*\cO_{\Bl_{\P_X(\cE'')}(\P_X(\cE_+))} \cong 
\cO_{\P_X(\cE_+)}
\end{equation*}
and therefore the derived pullback functor~$\rho^*$ is fully faithful.

Let $\pi_+ \colon \P_X(\cE_+) \to X$ and $\hat\pi_+ \colon \P_{\P_X(\cG)}(\hcE_+) \to X$ be the projections, 
so that $\hat\pi_+ = \pi_+ \circ \rho$, and we have a commutative diagram
\begin{equation*}
\xymatrix{
& \P_{\P_X(\cG)}(\hcE_+) \ar[dl] \ar[dd]_{\hat\pi_+} \ar[dr]^\rho 
\\
\P_X(\cG) \ar[dr]_\pi &&
\P_X(\cE_+) \ar[dl]^{\pi_+}
\\
& X.
}
\end{equation*}
Let furthermore~$H_+$ and~$\hat{H}_+$ be the relative hyperplane classes of~$\P_X(\cE_+)$ and~$\P_{\P_X(\cG)}(\hcE_+)$, respectively, 
so that~\mbox{$\rho^*\cO(H_+) \cong \cO(\hat{H}_+)$}.
Note that the quadratic form $\hat{q}_+$ can be represented by a section 
of the line bundle~$\hat\pi_+^*\cL^\vee \otimes \cO(2\hat{H}_+)$ on~$\P_{\P_X(\cG)}(\hcE_+)$.
Using full faithfulness of~$\rho^*$ we compute 
\begin{equation*}
\Hom(\hat\pi_+^*\cL, \cO(2\hat{H}_+)) =
\Hom(\rho^*\pi_+^*\cL, \rho^*\cO(2H_+)) \cong
\Hom(\pi_+^*\cL, \cO(2H_+)). 
\end{equation*}
Thus, $\hat{q}_+$ is (in a unique way) the pullback of a section~$q_+$ of the line bundle $\pi_+^*\cL^\vee \otimes \cO(2H_+)$ on~$\P_X(\cE_+)$, 
i.e., $\hat{q}_+ = \rho^*(q_+)$. 
Furthermore, $q_+$ induces a morphism
\begin{equation*}
q_+ \colon \cL \to \Sym^2\cE_+^\vee
\end{equation*}
on~$X$.
It remains to show that $(\cE,q)$ is a hyperbolic reduction of $(\cE_+,q_+)$.

First, note that a combination of~\eqref{eq:tce-pp} and the second row of~\eqref{eq:ce-plus} shows that $\cE_+$ has a filtration
\begin{equation*}
0 \hookrightarrow \cL^\vee \otimes \cG^\vee \hookrightarrow \cE'' \hookrightarrow \cE_+
\end{equation*}
with factors $\cL^\vee \otimes \cG^\vee$, $\cE$, and $\cG$, respectively.
In particular, there is an exact sequence
\begin{equation*}
0 \to \cE \to \cE_+/(\cL^\vee \otimes \cG^\vee) \to \cG \to 0.
\end{equation*}
and the diagram~\eqref{eq:ce-plus} implies that the sequence~\eqref{eq:tce-prime} is its pullback.
Using the natural isomorphism~$\Ext^1(\cO(-1),\pi^*\cE) \cong \Ext^1(\cG,\cE)$ and the definition of~\eqref{eq:tce-prime}
we conclude that the extension class of the above sequence is~$\eps$.
So, we only need to show that the subbundle $\cL^\vee \otimes \cG^\vee \hookrightarrow \cE_+$ is regular isotropic
and that the induced quadratic form on~$\cE$ coincides with~$q$.

The first follows immediately from the fact that $\pi^*\cL^\vee \otimes \Omega_\pi(1) \subset \hcE_+$
is contained in the kernel of the quadratic form~$\hat{q}_+$ (as it was mentioned above)
and that the subbundle $\pi^*\cL^\vee \otimes \cO(1) \subset \tcE_+$ is isotropic for the quadratic form~$\tilde{q}_+$
(because $(\tcE_+,\tilde{q}_+)$ is a hyperbolic extension).
Moreover, by the same reason the induced quadratic form on $\pi^*\cE$ coincides with $\pi^*q$.

To finish the proof of the theorem we must check that any hyperbolic extension of~$(\cE,q)$ comes from the above construction.
So, assume that~$(\cE_+,q_+)$ is a hyperbolic extension of~$(\cE,q)$ with respect to~$\eps$.
Define the bundle~$\hcE_+$ from the diagram~\eqref{eq:ce-plus}, consider the blowup morphism~$\rho$ as above,
and the pullback~\mbox{$\hat{q}_+ = \rho^*(q_+)$} of the quadratic form~$q_+$.
It defines a quadratic form on~$\hcE_+$ over~$\P_X(\cG)$.
It is easy to see that~$\pi^*\cL^\vee \otimes \Omega_\pi(1)$ is contained in the kernel of~$\hat{q}_+$
and that the quotient~$(\tcE_+,\tilde{q}_+)$ (where~$\tcE_+$ is defined by the first sequence in~\eqref{eq:hce-plus})
is a hyperbolic extension of~$\pi^*\cE$ with respect to~\eqref{eq:tce-prime}.
Therefore, by the uniqueness result in Proposition~\ref{prop:he-rank-1} 
this quadratic form coincides with the one constructed in the proof
and the rest of the construction shows that~$(\cE_+,q_+)$ coincides with one of the hyperbolic extensions of the theorem.
\end{proof}

The non-triviality of the construction of hyperbolic extension is demonstrated by the following.

\begin{remark}
\label{rem:q-eps-0}
If $q(\eps) \in \Ext^1(\cG,\cL^\vee \otimes \cE^\vee)$ is zero then $\cE_+$ 
is an extension of $\cG$ by $(\cL^\vee \otimes \cG^\vee) \oplus \cE$.
The component of its extension class in $\Ext^1(\cG,\cE)$ equals~$\eps$,
and the component in $\Ext^1(\cG, \cL^\vee \otimes \cG^\vee)$ is in general non-trivial;
one can identify it with the Massey product $\mu(\eps,q,\eps)$.
\end{remark}

The operation of hyperbolic extension is transitive in the following sense.

\begin{lemma}
Let $(\cE_+,q_+)$ be a hyperbolic extension of $(\cE,q)$ with respect to a $q$-isotropic extension class~$\eps \in \Ext^1(\cG,\cE)$
and let $(\cE_{++},q_{++})$ be a hyperbolic extension of $(\cE_+,q_+)$ 
with respect to a $q_+$-isotropic extension class~$\eps_+ \in \Ext^1(\cG_+,\cE_+)$.
Then $(\cE_{++},q_{++})$ is a hyperbolic extension of $(\cE,q)$.
\end{lemma}

\begin{proof}
By definition the hyperbolic reduction of $(\cE_{++},q_{++})$ 
with respect to \mbox{$\cL^\vee \otimes \cG_+^\vee \hookrightarrow \cE_{++}$} is~$(\cE_+,q_+)$
and the hyperbolic reduction of $(\cE_{+},q_{+})$ with respect to $\cL^\vee \otimes \cG^\vee \hookrightarrow \cE_{+}$ is~$(\cE,q)$.
Therefore, by Lemma~\ref{lemma:hr-transitive} we see that~$(\cE,q)$ is a hyperbolic reduction of $(\cE_{++},q_{++})$,
hence by definition we conclude that $(\cE_{++},q_{++})$ is a hyperbolic extension of $(\cE,q)$.
\end{proof}

\subsection{Hyperbolic equivalence}
\label{subsection:he}

We combine the notions of hyperbolic reduction and extension defined in the previous sections into the notion of hyperbolic equivalence.

\begin{definition}
We say that two quadratic forms~$q_1 \colon \cE_1 \otimes \cL \to \cE_1^\vee$ 
and~$q_2 \colon \cE_2 \otimes \cL \to \cE_2^\vee$
or two quadric bundles~$Q_1 \to X$ and~$Q_2 \to X$
are {\sf hyperbolically equivalent}
if they can be connected by a chain of hyperbolic reductions and hyperbolic extensions.
\end{definition}

Since the operations of hyperbolic reduction and hyperbolic extension are mutually inverse by definition,
this is an equivalence relation.
In this subsection we discuss {\sf hyperbolic invariants}, i.e., invariants 
of quadratic forms and quadric bundles with respect to hyperbolic equivalence.

Recall the invariants~\eqref{eq:wx} and~\eqref{eq:hw} with values in the (non-unimodular) Witt group~$\rWnu(\kk)$ defined in the Introduction.
The hyperbolic invariance of~\eqref{eq:wx} is obvious.

\begin{lemma}
\label{lemma:he-rank}
For any $\kk$-point $x \in X$ and a fixed trivialization of the fiber~$\cL_x$ of the line bundle~$\cL$
the class $\rw_x(\cE.q) = [(\cE_x,q_x)] \in \rWnu(\kk)$ is hyperbolic invariant.
In particular, the parity of~$\rk(\cE)$ is hyperbolic invariant.
\end{lemma}

\begin{proof}
This follows immediately from the fact that if $(\cE_-,q_-)$ is the hyperbolic reduction of~$(\cE,q)$ 
with respect to a regular isotropic subbundle~$\cF$ then $(\cE_{-,x},q_{-,x})$ 
is the sublagrangian reduction of~$(\cE_x,q_x)$ with respect to the subspace~$\cF_x \subset \cE_x$, see~\cite[\S1.1.5]{Balmer}.

Applying the rank parity homomorphism~$\rWnu(\kk) \to \ZZ/2$ we deduce the invariance of the parity of~$\rk(\cE)$ from that of~$\rw_x(\cE,q)$;
alternatively, this invariance can be seen directly from the construction.
\end{proof}

The hyperbolic invariance of~\eqref{eq:hw} requires a bit more work.

\begin{lemma}
\label{lemma:he-rank-cl-prime}
If~$X$ is smooth and proper, $\cL \otimes \omega_X$ is a square in~$\Pic(X)$, and~$n = \dim(X)$ is divisible by~$4$,
the class $\hw(\cE,q) \in \rWnu(\kk)$ is hyperbolic invariant.
In particular, the parity of the rank of the form~$H^{n/2}(q)$ defined by~\eqref{def:q-cl-prime} is hyperbolic invariant.
\end{lemma}
\begin{proof}
Let~$\cM$ be a square root of~$\cL \otimes \omega_X$.
By Serre duality we have
\begin{equation*}
H^{n/2}(X, \cE \otimes \cM)^\vee = 
H^{n/2}(X, \cE^\vee \otimes {\cM}^\vee \otimes \omega_X) \cong
H^{n/2}(X, \cE^\vee \otimes \cL^\vee \otimes \cM).
\end{equation*}
Therefore, the pairing~\eqref{def:q-cl-prime} can be rewritten as the composition of the morphism 
\begin{equation}
\label{eq:q-cl-prime}
H^{n/2}(X, \cE \otimes \cM) \xrightarrow{\ q\ } H^{n/2}(X, \cE^\vee \otimes \cL^\vee \otimes \cM)
\end{equation}
and the Serre duality pairing.

Now assume that $\cF \hookrightarrow \cE$ is a regular isotropic subbundle and $(\cE_-,q_-)$ is the hyperbolic reduction.
It is enough to check that $\hw(\cE,q) = \hw(\cE_-,q_-)$.
Note that $\cE_- \otimes \cM$ and $\cE_-^\vee \otimes \cL^\vee \otimes \cM$ by definition are the cohomology bundles (in the middle terms) of the complexes
\begin{equation}
\label{eq:monad-ce-minus}
\big\{ \cF \otimes \cM \hookrightarrow \cE \otimes \cM \twoheadrightarrow \cF^\vee \otimes \cL^\vee \otimes \cM \big\}
\quad\text{and}\quad 
\big\{ \cF \otimes \cM \hookrightarrow \cE^\vee \otimes \cL^\vee \otimes \cM \twoheadrightarrow \cF^\vee \otimes \cL^\vee \otimes \cM \big\}
\end{equation}
and the morphism $q_- \colon \cE_- \otimes \cM \to \cE_-^\vee \otimes \cL^\vee \otimes \cM$ is induced by the morphism of complexes
\begin{equation*}
\xymatrix{
\cF \otimes \cM \ar[r] \ar[d]^\id & 
\cE \otimes \cM \ar[r] \ar[d]^q & 
\cF^\vee \otimes \cL^\vee \otimes \cM \ar[d]^\id
\\
\cF \otimes \cM \ar[r] & 
\cE^\vee \otimes \cL^\vee \otimes \cM \ar[r] &
\cF^\vee \otimes \cL^\vee \otimes \cM. 
}
\end{equation*}
Therefore, the morphism of cohomology $H^{n/2}(X, \cE_- \otimes \cM) \xrightarrow{\ q\ } H^{n/2}(X, \cE_-^\vee \otimes \cL^\vee \otimes \cM)$ 
is computed by the morphism of the spectral sequences whose first pages look like
\begin{equation*}
\bE_1^{\bullet,\bullet}(\cE_- \otimes \cM) = \left\{
\vcenter{\xymatrix@R=0ex@C=2.5ex{
H^{n/2+1}(X, \cF \otimes \cM) \ar[r] \ar@{..>}[drr] &
H^{n/2+1}(X, \cE \otimes \cM) \ar[r] &
H^{n/2+1}(X, \cF^\vee \otimes \cL^\vee \otimes \cM)
\\
H^{n/2\hphantom{{} + 1}}(X, \cF \otimes \cM) \ar[r] \ar@{..>}[drr] &
H^{n/2\hphantom{{} + 1}}(X, \cE \otimes \cM) \ar[r] &
H^{n/2\hphantom{{} + 1}}(X, \cF^\vee \otimes \cL^\vee \otimes \cM)
\\
H^{n/2-1}(X, \cF \otimes \cM) \ar[r] &
H^{n/2-1}(X, \cE \otimes \cM) \ar[r] &
H^{n/2-1}(X, \cF^\vee \otimes \cL^\vee \otimes \cM)
}}\right\}
\end{equation*}
(dotted arrows show the directions of the only higher differentials~$\bd_2$) and
\begin{equation*}
\bE_1^{\bullet,\bullet}(\cE_-^\vee \otimes \cL^\vee \otimes \cM) = \left\{
\vcenter{\xymatrix@R=0ex@C=2.3ex{
H^{n/2+1}(X, \cF \otimes \cM) \ar[r] \ar@{..>}[drr] &
H^{n/2+1}(X, \cE^\vee \otimes \cL^\vee \otimes \cM) \ar[r] &
H^{n/2+1}(X, \cF^\vee \otimes \cL^\vee \otimes \cM)
\\
H^{n/2\hphantom{{} + 1}}(X, \cF \otimes \cM) \ar[r] \ar@{..>}[drr] &
H^{n/2\hphantom{{} + 1}}(X, \cE^\vee \otimes \cL^\vee \otimes \cM) \ar[r] &
H^{n/2\hphantom{{} + 1}}(X, \cF^\vee \otimes \cL^\vee \otimes \cM)
\\
H^{n/2-1}(X, \cF \otimes \cM) \ar[r] &
H^{n/2-1}(X, \cE^\vee \otimes \cL^\vee \otimes \cM) \ar[r] &
H^{n/2-1}(X, \cF^\vee \otimes \cL^\vee \otimes \cM)
}}\right\}.
\end{equation*}
Moreover, the morphism of spectral sequences is equal to the identity on the first 
and last columns and is induced by~$q$ on the middle column.
On the other hand, by Serre duality
\begin{equation*}
H^{i}(X, \cF \otimes \cM)^\vee = 
H^{n - i}(X, \cF^\vee \otimes {\cM}^\vee \otimes \omega_X) \cong
H^{n - i}(X, \cF^\vee \otimes \cL^\vee \otimes \cM)
\end{equation*}
hence the morphism of spectral sequences is self-dual.

It follows that $(H^{n/2}(X,\cE_- \otimes \cM),H^{n/2}(q_-))$ is obtained from~$(H^{n/2}(X,\cE \otimes \cM),H^{n/2}(q))$
by a composition of the hyperbolic reduction with respect to the regular isotropic subspace
\begin{equation*}
\Ima(\bE_1^{-1,n/2}(\cE_- \otimes \cM) \to \bE_1^{0,n/2}(\cE_- \otimes \cM)) = 
\Ima\left(H^{n/2}(X, \cF \otimes \cM) \to H^{n/2}(X, \cE \otimes \cM)\right)
\end{equation*}
followed by a hyperbolic extension with respect to the space
\begin{equation*}
\bE_3^{1,n/2-1}(\cE_- \otimes \cM) = 
\Coker\left( 
\bE_2^{-1,n/2}(\cE_- \otimes \cM) \xrightarrow{\ \bd_2\ } \bE_2^{1,n/2-1}(\cE_- \otimes \cM).
\right)
\end{equation*}
Therefore, we have the required equality~$\hw(\cE_-,q_-) = \hw(\cE,q)$ 
in the Witt group~$\rWnu(\kk)$.

Applying the rank parity homomorphism~$\rWnu(\kk) \to \ZZ/2$ we deduce 
the invariance of the parity of the rank of~$H^{n/2}(q)$ from that of~$\hw(\cE,q)$.
\end{proof}

Other hyperbolic invariants of quadric bundles have been listed in Proposition~\ref{prop:he-invariants}.
We are ready now to prove this proposition.

\begin{proof}[Proof of Proposition~\textup{\ref{prop:he-invariants}}]
Since assertion~\eqref{item:dim-mod-2} is clear from the definition (or follows from Lemma~\ref{lemma:he-rank}),
it is enough to prove assertions~\eqref{item:he-cokernel}--\eqref{item:he-witt} of the proposition.
Moreover, in most cases it is enough to prove the assertions for a single hyperbolic reduction.
So, assume that~\eqref{eq:morphism-general} is a self-dual morphism 
and $(\cE_-,q_-)$ is its hyperbolic reduction with respect to a regular isotropic subbundle~$\cF \hookrightarrow \cE$.

By Lemma~\ref{lemma:h-reduction} we have~$\cC(q) \cong \cC(q_-)$, an isomorphism compatible with the shifted quadratic forms; 
this proves assertion~\eqref{item:he-cokernel}.
Furthermore, the equality of the discriminant divisors
\begin{equation*}
D = \supp(\cC(Q)) = \supp(\cC(Q')) = D'
\end{equation*}
follows as well and proves~\eqref{item:he-discriminant}.
Similarly, \eqref{item:he-witt} follows from Lemma~\ref{lemma:witt} and Witt's Cancellation Theorem.

Now we prove~\eqref{item:he-clifford}.
We refer to~\cite{K08} for generalities about sheaves of Clifford algebras and modules.
Here we just recall that for a vector bundle $\cE$ with a quadratic form $q \colon \cL \to \Sym^2\cE^\vee$ we denote 
\begin{align*}
\Cl_0(\cE,q) &= \hbox to 1.em{$\cO$} \oplus (\bw2\cE \otimes \cL) \oplus (\bw4\cE \otimes \cL^{\otimes 2}) \oplus \dots,\\
\Cl_1(\cE,q) &= \hbox to 1.em{$\cE$} \oplus (\bw3\cE \otimes \cL) \oplus (\bw5\cE \otimes \cL^{\otimes 2}) \oplus \dots,
\end{align*}
and set $\Cl_{i+2}(\cE,q) = \cL^\vee \otimes \Cl_i(\cE,q)$.
The Clifford multiplication (see~\cite[\S3]{K14})
\begin{equation*}
\Cl_i(\cE,q) \otimes \Cl_j(\cE,q) \to \Cl_{i+j}(\cE,q)
\end{equation*}
(induced by~$q$ and the wedge product on~$\bw\bullet\cE$)
provides $\Cl_0(\cE,q)$ with the structure of $\cO_X$-algebra (called the {\sf sheaf of even parts of Clifford algebras}) 
and each $\Cl_i(\cE,q)$ with the structure of $\Cl_0(\cE,q)$-bimodule.
In the case where the line bundle $\cL$ is trivial, the sum 
\begin{equation*}
\Cl(\cE,q) = \Cl_0(\cE,q) \oplus \Cl_1(\cE,q)
\end{equation*}
also acquires a structure of~$\cO_X$-algebra (called the {\sf total Clifford algebra}), 
which is naturally $\ZZ/2$-graded. 

Now consider the subbundle~$\cF^\perp \subset \cE$ defined by~\eqref{def:cf-perp}.
It comes with the quadratic form $q_{\cF^\perp}$, the restriction of the form~$q$,
so that the subbundle~$\cF \subset \cF^\perp$ is contained in the kernel of~$q_{\cF^\perp}$ 
and the induced quadratic form on the quotient~$\cF^\perp/\cF = \cE_-$ coincides with~$q_-$.
Thus, the maps~$\cF^\perp \hookrightarrow \cE$ and~$\cF^\perp \twoheadrightarrow \cE_-$ are morphisms of quadratic spaces.
Therefore, they are compatible with the Clifford multiplications 
and induce $\cO_X$-algebra morphisms of sheaves of even parts of Clifford algebras
\begin{equation*}
\Cl_0(\cF^\perp,q_{\cF^\perp}) \hookrightarrow \Cl_0(\cE,q)
\qquad\text{and}\qquad 
\Cl_0(\cF^\perp,q_{\cF^\perp}) \twoheadrightarrow \Cl_0(\cE_-,q_-).
\end{equation*}
The kernel of the second morphism is the two-sided ideal 
\begin{equation*}
\cR := \Ima(\cF \otimes \Cl_{-1}(\cF^\perp,q_{\cF^\perp}) \to \Cl_0(\cF^\perp,q_{\cF^\perp})),
\end{equation*}
where the arrow is the natural morphism induced
by the embedding~$\cF \hookrightarrow \cF^\perp \hookrightarrow \Cl_1(\cF^\perp,q_{\cF^\perp})$ and the Clifford multiplication.

Now we denote~$k = \rk(\cF)$ and consider the right ideal in~$\Cl_0(\cE,q)$ defined as 
\begin{equation*}
\cP = \Ima \Big( \bw{k}\cF \otimes \Cl_{-k}(\cE,q) \to \Cl_0(\cE,q) \Big).
\end{equation*}
Since~$\cF^\perp$ is the orthogonal of~$\cF$ with respect to~$q$, 
the subalgebra~$\Cl_0(\cF^\perp,q_{\cF^\perp}) \subset \Cl_0(\cE,q)$ anticommutes with~$\bw{k}\cF \subset \Cl_{k}(\cE,q)$, 
hence~$\cP$ is invariant under the left action of~$\Cl_0(\cF^\perp,q_{\cF^\perp})$ on~$\Cl_0(\cE,q)$.
Furthermore, since~$\cF$ is isotropic, the Clifford multiplication vanishes on~$\cF \otimes \bw{k}\cF$, hence the ideal~$\cR$ annihilates~$\cP$.
Therefore, $\cP$ has the structure of a left module over the algebra 
\begin{equation*}
\Cl_0(\cF^\perp,q_{\cF^\perp}) / \cR \cong \Cl_0(\cE_-,q_-).
\end{equation*}
This structure obviously commutes with the right $\Cl_{0}(\cE,q)$-module structure, 
hence $\cP$ is naturally a~$(\Cl_0(\cE_-,q_-),\Cl_0(\cE,q))$-bimodule.
We show below that $\cP$ defines the required Morita equivalence.

The question now is local over~$X$, so we may assume that $\cL = \cO_X$ and there is an orthogonal direct sum decomposition
\begin{equation}
\label{eq:ce-sum}
\cE = \cE_- \oplus \cE_0,
\qquad
q = q_- \perp q_0,
\end{equation}
where $\cE_0 = \cF \oplus \cF^\vee$ and the quadratic form $q_0$ is given by the natural pairing $\cF \otimes \cF^\vee \to \cO_X$.
Furthermore, as~$\cL = \cO_X$, we can consider the total $\ZZ/2$-graded Clifford algebras.

On the one hand, the orthogonal direct sum decomposition~\eqref{eq:ce-sum} implies the natural isomorphism
\begin{equation*}
\Cl(\cE,q) \cong \Cl(\cE_-,q_-) \otimes \Cl(\cE_0,q_0)
\end{equation*}
(where the right-hand side is the tensor product in the category of~$\ZZ/2$-graded algebras),
compatible with the gradings.
On the other hand, since $\cF \subset \cE_0$ is Lagrangian, the algebra
\begin{equation*}
\Cl(\cE_0, q_0) \cong \cEnd(\bw{\bullet}\cF)
\end{equation*}
is Morita trivial, and its $\ZZ/2$-grading is induced by the natural $\ZZ/2$-grading of $\bw{\bullet}\cF$.
It follows that the~$(\Cl(\cE_-,q_-),\Cl(\cE,q))$-bimodule
\begin{equation*}
\tilde\cP = \Cl(\cE_-,q_-) \otimes \bw{\bullet}\cF
\end{equation*}
defines a Morita equivalence of $\Cl(\cE_-,q_-)$ and $\Cl(\cE,q)$, compatible with the grading.
Therefore, the even part of~$\tilde\cP$
\begin{equation*}
\tilde\cP_0 = \Big( \Cl_0(\cE_-,q_-) \otimes \bw{\mathrm{even}}\cF \Big) \oplus \Big( \Cl_1(\cE_-,q_-) \otimes \bw{\mathrm{odd}}\cF \Big)
\end{equation*}
defines a Morita equivalence between the even Clifford algebras $\Cl_0(\cE_-,q_-)$ and $\Cl_0(\cE,q)$.
Finally, a simple computation shows that the globally defined bimodule $\cP$ is locally isomorphic to the bimodule~$\tilde\cP_0$,
hence it defines a global Morita equivalence.

In conclusion we prove~\eqref{item:l-equivalence}.
To show that~$[Q] = [Q']$ we will first show that for any point~$x \in X$ 
there is a Zariski neighborhood~$x \in U \subset X$ such that~$Q_U \cong Q'_U$
(where~$Q_U = Q \times_X U$ and~$Q'_U = Q' \times_X U$), hence a fortiori~$[Q_U] = [Q'_U]$,
and after that we will use this local equality to deduce the global one.

Since we are going to work locally, we may assume that the line bundle~$\cL$ is trivial and the base is affine.
Then two things happen with hyperbolic extension --- 
first, any extension class~$\eps \in \Ext^1(\cG,\cE)$ vanishes (in particular, any such class is~$q$-isotropic),
and second, the group~$\Ext^1(\bw2\cG,\cE)$ vanishes as well, so that the result of hyperbolic extension becomes unambiguous.
Moreover, it is clear that this result becomes isomorphic to~$\cE_+ = \cE \oplus (\cG \oplus \cG^\vee)$,
the orthogonal direct sum of~$\cE$ and~$\cG \oplus \cG^\vee$, 
with the quadratic form on~$\cG \oplus \cG^\vee$ induced by duality.
Similarly, hyperbolic reduction reduces to splitting off an orthogonal summand~$\cF \oplus \cF^\vee$.
Thus, locally, hyperbolic equivalence turns into Witt equivalence (in the non-unimodular Witt ring of the base scheme).
Therefore, a hyperbolic equivalence between~$Q$ and~$Q'$ locally can be realized by a single quadric bundle~$\hat{Q}$ 
such that both~$Q$ and~$Q'$ are obtained from~$\hat{Q}$ by hyperbolic reduction.
In other words, we may assume that the quadrics~$Q$ and~$Q'$ correspond to quadratic forms
obtained from a single quadratic form~$(\hcE,\hat{q})$ 
by isotropic reduction with respect to regular isotropic subbundles~$\cF \subset \hcE$ and~$\cF' \subset \hcE$ of the same rank.
Below we prove isomorphism of~$Q$ and~$Q'$ in a neighborhood of~$x$
by induction on the rank of~$\cF$ and~$\cF'$.

First assume that the rank of~$\cF$ and~$\cF'$ is~$1$ and~$\hat{q}(\cF,\cF') \ne 0$ at~$x$ 
(hence also in a neighborhood of~$x$).
Since~$\cF$ and~$\cF'$ are isotropic, the restriction of~$\hat{q}$ to~$\cF \oplus \cF'$ is non-degenerate, 
hence there is an orthogonal direct sum decomposition
\begin{equation*}
\hcE = \bar\cE \oplus (\cF \oplus \cF').
\end{equation*}
Then obviously~$\cF^\perp = \bar\cE \oplus \cF$, hence the hyperbolic reduction of~$(\hcE,\hat{q})$ with respect to~$\cF$
is isomorphic to~$(\bar\cE,q\vert_{\bar\cE})$.
Similarly, the hyperbolic reduction of~$(\hcE,\hat{q})$ with respect to~$\cF'$
is isomorphic to~$(\bar\cE,q\vert_{\bar\cE})$ as well.
In particular, the two hyperbolic reductions are isomorphic.

On the other hand, assume that the rank of~$\cF$ and~$\cF'$ is~$1$ and~$\hat{q}(\cF,\cF')$ vanishes at~$x$.
Then we find (locally) yet another regular isotropic subbundle~$\cF'' \subset \hcE$ 
such that~$\hat{q}(\cF,\cF'') \ne 0$ and~$\hat{q}(\cF',\cF'') \ne 0$ at~$x$.
Let~$v,v' \in \hcE_x$ be the points corresponding to~$\cF$, $\cF'$.
Let~$v'' \in \hcE_x$ be a point such that~$\hq_x(v,v'') \ne 0$ and~$\hq_x(v',v'') \ne 0$.
The existence of a regular subbundle~$\cF$ implies rationality of~$\hQ$ over~$X$,
hence (maybe over a smaller neighborhood of~$x$) 
there esists a regular isotropic subbundle~$\cF''$ corresponding to the point~$v''$.
Now, when we have such~$\cF''$, we apply the previous argument 
and conclude that the hyperbolic reduction of~$(\hcE,\hat{q})$ with respect to~$\cF''$
is isomorphic to the hyperbolic reductions with respect to~$\cF$ and~$\cF'$, 
hence the latter two reductions are mutually isomorphic.

Now assume the rank of~$\cF$ and~$\cF'$ is bigger than~$1$.
Shrinking the neighborhood of~$x$ if necessary, we may split~$\cF = \cF_1 \oplus \cF_2$ and~$\cF' = \cF'_1 \oplus \cF'_2$,
where the rank of~$\cF_1$ and~$\cF'_1$ is~$1$.
The above argument shows that the isotropic reductions of~$(\hcE,\hat{q})$ 
with respect to~$\cF_1$ and~$\cF'_1$ are isomorphic.
Hence~$Q$ and~$Q'$ correspond to hyperbolic reductions of the same quadratic form 
with respect to regular isotropic subbundles~$\cF_2$ and~$\cF'_2$, which have smaller rank than~$\cF$ and~$\cF'$,
and therefore by induction~$Q$ and~$Q'$ are isomorphic.

Finally, we deduce the global result from the local results obtained above.
Indeed, the argument above and quasi-compactness of~$X$ imply that~$X$ has an open covering $\{U_i\}$ 
such that over each~$U_i$ we have an isomorphism~$Q_{U_i} \cong Q'_{U_i}$, 
hence an equality~$[Q_{U_i}] =  [Q'_{U_i}]$ in the Grothendieck ring of varieties. 
For any finite set~$I$ of indices set~$U_I = \cap_{i \in I}U_i$.
Then inclusion-exclusion gives
\begin{equation*}
[Q] = \sum_{|I| \ge 1} (-1)^{|I|-1} [Q_{U_I}],
\qquad 
[Q'] = \sum_{|I| \ge 1} (-1)^{|I|-1} [Q'_{U_I}],
\end{equation*}
and since by base change we have isomorphisms~$Q_{U_I} \cong Q'_{U_I}$, 
hence equalities~$[Q_{U_I}] =  [Q'_{U_I}]$ for each~$I$ with~$|I| \ge 1$,
the equality~$[Q] = [Q']$ follows.
\end{proof}

\begin{remark}
The same technique proves the following more general formula
\begin{equation}
\label{eq:quadric-relation}
[Q] = [Q']\,\LL^{d} + [X]\,[\P^{d-1}]\,(1  + \LL^{n - d + 1})
\end{equation}
for any hyperbolic equivalent quadric bundles~$Q/X$ and~$Q'/X$, where~$n = \dim(Q/X)$ 
and we assume that it is greater or equal than~$\dim(Q'/X)$, which we write in the form~$\dim(Q'/X) = n - 2d$.
Indeed, first~\eqref{eq:quadric-relation} can be proved over a small neighborhood of any point of~$X$;
for this the same argument reduces everything to the case where~$Q'$ is a hyperbolic reduction of~$Q$, 
in which case the formula is proved in~\cite[Corollary~2.7]{KS18}.
After that the inclusion-exclusion trick proves~\eqref{eq:quadric-relation} in general.
\end{remark}

\section{VHC resolutions on projective spaces}
\label{sec:hacm}

From now on we consider the case $X = \P^n$.
This section serves as a preparation for the next one.
Here we introduce a class of locally free resolutions (which we call VHC resolutions) which plays the main role in~\S\ref{sec:hacm-pn}
and show that on~$\P^n$ any sheaf of projective dimension~1 has a (essentially unique) VHC resolution,
see Corollary~\ref{cor:hacm-resolution} for existence and Theorem~\ref{prop:hacm-uniqueness} for uniqueness.

\subsection{Complexes of split bundles}

For each coherent sheaf $\cF$ on~$\P^n = \P(V)$ 
(and more generally, for any object of the bounded derived category $\Db(\P^n)$ of coherent sheaves) 
and each integer~$p$ we write
\begin{equation}
\label{def:hh}
\HH^p(\cF) := \bigoplus_{t = -\infty}^\infty H^p(\P^n,\cF(t)).
\end{equation}
This is a graded module over the homogeneous coordinate ring 
\begin{equation}
\label{def:ss}
\SS = \HH^0(\cO_{\P^n}) \cong
\Sym^\bullet(V^\vee) = 
\kk \oplus V^\vee \oplus \Sym^2V^\vee \oplus \dots.
\end{equation} 
For a sheaf~$\cF$ we will often consider the $\SS$-module of \emph{intermediate cohomology}
\begin{equation*}
\moplus_{p=1}^{n-1} \HH^p(\cF) = \moplus_{p=1}^{n-1} \left( \moplus_{t = -\infty}^\infty H^p(\P^n,\cF(t)) \right).
\end{equation*}
as a bigraded $\SS$-module; with index~$p$ corresponding to the \emph{homological} 
and index~$t$ corresponding to the \emph{internal} grading.
We will use notation~$[p]$ and~$(t)$ for the corresponding shifts of grading. 

Recall the following well-known result.

\begin{lemma}
\label{lemma:epi-from-split}
Let $\cF$ be a coherent sheaf, so that the $\SS$-module $\HH^0(\cF)$ is finitely generated.
The minimal epimorphism  $\moplus \SS(t_i) \to \HH^0(\cF)$ of graded $\SS$-modules
gives rise to an epimorphism $\moplus \cO(t_i) \to \cF$ 
such that the induced morphism $\moplus \SS(t_i) = \HH^0(\moplus \cO(t_i)) \to \HH^0(\cF)$ coincides with the original epimorphism.
\end{lemma}
\begin{proof}
Any morphism of graded $\SS$-modules $\SS(t) \to \HH^0(\cF)$ 
is given by an element in the graded component~$H^0(\P^n,\cF(-t)) = \Hom(\cO(t),\cF)$ of~$\HH^0(\cF)$,
hence the epimorphism $\moplus \SS(t_i) \to \HH^0(\cF)$ corresponds to a morphism $\moplus \cO(t_i) \to \cF$.
The only nontrivial statement here is the surjectivity of this morphism.
To prove it let $\cK$ and $\cC$ denote its kernel and cokernel, so that we have an exact sequence
\begin{equation*}
0 \to \cK \to \moplus \cO(t_i) \to \cF \to \cC \to 0.
\end{equation*}
We need to show~$\cC = 0$.
When twisted by $\cO(t)$ with $t \gg 0$ all sheaves above have no higher cohomology, therefore there is an exact sequence
\begin{equation*}
0 \to H^0(\P^n,\cK(t)) \to  H^0(\P^n, \moplus \cO(t_i+t)) \to H^0(\P^n,\cF(t)) \to H^0(\P^n,\cC(t)) \to 0.
\end{equation*}
By assumption the middle arrow is surjective, hence $H^0(\P^n,\cC(t)) = 0$ for $t \gg 0$.
Therefore, $\cC = 0$.
\end{proof}

We will say that $\cE$ is a {\sf split bundle} if it is isomorphic to a direct sum of line bundles
(note that by Horrocks' Theorem a vector bundle~$\cE$ is split if and only if $\HH^p(\cE) = 0$ for all $1 \le p \le n-1$).

We will need the following simple generalization of the Horrocks' Theorem.

\begin{lemma}
\label{lem:split-resolution}
Let~$\cE$ be a vector bundle on~$\P^n$ and let~$0 \le \ell \le n - 1$.
Then $\HH^p(\cE) = 0$ for all~$p$ such that~$1 \le p \le n - \ell - 1$ if and only if~$\cE$ has a resolution of length~$\ell$
\begin{equation*}
0 \to \cL_\ell \to \dots \to \cL_1 \to \cL_0 \to \cE \to 0
\end{equation*}
by split bundles.
\end{lemma}
\begin{proof}
We use induction on~$\ell$.
If $\ell = 0$ the result follows from the Horrocks' Theorem.
Assume~\mbox{$\ell > 0$}.
Choose an epimorphism $\cL_0 \twoheadrightarrow \cE$ from a split bundle~$\cL_0$ which is surjective on~$\HH^0$ 
(it exists by Lemma~\ref{lemma:epi-from-split}) and let $\cE'$ be its kernel.
The cohomology exact sequence 
\begin{equation*}
\dots \to \HH^{p-1}(\cL_0) \to \HH^{p-1}(\cE) \to \HH^p(\cE') \to \HH^p(\cL_0) \to \HH^p(\cE) \to \dots 
\end{equation*}
implies that $\HH^p(\cE') = 0$ for $1 \le p \le n - \ell$.
By the induction hypothesis $\cE'$ has a resolution of length~$\ell - 1$
\begin{equation*}
0 \to \cL_\ell \to \dots \to \cL_1 \to \cE' \to 0
\end{equation*}
by split bundles.
It follows that the complex $\cL_\ell \to \dots \to \cL_1 \to \cL_0$ 
(where the morphism $\cL_1 \to \cL_0$ is defined as the composition $\cL_1 \twoheadrightarrow \cE' \hookrightarrow \cL_0$)
is a resolution of~$\cE$ of length~$\ell$ by split bundles.

The converse statement follows immediately from the hypercohomology spectral sequence applied to the resolution
since the intermediate cohomology of split bundles vanishes.
\end{proof}

The following obvious observation about complexes is quite useful.

\begin{lemma}
\label{lemma:splitting-complex}
Let $\cL_\bullet$ be a complex of coherent sheaves
such that $\cL_i = \cO(t) \oplus \cL'_i$, $\cL_{i-1} = \cO(t) \oplus \cL'_{i-1}$ for some~$i \in \ZZ$ and~$t \in \ZZ$,
and the differential $\rd_i \colon \cL_i \to \cL_{i-1}$ of $\cL_\bullet$ 
takes the summand~$\cO(t)$ of~$\cL_i$ isomorphically to the summand~$\cO(t)$ of~$\cL_{i-1}$.
Then there is an isomorphism of complexes
\begin{equation}
\label{eq:complex-splitting}
\cL_\bullet \cong \cL'_\bullet \oplus (\cO(t) \xrightarrow{\ \id\ } \cO(t))[i].
\end{equation} 
\end{lemma}
\begin{proof}
By assumption the differential $\rd_i$ can be written in the form
\begin{equation*}
\cO(t) \oplus \cL'_i \xrightarrow{\ \left(\begin{smallmatrix} 1 & f \\ 0 & \rd'_i \end{smallmatrix}\right)\ } \cO(t) \oplus \cL'_{i-1}
\end{equation*}
for some $f \in \Hom(\cL'_i,\cO(t))$, $\rd'_i \in \Hom(\cL'_i,\cL'_{i-1})$. 
After the modification of the direct sum decomposition of $\cL_i$ 
by the automorphism $\left(\begin{smallmatrix} 1 & f \\ 0 & 1 \end{smallmatrix}\right) \in \End(\cO(t) \oplus \cL'_i)$
this differential takes the form $\left(\begin{smallmatrix} 1 & 0 \\ 0 & \rd'_i \end{smallmatrix}\right)$.
Then the equalities $\rd_i \circ \rd_{i+1} = 0$ and $\rd_{i-1} \circ \rd_i = 0$ imply that
\begin{equation*}
\rd_{i+1} = \left(\begin{smallmatrix} 0 \\ \rd'_{i+1} \end{smallmatrix}\right) \in \Hom(\cL_{i+1}, \cO(t) \oplus \cL'_i),
\qquad 
\rd_{i-1} = \left(\begin{smallmatrix} 0, & \rd'_{i-1} \end{smallmatrix}\right) \in \Hom(\cO(t) \oplus \cL'_{i-1}, \cL'_{i-2}),
\end{equation*}
which implies~\eqref{eq:complex-splitting},
where $\cL'_j = \cL_j$ for $j \not\in \{i,i-1\}$.
\end{proof}

Now let
\begin{equation*}
\cL_\bullet := \{\cL_\ell \to \cL_{\ell - 1} \to \dots \to \cL_1 \to \cL_0\}
\end{equation*}
be a complex of split bundles on~$\P^n$. 
Since split bundles have no intermediate cohomology, the first page of the hypercohomology spectral sequence of~$\cL_\bullet$ has only two nontrivial rows:
\begin{equation*}
\xymatrix@R=0ex{
H^n(\cL_\ell) \ar[r] \ar@{-->}[ddddrr] & 
H^n(\cL_{\ell-1}) \ar[r] \ar@{-->}[ddddrr] &  
\dots \ar[r] & 
\dots \ar[r] & 
H^n(\cL_{n+1}) \ar[r] \ar@{-->}[ddddrr] & 
\dots \ar[r] & 
H^n(\cL_0)
\\
0 & 0 & \dots & \dots & 0 & \dots & 0
\\
\vdots &&&&&& \vdots 
\\
0 & \dots & 0 & 0 & \dots & \dots & 0
\\
H^0(\cL_\ell) \ar[r] & 
\dots \ar[r] & 
H^0(\cL_{\ell-n-1}) \ar[r] & 
H^0(\cL_{\ell-n-2}) \ar[r] & 
\dots \ar[r] & 
\dots \ar[r] & 
H^0(\cL_0),
}
\end{equation*}
one formed by~$H^0(\P^n,\cL_i)$ and the other by~$H^n(\P^n,\cL_i)$.
The dashed arrows show the only non-trivial higher differentials~$\bd_{n+1}$ --- 
these differentials are directed~$n$ steps down and~$n+1$ steps to the right.
Therefore, if~$\ell \le n$ there are no higher differentials, and if~$\ell = n + 1$ there is exactly one, 
which acts from~$H^\tp(\P^n,\cL_\bullet)$ to~$H^\bt(\P^n,\cL_\bullet)$, where we define
\begin{equation}
\label{def:h-top-bot}
\begin{aligned}
H^\tp(\P^n, \cL_\bullet) &:= \Ker(H^n(\P^n,\cL_\ell) \to H^n(\P^n,\cL_{\ell - 1})),
\\
H^\bt(\P^n, \cL_\bullet) &:= \Coker(H^0(\P^n,\cL_1) \to H^0(\P^n,\cL_0)).
\end{aligned}
\end{equation} 
We also set~$\HH^\tp(\cL_\bullet) = \moplus_t H^\tp(\P^n, \cL_\bullet(t))$ 
and~$\HH^\bt(\cL_\bullet) = \moplus_t H^\bt(\P^n, \cL_\bullet(t))$,

\begin{lemma}
\label{lemma:split-n}
If a complex $\cL_\bullet$ of split bundles quasiisomorphic to an object~$\cF$ of the derived category~$\Db(\P^n)$ 
has length~$\ell = n$ then there is a canonical exact sequence
\begin{equation*}
0 \to \HH^\bt(\cL_\bullet) \to \HH^0(\cF) \to \HH^\tp(\cL_\bullet) \to 0.
\end{equation*}
\end{lemma}
\begin{proof}
This follows immediately from the hypercohomology spectral sequence.
\end{proof}

The next two lemmas are crucial for the rest of the paper.

\begin{lemma}
\label{lemma:split-n+1}
If an acyclic complex $\cL_\bullet$ of split bundles has length~$\ell = n + 1$ then the following conditions are equivalent
\begin{enumerate}
\item 
\label{cond:h-bt}
$\HH^\bt(\cL_\bullet) = 0$;
\item 
\label{cond:h-tp}
$\HH^\tp(\cL_\bullet) = 0$;
\item 
\label{cond:h-bt-tp}
the canonical morphism $\bd_{n+1} \colon \HH^\tp(\cL_\bullet) \to \HH^\bt(\cL_\bullet)$ is zero;
\item 
\label{cond:sum}
the complex $\cL_\bullet$ is isomorphic to a direct sum of shifts of trivial complexes $\cO(t) \xrightarrow{\ \id\ } \cO(t)$.
\end{enumerate}
\end{lemma}

\begin{proof}
Since $\cL_\bullet$ is acyclic, its hypercohomology spectral sequence converges to zero, 
hence the canonical morphism $\bd_{n+1} \colon \HH^\tp(\cL_\bullet) \to \HH^\bt(\cL_\bullet)$ is an isomorphism.
It follows that~\eqref{cond:h-bt}, \eqref{cond:h-tp}, and~\eqref{cond:h-bt-tp} are equivalent.

Now we prove~\eqref{cond:h-bt-tp}~$\implies$~\eqref{cond:sum}. 
So, assume~\eqref{cond:h-bt-tp} holds.
Then for each~$t$ the hypercohomology spectral sequence of~$\cL_\bullet(-t)$ degenerates on the second page;
in particular the bottom row of the first page is exact.
Let~$t$ be the maximal integer such that~$\cO(t)$ appears as one of summands of one of the split bundles~$\cL_i$.
Then the bottom row of the first page of the hypercohomology spectral sequence of~$\cL_\bullet(-t)$ 
is nonzero and takes the form
\begin{equation*}
H^0(\P^n, \cL_\bullet(-t)) = \Big\{ \kk^{m_\ell} \to \kk^{m_{\ell-1}} \to \dots \to \kk^{m_1} \to \kk^{m_0} \Big\}, 
\end{equation*}
where~$m_i$ is the multiplicity of~$\cO(t)$ in~$\cL_i$.
Since this complex is exact, it is a direct sum of shifts of trivial complexes $\kk \xrightarrow{\ \id\ } \kk$.
Since~$\Hom(\cO(t),\cO(t')) = 0$ for all~$t' < t$, it follows that~$\cL_\bullet$ 
contains the subcomplex~$H^0(\P^n, \cL_\bullet(-t)) \otimes \cO(t)$;
this subcomplex is isomorphic to a direct sum of shifts of trivial complexes~$\cO(t) \xrightarrow{\ \id\ } \cO(t)$,
and each of its terms is a direct summand of the corresponding term of~$\cL_\bullet$.
Applying Lemma~\ref{lemma:splitting-complex} to one of these trivial subcomplexes 
we obtain the direct sum decomposition~\eqref{eq:complex-splitting}. 
The condition~\eqref{cond:h-bt-tp} holds for~$\cL'_\bullet$
(because it is a direct summand of~$\cL_\bullet$), 
hence by induction~$\cL'_\bullet$ is the sum of shifts of trivial complexes,
and hence the same is true for~$\cL_\bullet$;
which means that~\eqref{cond:sum} holds.

The implication \eqref{cond:sum}~$\implies$~\eqref{cond:h-bt-tp} is evident. 
\end{proof}

\begin{lemma}
\label{lemma:morphism-of-complexes}
Assume objects $\cF$ and $\cF'$ in~$\Db(\P^n)$ are quasiisomorphic 
to complexes $\cL_\bullet$ and $\cL'_\bullet$ of split bundles of length~$\ell$.
If $\ell < n$ then any morphism $\varphi \colon \cF \to \cF'$ is induced by a morphism of complexes
\begin{equation*}
\xymatrix{
\cL_\ell \ar[r] \ar[d]_{\varphi_\ell} &
\cL_{\ell-1} \ar[r] \ar[d]_{\varphi_{\ell-1}} &
\dots \ar[r] &
\cL_1 \ar[r] \ar[d]_{\varphi_1} &
\cL_0 \ar[d]_{\varphi_0} 
\\
\cL'_\ell \ar[r] &
\cL'_{\ell-1} \ar[r] &
\dots \ar[r] &
\cL'_1 \ar[r] &
\cL'_0
}
\end{equation*}
If $\ell = n$ the same is true for a morphism~$\varphi \colon \cF \to \cF'$
if and only if the composition 
\begin{equation}
\label{eq:bt-tp-composition}
\HH^\bt(\cL_\bullet) \hookrightarrow \HH^0(\cF) \xrightarrow{\ \HH^0(\varphi)\ } \HH^0(\cF') \twoheadrightarrow \HH^\tp(\cL'_\bullet)
\end{equation}
vanishes, where the first and last morphisms are defined in Lemma~\textup{\ref{lemma:split-n}}.
Moreover, in both cases a morphism of complexes~$\varphi_\bullet$ inducing a morphism~$\varphi$ as above
is unique up to a homotopy~$h_\bullet \colon \cL_\bullet \to \cL'_{\bullet + 1}$.
\end{lemma}

\begin{proof}
Obviously, the first page of the spectral sequence
\begin{equation*}
\bE_1^{p,q} = \moplus_i \Ext^q(\cL_i,\cL'_{i-p}) \Rightarrow \Ext^{p+q}(\cF,\cF')
\end{equation*}
is non-zero only when $-\ell \le p \le \ell$ and $q \in \{0,n\}$.
Consequently, we have an exact sequence
\begin{equation*}
0 \to \bE_\infty^{0,0} \to \Hom(\cF, \cF') \to \bE_\infty^{-n,n} \to 0
\end{equation*}
and (under the assumption~$\ell \le n$) the last term is non-zero only if~$\ell = n$.
Furthermore, we have
\begin{equation*}
\bE_\infty^{0,0} = \bE_2^{0,0} = 
\Ker \Big( \moplus_i \Hom(\cL_i,\cL'_{i}) \to \moplus_i \Hom(\cL_i,\cL'_{i-1}) \Big) \Big/
\Ima \Big( \moplus_i \Hom(\cL_i,\cL'_{i+1}) \to \moplus_i \Hom(\cL_i,\cL'_{i}) \Big),
\end{equation*}
hence a morphism $\varphi \colon \cF \to \cF'$ can be represented by a morphism of complexes $\varphi_\bullet \colon \cL_\bullet \to \cL'_\bullet$
if and only if it comes from~$\bE_\infty^{0,0}$.
In particular, this holds true for~$\ell < n$ since in this case~$\bE_\infty^{-n,n} = 0$.

Now assume $\ell = n$.
We have $\bE_\infty^{-n,n} \subset \Ext^n(\cL_0,\cL'_{n})$ and it is easy to see that 
if $\partial\varphi \in \Ext^n(\cL_0,\cL'_n)$ is the image of~$\varphi$ under the composition
\begin{equation*}
\Hom(\cF,\cF') \to \bE_\infty^{-n,n} \to \Ext^n(\cL_0,\cL'_{n})
\end{equation*}
then the composition
\begin{equation*}
\HH^0(\cL_0) \twoheadrightarrow \HH^\bt(\cL_\bullet) \to \HH^\tp(\cL'_\bullet) \hookrightarrow \HH^n(\cL'_n)
\end{equation*}
(where the middle arrow is~\eqref{eq:bt-tp-composition}) is given by~$\partial\varphi$.
Thus, if~\eqref{eq:bt-tp-composition} vanishes then~$\partial\varphi = 0$, 
and it follows that~$\varphi$ is in the image of $\bE_\infty^{0,0}$, hence is induced by a morphism of complexes.

Conversely, if~$\varphi$ is given by a morphism of complexes~$\varphi_\bullet$, the commutative diagram
\begin{equation*}
\xymatrix{
0 \ar[r] & 
\HH^\bt(\cL_\bullet) \ar[r] \ar[d]_{\HH^0(\varphi_0)} &
\HH^0(\cF) \ar[r] \ar[d]_{\HH^0(\varphi)} &
\HH^\tp(\cL_\bullet) \ar[d]_{\HH^n(\varphi_n)} \ar[r] &
0
\\
0 \ar[r] & 
\HH^\bt(\cL'_\bullet) \ar[r] &
\HH^0(\cF') \ar[r] &
\HH^\tp(\cL'_\bullet) \ar[r] &
0,
}
\end{equation*}
where the rows are the exact sequences of Lemma~\ref{lemma:split-n}, shows that~\eqref{eq:bt-tp-composition} is zero.

The uniqueness up to homotopy of~$\varphi_\bullet$ in both cases follows from the above formula for~$\bE_\infty^{0,0}$.
\end{proof}

\subsection{VHC resolutions and uniqueness}

The notion of a VHC resolution is based on the following

\begin{definition}
\label{def:lacm-uacm}
We will say that a vector bundle~$\cE$ on~$\P^n$ has
\begin{itemize}
\item 
the {\sf vanishing lower cohomology} property if
\begin{equation*}
\HH^p(\cE) = 0
\qquad\text{for $1 \le p \le \lfloor n/2 \rfloor$;}\hphantom{{} - n}
\end{equation*}
\item 
the {\sf vanishing upper cohomology} property if
\begin{equation*}
\HH^p(\cE) = 0
\qquad\text{for $\lceil n/2 \rceil \le p \le n - 1$}.
\end{equation*}
\end{itemize}
We will abbreviate these properties to {\sf VLC} and {\sf VUC}, respectively.
\end{definition}

\begin{example}
Every split bundle is both VLC and VUC. 
Moreover,
\begin{itemize}
\item 
every vector bundle on~$\P^1$ is both VLC and VUC since the conditions are void;
\item 
a vector bundle on~$\P^2$ is VLC if and only if it is VUC if and only if it is split;
\item 
if $1 \le p,q \le n - 1$ and $t \in \ZZ$ we have 
\begin{equation}
\label{eq:h-q-omega-p}
H^q(\P^n, \Omega^p(t)) = 
\begin{cases}
\kk, & \text{if $q = p$ and $t = 0$},\\
0, & \text{otherwise}.
\end{cases}
\end{equation}
Thus, $\Omega^p(t)$ is VLC if and only if $p > \lfloor n/2 \rfloor$
and it is VUC if and only if $p < \lceil n/2 \rceil$.
\end{itemize}
Note that for even $n$ the bundle $\Omega^{n/2}(t)$ is neither VLC nor VUC.
\end{example}

\begin{lemma}
\label{lemma:hacm-duality}
The properties VLC and VUC are invariant under twists, direct sums, and passing to direct summands.
Moreover, a vector bundle~$\cE$ is VLC if and only if~$\cE^\vee$ is VUC.
Finally, if a bundle~$\cE$ is VLC and VUC at the same time, it is split.
\end{lemma}
\begin{proof}
Follows from the definition, Serre duality, and Horrock's Theorem.
\end{proof}

Below we give a characterization of VLC and VUC bundles in terms of resolutions by split bundles.

\begin{lemma}
\label{lemma:hacm-resolutions}
A vector bundle $\cE$ on $\P^n$ is VLC if and only if there is an exact sequence
\begin{equation*}
0 \to \cL_{\lfloor (n-1)/2 \rfloor} \to \dots \to \cL_0 \to \cE \to 0,
\end{equation*}
where $\cL_i$ are split bundles.

A vector bundle $\cE$ on $\P^n$ is VUC if and only if there is an exact sequence
\begin{equation*}
0 \to \cE \to \cL_0 \to \dots \to \cL_{\lfloor (n-1)/2 \rfloor} \to 0,
\end{equation*}
where $\cL_i$ are split bundles.
\end{lemma}

\begin{proof}
First assume that $\cE$ is a VLC vector bundle.
\begin{itemize}
\item 
If~\mbox{$n = 2k$} then $\HH^p(\cE) = 0$ for $1 \le p \le k = n - k$;
by Lemma~\ref{lem:split-resolution} this is equivalent to the existence 
of a resolution of length~$\ell = k - 1 = \lfloor (n-1)/2 \rfloor$ by split bundles;
\item
Similarly, if~\mbox{$n = 2k + 1$} then~$\HH^p(\cE) = 0$ for $1 \le p \le k = n - (k + 1)$;
by Lemma~\ref{lem:split-resolution} this is equivalent to the existence 
of a resolution of length~$\ell = k = \lfloor (n-1)/2 \rfloor$ by split bundles.
\end{itemize}
The case of a VUC bundle follows from this and Lemma~\ref{lemma:hacm-duality} by duality.
\end{proof}

\begin{definition}
\label{def:hacm}
We will say that a locally free resolution $0 \to \cE_\la \to \cE_\ua \to \cF \to 0$ of a sheaf $\cF$ 
has the~{\sf VHC} ({\sf vanishing of half cohomology}) property (or simply is a {\sf VHC resolution}),
if $\cE_\la$ is a VLC vector bundle and $\cE_\ua$ is a VUC vector bundle, see Definition~\ref{def:lacm-uacm}.
\end{definition}

The cohomology of bundles constituting a VHC resolution of a sheaf~$\cF$ are related to the cohomology of~$\cF$ as follows.

\begin{lemma}
\label{lemma:hhk-ce+}
Let $0 \to \cE_\la \to \cE_\ua \to \cF \to 0$ be a VHC resolution on~$\P^n$ and assume $1 \le p \le n-1$.
If~\mbox{$n = 2k$} then
\begin{align*}
\HH^p(\cE_\la) &= 
\begin{cases}
0, & \text{if $1 \le p \le k$,} \\
\HH^{p-1}(\cF), & \text{if $k + 1 \le p \le n - 1$,}
\end{cases}
&
\HH^p(\cE_\ua) &= 
\begin{cases}
\HH^p(\cF), & \text{if $1 \le p \le k - 1$,} \\
0, & \text{if $k \le p \le n - 1$.}
\end{cases}
\intertext{If $n = 2k + 1$ then}
\HH^p(\cE_\la) &= 
\begin{cases}
0, & \text{if $1 \le p \le k$,} \\
\HH^{p-1}(\cF), & \text{if $k + 2 \le p \le n - 1$,}
\end{cases}
&
\HH^p(\cE_\ua) &= 
\begin{cases}
\HH^p(\cF), & \text{if $1 \le p \le k - 1$,} \\
0, & \text{if $k + 1 \le p \le n - 1$,}
\end{cases}
\end{align*}
while $\HH^{k+1}(\cE_\la)$ and $\HH^k(\cE_\ua)$ fit into an exact sequence of graded~$\SS$-modules
\begin{equation}
\label{eq:hhk-ce-pm}
0 \to \HH^k(\cE_\ua) \to \HH^k(\cF) \to \HH^{k+1}(\cE_\la) \to 0.
\end{equation}
\end{lemma}
\begin{proof}
Follows immediately from the long exact sequences of cohomology groups and the vanishings in the definition of VLC and VUC bundles.
\end{proof}

If $n = 2k + 1$ we will often use sequence~\eqref{eq:hhk-ce-pm} to identify~$\HH^k(\cE_\ua)$ with an~$\SS$-submodule of~$\HH^k(\cF)$.

\begin{lemma}
\label{lemma:hacm-split}
Let $0 \to \cE_\la \to \cE_\ua \to \cF \to 0$ be a VHC resolution on~$\P^n$.
Set~$k = \lfloor (n-1)/2 \rfloor$.
Then the object~$\cF[k] \in \Db(\P^n)$ is quasiisomorphic to a complex of split bundles
\begin{equation*}
\{\cL_{2k + 1} \to \cL_{2k} \to \dots \to \cL_1 \to \cL_0\}
\end{equation*}
such that its first half $\{\cL_{2k + 1} \to \dots \to \cL_{k+1}\}$ is a resolution of~$\cE_\la$ 
and its second half $\{\cL_{k} \to \dots \to \cL_{0}\}$ is a resolution of~$\cE_\ua$.

Moreover, if $n = 2k + 1$ then $\HH^k(\cE_\ua) = \HH^\bt(\cL_\bullet)$, $\HH^{k+1}(\cE_\la) = \HH^\tp(\cL_\bullet)$, 
and the exact sequence~\eqref{eq:hhk-ce-pm}
coincides with the exact sequence of Lemma~\textup{\ref{lemma:split-n}}.
\end{lemma}
\begin{proof}
By Lemma~\ref{lemma:hacm-resolutions} the sheaves $\cE_\la$ and~$\cE_\ua$ have resolutions of length~$k$ by split bundles, 
which we can write in the form
\begin{equation}
\label{eq:cem-cep-resolutions}
0 \to \cL_{2k+1} \to \dots \to \cL_{k+1} \to \cE_\la \to 0
\qquad\text{and}\qquad 
0 \to \cE_\ua \to \cL_k \to \dots \to \cL_0 \to 0.
\end{equation}
The morphism $\cE_\la \to \cE_\ua$ gives a morphism $\cL_{k+1} \to \cL_k$, 
which allows us to concatenate the resolutions into a single complex
\begin{equation*}
\{\cL_{2k + 1} \to \cL_{2k} \to \dots \to \cL_1 \to \cL_0\}
\end{equation*}
of split bundles quasiisomorphic to $\Cone(\cE_\la \to \cE_\ua)[k] \cong \cF[k]$.
If $n = 2k + 1$ the hypercohomology spectral sequences of~\eqref{eq:cem-cep-resolutions} show that
$\HH^k(\cE_\ua) = \HH^\bt(\cL_\bullet)$ and $\HH^{k+1}(\cE_\la) = \HH^\tp(\cL_\bullet)$,
and allow us to identify the exact sequences of Lemma~\textup{\ref{lemma:hhk-ce+}} and Lemma~\textup{\ref{lemma:split-n}}.
\end{proof}

For the uniqueness result stated below we need the following technical notion.

\begin{definition}
\label{def:linearly-minimal}
A VHC resolution is {\sf linearly minimal} if it has no trivial complex $\cO(t) \xrightarrow{\ \id\ } \cO(t)$ as a direct summand.
In other words, if $f \colon \cE_\la \to \cE_\ua$ is not isomorphic to $\id \oplus f' \colon \cO(t) \oplus \cE'_\la \to \cO(t) \oplus \cE'_\ua$.
\end{definition}

Clearly, any VHC resolution is isomorphic to the direct sum of a linearly minimal VHC resolution 
and several trivial complexes $\cO(t_i) \xrightarrow{\ \id\ } \cO(t_i)$.

\begin{theorem}
\label{prop:hacm-uniqueness}
Let $0 \to \cE_\la \xrightarrow{\ f\ } \cE_\ua \to \cF \to 0$ and $0 \to \cE'_\la \xrightarrow{\ f'\ } \cE'_\ua \to \cF \to 0$ 
be linearly minimal VHC resolutions of the same sheaf\/~$\cF$.
If $n = 2k + 1$ assume also we have an equality $\HH^k(\cE_\ua) = \HH^k(\cE'_\ua)$ of~$\SS$-submodules in~$\HH^k(\cF)$
with respect to the embeddings given by~\eqref{eq:hhk-ce-pm}.
Then the resolutions are isomorphic, i.e., there is a commutative diagram
\begin{equation*}
\xymatrix{
0 \ar[r] &
\cE_\la \ar[r]^f \ar[d]_{\varphi_\la} &
\cE_\ua \ar[r] \ar[d]^{\varphi_\ua} \ar@{..>}[dl]_h &
\cF \ar[r] \ar[d]^{\id} &
0
\\
0 \ar[r] &
\cE'_\la \ar[r]^{f'} &
\cE'_\ua \ar[r] &
\cF \ar[r] &
0,
}
\end{equation*}
where~$\varphi_\la$ and~$\varphi_\ua$ are isomorphisms.
Moreover, an isomorphism~$(\varphi_\la, \varphi_\ua)$ of resolutions 
inducing the identity morphism of~$\cF$ is unique up to a homotpy~$h \colon \cE_\ua \to \cE'_\la$.
Finally, the endomorphisms~$\varphi_\la^{-1} \circ h \circ f$ and~$\varphi_\ua^{-1} \circ f' \circ h$ 
of~$\cE_\la$ and~$\cE_\ua$ induced by any homotopy~$h$ are nilpotent.
\end{theorem}

\begin{proof}
Let~\mbox{$k = \lfloor (n-1)/2 \rfloor$}, so that~\mbox{$n = 2k + 1$} or~\mbox{$n = 2k + 2$}.
By Lemma~\ref{lemma:hacm-split} the object~$\cF[k]$ is quasiisomorphic to complexes of split bundles
\begin{equation*}
\{\cL_{2k + 1} \to \cL_{2k} \to \dots \to \cL_1 \to \cL_0\}
\qquad\text{and}\qquad
\{\cL'_{2k + 1} \to \cL'_{2k} \to \dots \to \cL'_1 \to \cL'_0\}
\end{equation*}
corresponding to the resolutions $\cE_\la \to \cE_\ua$ and $\cE'_\la \to \cE'_\ua$, respectively.
Using linear minimality we can assume that each of these complexes 
has no trivial complex $\cO(t) \xrightarrow{\ \id\ } \cO(t)$ as a direct summand.

If $n = 2k + 2$, the lengths of the resolutions are less than $n$, hence the first part of Lemma~\ref{lemma:morphism-of-complexes}
ensures that the identity morphism $\cF \to \cF$ is induced by a morphism of complexes.
If $n = 2k + 1$ we use the second part of Lemma~\ref{lemma:morphism-of-complexes}
(the composition~\eqref{eq:bt-tp-composition} vanishes due to the assumption $\HH^k(\cE_\ua) = \HH^k(\cE'_\ua)$ and Lemma~\ref{lemma:hacm-split}) 
and obtain the same conclusion.
Thus, we obtain a quasiisomorphism of complexes of split bundles
\begin{equation}
\label{eq:quasiisomorphism}
\vcenter{\xymatrix{
\cL_{2k+1} \ar[r] \ar[d]_{\varphi_{2k+1}} &
\cL_{2k} \ar[r] \ar[d]_{\varphi_{2k}} &
\dots \ar[r] &
\cL_1 \ar[r] \ar[d]_{\varphi_1} &
\cL_0 \ar[d]_{\varphi_0} 
\\
\cL'_{2k+1} \ar[r] &
\cL'_{2k-1} \ar[r] &
\dots \ar[r] &
\cL'_1 \ar[r] &
\cL'_0.
}}
\end{equation}
We prove below that it is necessarily an isomorphism, i.e., that each~$\varphi_i$ is an isomorphism.
For this we use the induction on the sum of ranks of the bundles~$\cL_i$.

The base of the induction follows from Lemma~\ref{lemma:split-n+1}.
Indeed, if $\cL_\bullet = 0$ then $\cL'_\bullet$ is acyclic, hence is the sum of trivial complexes.
But by assumption it has no trivial summands, hence $\cL'_\bullet = 0$.

Now assume that $\cL_\bullet \ne 0$.
The totalization of~\eqref{eq:quasiisomorphism} is the following acyclic complex
\begin{equation}
\label{eq:totalization}
\cL_{2k+1} \to \cL_{2k} \oplus \cL'_{2k+1} \to \dots \to \cL_0 \oplus \cL'_1 \to \cL'_0 
\end{equation}
of split bundles of length $2k + 2$.
If $n = 2k + 2$ we can formally add the zero term on the right and obtain an acyclic complex of length~$\ell = n + 1$ of split bundles
for which the condition~\eqref{cond:h-bt} of Lemma~\ref{lemma:split-n+1} holds true.
If~\mbox{$n = 2k + 1$} the condition~\eqref{cond:h-bt} of Lemma~\ref{lemma:split-n+1} 
follows from the assumption $\HH^k(\cE_\ua) = \HH^k(\cE'_\ua)$.
In both cases Lemma~\ref{lemma:split-n+1} implies that~\eqref{eq:totalization} is isomorphic to a direct sum of shifts of trivial complexes. 

To make this direct sum decomposition more precise, we consider
as in the proof of Lemma~\ref{lemma:split-n+1} the maximal integer $t$
such that $\cO(t)$ appears as one of summands of one of the split bundles~$\cL_i$ or~$\cL'_i$.
Twisting~\eqref{eq:quasiisomorphism} by $\cO(-t)$ and applying the functor~$H^0(\P^n,-)$ 
we obtain a nonzero bicomplex
\begin{equation*}
\xymatrix{
\kk^{m_{2k+1}} \ar[r] \ar[d] & 
\kk^{m_{2k}} \ar[r] \ar[d] & 
\dots \ar[r] & 
\kk^{m_1} \ar[r] \ar[d] & 
\kk^{m_0} \ar[d]
\\
\kk^{m'_{2k+1}} \ar[r] & 
\kk^{m'_{2k}} \ar[r] & 
\dots \ar[r] & 
\kk^{m'_1} \ar[r] & 
\kk^{m'_0}
}
\end{equation*}
(as before, $m_i$ and $m'_i$ are the multiplicities of $\cO(t)$ in~$\cL_i$ and $\cL'_i$, respectively)
with acyclic totalization.
If any of the horizontal arrows in this bicomplex is nontrivial, 
Lemma~\ref{lemma:splitting-complex} implies that the trivial complex~\mbox{$\cO(t) \to \cO(t)$}
is a direct summand of either~$\cL_\bullet$ or~$\cL'_\bullet$, which contradicts the linear minimality assumption.
Therefore, the horizontal arrows are zero, and hence the vertical arrows are all isomorphisms.

This means that $m_i = m'_i$ for all $i$ and we can write 
\begin{equation*}
\cL_i = \cO(t)^{\oplus m_i} \oplus \bar\cL_i,
\qquad 
\cL'_i = \cO(t)^{\oplus m_i} \oplus \bar\cL'_i,
\qquad 
\varphi_i = \left(\begin{smallmatrix} 1 & \psi_i \\ 0 & \bar\varphi_i \end{smallmatrix}\right),
\end{equation*}
and that $\bar\varphi_\bullet \colon \bar\cL_\bullet \to \bar\cL'_\bullet$ 
is a quasiisomorphism of complexes of split bundles which have no trivial summands.
Moreover, we have $\sum\rk(\bar\cL_i) < \sum\rk(\cL_i)$.
By induction, we deduce that~$\bar\varphi_i$ is an isomorphism for each~$i$,
hence so is~$\varphi_i$.

Since $\varphi_\bullet$ is an isomorphism of complexes, it induces an isomorphism of resolutions 
of $\cE_\la$ and $\cE'_\la$ and of $\cE_\ua$ and $\cE'_\ua$, compatible with the maps $\cE_\la \to \cE_\ua$ and $\cE'_\la \to \cE'_\ua$,
hence an isomorphism~$(\varphi_\la,\varphi_\ua)$ of the original VHC resolutions.
This proves the first part of the theorem.

Further, recall that by Lemma~\ref{lemma:morphism-of-complexes} the morphism~$\varphi_\bullet$ in~\eqref{eq:quasiisomorphism} 
inducing the identity of~$\cF$ is unique up to a homotopy $h_\bullet \colon \cL_\bullet \to \cL'_{\bullet + 1}$.
Note that first part~$(h_i)_{0 \le i \le k-1}$ of such a homotopy 
replaces the morphism~$(\varphi_i)_{0 \le i \le k}$ 
of the right resolutions of~$\cE_\ua$ and~$\cE'_\ua$ by a homotopy equivalent morphism,
hence it does not change~$\varphi_\ua$, and a fortiori does not change~$\varphi_\la$.
Similarly, the last part~$(h_i)_{k+1 \le i \le 2k}$ of a homotopy does not change~$(\varphi_\la,\varphi_\ua)$.
Finally, it is clear that the middle component~$h_k \colon \cL_k \to \cL'_{k+1}$ of a homotopy 
modifies~$(\varphi_\la,\varphi_\ua)$ by the homotopy
\begin{equation*}
\cE_\ua \hookrightarrow \cL_k \xrightarrow{\ h_k\ } \cL'_{k+1} \twoheadrightarrow \cE'_\la
\end{equation*}
of the VHC resolutions.
This proves the second part of the theorem.

So, it only remains to check the nilpotence of the induced endomorphisms of~$\cE_\la$ and~$\cE_\ua$.
For this let us write
\begin{align*}
\cL_{k+1} &= \moplus \cO(a_i),
&
\cL_k &= \moplus \cO(b_i),
&
\cL'_{k+1} &= \moplus \cO(a'_i),
&
\cL'_k &= \moplus \cO(b'_i)
\\
\intertext{and for each $c \in \ZZ$ define finite filtrations of these bundles by}
\rF_{\ge c}\cL_{k+1} &= \moplus_{a_i \ge c} \cO(a_i),
&
\rF_{\ge c}\cL_k &= \moplus_{b_i \ge c} \cO(b_i),
&
\rF_{\ge c}\cL'_{k+1} &= \moplus_{a'_i \ge c} \cO(a'_i),
&
\rF_{\ge c}\cL'_k &= \moplus_{b'_i \ge c} \cO(b'_i).
\end{align*}
Then the morphism~$\cL_{k+1} \to \cL_k$ induced by~$f$ takes~$\rF_{\ge c}\cL_{k+1}$ to~$\rF_{\ge c + 1}\cL_k$
(because~$f$ is assumed to be linearly minimal)
and obviously any morphism~$h \colon \cL_k \to \cL'_{k+1}$ 
takes~$\rF_{\ge c + 1}\cL_{k}$ to~$\rF_{\ge c + 1}\cL'_{k + 1}$.
Since~$\varphi_\la$ is an isomorphism, we conclude that the composition~$\varphi_\la^{-1} \circ h \circ f$ is induced by an endomorphism of~$\cL_{k+1}$
that takes~$\rF_{\ge c}\cL_{k+1}$ to~$\rF_{\ge c + 1}\cL_{k+1}$, hence is nilpotent. 
A similar argument works for~$\varphi_\ua^{-1} \circ f' \circ h$. 
\end{proof}

\subsection{Existence of VHC resolutions}
\label{subsec:vhc-existence}

The results of this subsection are not necessary for~\S\ref{sec:hacm-pn},
but the technique used in their proofs is similar.

\begin{definition}
\label{def:shadow}
Let~$\cF$ be a coherent sheaf and~$1 \le k \le n - 1$.
We will say that a graded~$\SS$-submodule~$\rA^k \subset \HH^k(\cF)$ is {\sf shadowless}
if for any~$t_0 \in \ZZ$ such that~$\rA^{k}_{t_0} \ne 0$ we have~$\rA^{k}_{t} = H^{k}(\P^n, \cF(t))$ for any~$t > t_0$.
Similarly, for any~$1 \le p_0 \le n-1$ and any~$t_0 \in \ZZ$ we define the {\sf shadow} of~$(p_0,t_0)$ as the set 
\begin{equation}
\label{eq:shadow}
\Sha(p_0,t_0) = \{ (p,t) \mid 1 \le p \le p_0 \text{ and } t > t_0 \}
\end{equation}
and say that a bigraded $\SS$-submodule $\rA \subset \moplus_{p=1}^{n-1} \HH^p(\cF)$ is {\sf shadowless}
if for any~$(p_0,t_0)$ such that~$\rA^{p_0}_{t_0} \ne 0$ we have~$\rA^{p}_{t} = H^{p}(\P^n, \cF(t))$ for any~$(p,t) \in \Sha(p_0,t_0)$.
\end{definition}

To understand the meaning of this notion observe the following.
Let $\cT$ be the tangent bundle of~$\P^n$.
Recall the Koszul resolution of its exterior power
\begin{equation}
\label{eq:koszul-t}
0 \to \cO \to V \otimes \cO(1) \to \dots \to \wedge^sV \otimes \cO(s) \to \wedge^s\cT \to 0,
\end{equation} 
where $V$ is a vector space such that $\P^n = \P(V)$.
If $\cF$ is a sheaf on~$\P^n$, tensoring~\eqref{eq:koszul-t} by~$\cF(t)$ 
we obtain the hypercohomology spectral sequence
\begin{equation*}
\bE_1^{i,j} = \wedge^{i}V \otimes H^j(\P^n,\cF(i+t)) \Rightarrow H^{i+j-s}(\P^n, \wedge^s\cT \otimes \cF(t)).
\end{equation*}
The following picture shows the arrows~$\bd_r$, $1 \le r \le p$, 
of the spectral sequence with source at the terms~$\bE_r^{t,p}$,
as well as the terms that in the limit compute the filtration on~$H^{p-s}(\P^n,\wedge^{s}\cT \otimes \cF(t))$ 
(these terms are circled), and the shadow of~$(p,t)$.
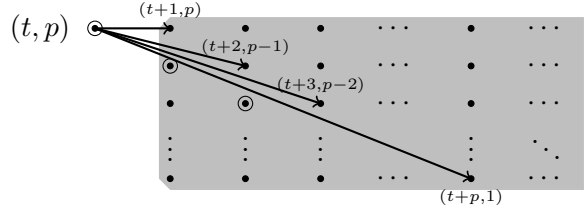
\begin{equation*}
\begin{tikzpicture}[xscale = .5, yscale = .5]
\fill[color = lightgray] (2,-0.3) -- (1.7,0) -- (1.7,4) -- (2,4.3) -- (13,4.3) -- (13,-0.3) -- (2,-0.3) -- (1.7,0);
\filldraw[black] (0,4) circle (.2em) node[left=.5em]{$(t,p)$};
\foreach \i in {1,2,3,5} \filldraw[black] (2*\i,4) circle (.2em);
\foreach \i in {1,2,3,5} \filldraw[black] (2*\i,3) circle (.2em);
\foreach \i in {1,2,3,5} \filldraw[black] (2*\i,2) circle (.2em);
\foreach \i in {1,2,3,5} \filldraw[black] (2*\i,0) circle (.2em);
\foreach \i in {0,1,2} \draw (2*\i,4-\i) circle (.5em);
\path{} (2,1) node {$\vdots$} (4,1) node {$\vdots$} (6,1) node {$\vdots$} (10,1) node {$\vdots$}
	(8,2) node {$\cdots$} (8,3) node {$\cdots$} (8,4) node {$\cdots$} (8,0) node {$\cdots$}
	(12,2) node {$\cdots$} (12,3) node {$\cdots$} (12,4) node {$\cdots$} (12,0) node {$\cdots$} (12,1) node {$\ddots$};
\draw[->,thick] (0,4) -- (2,4) node[above] {$\scriptscriptstyle(t+1,p)$};
\draw[->,thick] (0,4) -- (4,3) node[above] {$\scriptscriptstyle(t+2,p-1)$};
\draw[->,thick] (0,4) -- (6,2) node[above] {$\scriptscriptstyle(t+3,p-2)$};
\draw[->,thick] (0,4) -- (10,0) node[below] {$\scriptscriptstyle(t + p,1)$};
\end{tikzpicture}
\end{equation*}
It is important that the arrows~$\bd_r$, $1 \le r \le p$, 
applied to the terms~$\bE_r^{t,p}$ of the spectral sequence land in its shadow.
This property will be used in Propositions~\ref{prop:ce-a} and~\ref{prop:eps-p} below.

\begin{proposition}
\label{prop:ce-a}
For any coherent sheaf~$\cF$ on~$\P^n$ and any finite-dimensional shadowless $\SS$-submodule 
\begin{equation*}
\rA \subset \moplus_{p=1}^{n-1} \HH^p(\cF)
\end{equation*}
there exists a vector bundle $\cE_\rA$
and an epimorphism $\pi_\rA \colon \cE_\rA \twoheadrightarrow \cF$
such that 
\begin{itemize}
\item 
the map $\HH^0(\cE_A) \xrightarrow{\ \HH^0(\pi_\rA)\ } \HH^0(\cF)$ is surjective, and
\item 
the map $\moplus_{p=1}^{n-1} \HH^p(\cE_A) \xrightarrow{\ \oplus \HH^p(\pi_\rA)\ } \moplus_{p=1}^{n-1} \HH^p(\cF)$ is an isomorphism onto~$\rA$.
\end{itemize}
\end{proposition}

Note that the assumption~$\dim(\rA) < \infty$ in the proposition is necessary 
because~$\HH^p(\cE)$ is finite dimensional for any vector bundle~$\cE$ if~$1 \le p \le n - 1$.
On the other hand, if~$\cF$ is a coherent sheaf with $p$-dimensional support then~$\HH^p(\cF)$ is not finite-dimensional,
so a priori~$\rA$ could have infinite dimension.

\begin{proof}
We argue by induction on $\dim(\rA)$.
If $\rA = 0$ we take $\cE_\rA$ to be the split bundle that corresponds to a free $\SS$-module surjecting onto~$\HH^0(\cF)$
as in Lemma~\ref{lemma:epi-from-split}.
The desired condition is tautologically true.

Assume $\dim(\rA) > 0$.
Let 
\begin{equation*}
p_0 = \min\{p \ge 1 \mid \rA^p \ne 0\}
\qquad\text{and}\qquad
t_0 = \max\{t \mid \rA^{p_0}_{t} \ne 0\}.
\end{equation*}
Since $\rA$ is shadowless we have~$H^p(\P^n,\cF(t)) = \rA^p_t = 0$ for all~$(p,t) \in \Sha(p_0,t_0)$
(the first equality holds because~$\rA$ is shadowless and the second follows from the above definition of~$(p_0,t_0)$).
In particular, the subspace~$\rA^{p_0}_{t_0} \subset H^{p_0}(\P^n,\cF(t_0))$ 
sits in the kernels of differentials~$\bd_1, \dots, \bd_{p_0-1}$
of the hypercohomology spectral sequence of~$\cF(t_0)$ tensored with the Koszul complex~\eqref{eq:koszul-t} for~$s = p_0 - 1$.
Moreover, $H^{p_0}(\P^n,\cF(t_0))$ is the only nonzero subspace on the diagonal of the spectral sequence 
that in the limit computes the filtration on~$H^1(\P^n,\wedge^{p_0-1}\cT \otimes \cF(t_0))$.
Therefore, we obtain an inclusion
\begin{equation*}
\rA^{p_0}_{t_0} \subset 
H^{p_0}(\P^n,\cF(t_0)) = 
H^1(\P^n,\wedge^{p_0-1}\cT \otimes \cF(t_0)) = 
\Ext^1(\Omega^{p_0 - 1}(-t_0), \cF)
\end{equation*}
which induces an extension
\begin{equation*}
0 \to \cF \to \cF' \to \rA^{p_0}_{t_0} \otimes \Omega^{p_0 - 1}(-t_0) \to 0
\end{equation*}
such that the connecting morphism 
$\rA^{p_0}_{t_0} = H^{p_0 - 1}(\P^n, \rA^{p_0}_{t_0} \otimes \Omega^{p_0 - 1}) \to H^{p_0}(\P^n, \cF(t_0))$ 
is the natural embedding (the first identification uses~\eqref{eq:h-q-omega-p}).
Now the cohomology exact sequence implies that
\begin{equation*}
\moplus_{p=1}^{n-1} \HH^{p}(\cF') = \left( \moplus_{p=1}^{n-1} \HH^{p}(\cF) \right) \Big/\rA^{p_0}_{t_0},
\end{equation*}
hence the quotient $\SS$-module $\rA' := \rA/\rA^{p_0}_{t_0}$ is an $\SS$-submodule in $\moplus \HH^{p}(\cF')$.
Clearly~$\dim(\rA') < \dim(\rA)$ and it is straightforward to check that~$\rA'$ is shadowless.
By the induction hypothesis there is a vector bundle~$\cE_{\rA'}$ and an epimorphism~$\pi_{\rA'} \colon \cE_{\rA'} \twoheadrightarrow \cF'$
inducing surjection on~$\HH^0$ and the natural embedding of~$\rA'$ into the intermediate cohomology of~$\cF'$.
We define~$\cE_\rA$ as the kernel of the composition of epimorphisms
\begin{equation*}
\cE_{\rA'} \twoheadrightarrow \cF' \twoheadrightarrow \rA^{p_0}_{t_0} \otimes \Omega^{p_0 - 1}(-t_0).
\end{equation*}
By construction the map $\pi_{\rA'}$ lifts to a map $\pi_\rA$ that fits into a commutative diagram
\begin{equation*}
\xymatrix{
0 \ar[r] &
\cE_\rA \ar[r] \ar[d]_{\pi_\rA} & 
\cE_{\rA'} \ar@{->>}[d]_{\pi_{\rA'}} \ar[r] &
\rA^{p_0}_{t_0} \otimes \Omega^{p_0 - 1}(-t_0) \ar[r] \ar@{=}[d] & 
0
\\
0 \ar[r] &
\cF \ar[r] & 
\cF' \ar[r] &
\rA^{p_0}_{t_0} \otimes \Omega^{p_0 - 1}(-t_0) \ar[r] & 
0
}
\end{equation*}
The surjectivity of~$\pi_{\rA'}$ and~$\HH^0(\pi_{\rA'})$ implies that of~$\pi_\rA$ and~$\HH^0(\pi_{\rA})$.
Similarly, it follows that~$\moplus \HH^p(\pi_{\rA})$ is an isomorphism onto~$\rA$.
Thus, the required result holds for~$\rA$.
\end{proof}

\begin{corollary}
\label{cor:hacm-resolution}
If $\cF$ is a sheaf on~$\P^n$ of projective dimension at most~$1$
then~$\cF$ has a VHC resolution. 
\end{corollary}

\begin{proof}
Since the projective dimension of $\cF$ is at most~1, there exists a locally free resolution
\begin{equation*}
0 \to \cE_1 \to \cE_0 \to \cF \to 0.
\end{equation*}
Since $\HH^p(\cE_i)$ is finite dimensional for $1 \le p \le n-1$ the  cohomology exact sequence 
\begin{equation*}
\dots \to \HH^p(\cE_0) \to \HH^p(\cF) \to \HH^{p+1}(\cE_1) \to \dots
\end{equation*}
implies that $\HH^p(\cF)$ is finite dimensional for $1 \le p \le n-2$.
Let 
\begin{equation*}
\label{eq:ra-typical}
\rA := 
\begin{cases}
\moplus_{p=1}^{k-1} \HH^p(\cF), & \text{if $n = 2k$},\\
\moplus_{p=1}^{k-1} \HH^p(\cF) \oplus \rA^k, & \text{if $n = 2k + 1$},
\end{cases}
\end{equation*}
where $\rA^k \subset \HH^k(\cF)$ is any finite-dimensional shadowless $\SS$-submodule (e.g., $\rA^k = 0$).
Note that we have~$k - 1 \le n - 2$ as soon as~$n \ge 1$, hence~$\rA$ is finite-dimensional.
Moreover, $\rA$ is shadowless by construction.

Let $\pi_\rA \colon \cE_\rA \to \cF$ be the epimorphism constructed in Proposition~\ref{prop:ce-a}
and let $\cK_\rA = \Ker(\pi_\rA)$, so that
\begin{equation*}
0 \to \cK_\rA \to \cE_\rA \to \cF \to 0
\end{equation*}
is an exact sequence.
First, $\HH^p(\cE_\rA) = \rA^p = 0$ for $\lceil n/2\rceil \le p \le n-1$ by definition, hence $\cE_\rA$ is~VUC.
Furthermore, $\cK_\rA$ is locally free because the projective dimension of~$\cF$ is at most~1.
Finally, the cohomology exact sequence implies that $\HH^p(\cK_\rA) = 0$ for~$1 \le p \le k = \lfloor n/2 \rfloor$, hence~$\cK_\rA$ is VLC. 
\end{proof}

\section{Hyperbolic equivalence on projective spaces}
\label{sec:hacm-pn}

In this section we prove Theorem~\ref{thm:he-hacm} on VHC modifications of quadratic forms
and deduce from it Theorem~\ref{thm:he-intro} and Corollary~\ref{corollary:invariants} from the Introduction.
In~\S\ref{subsec:sym-sheaves} we remind a characterization of cokernel sheaves of quadratic forms (\emph{symmetric sheaves}),
in~\S\ref{subsec:elementary-modifications} we define elementary modifications of quadratic forms 
with respect to some intermediate cohomology classes, and
in~\S\ref{subsec:hacm-modifications} we state the Modification Theorem (Theorem~\ref{thm:he-hacm})
and prove it by applying an appropriate sequence of elementary modifications.
Finally, in~\S\ref{subsec:proofs} we combine these results to prove Theorem~\ref{thm:he-intro} and Corollary~\ref{corollary:invariants}.

\subsection{Reminder on symmetric sheaves}
\label{subsec:sym-sheaves}

For a scheme $Y$ and an object $\cC \in \Db(Y)$ we write 
\begin{equation*}
\cC^\vee := \cRHom(\cC,\cO_Y)
\end{equation*}
for the derived dual of~$\cC$.
Note that the cohomology sheaves~$\cH^i(\cC^\vee)$ of~$\cC^\vee$ are isomorphic to the local $\Ext$-sheaves~$\cExt^i(\cC,\cO_Y)$.

\begin{definition}[{cf.~\cite[Definition~0.2]{CC97}}]
\label{def:sym-sheaf}
We say that a coherent sheaf~$\cC$ on~$\P^n$ is {\sf $(d,\delta)$-symmetric}, if~\mbox{$\cC \cong i_*\cR$}, 
where $i \colon D \hookrightarrow \P^n$ is the embedding of a degree~$d$ hypersurface
and $\cR$ is a coherent sheaf on~$D$ endowed with a symmetric morphism 
\begin{equation*}
\cR \otimes \cR \to \cO_D(\delta)
\end{equation*}
such that the induced morphism $\cR(-\delta) \to \cR^\vee$ (where the duality is applied on~$D$) is an isomorphism.
\end{definition}

Note that $d$, $\delta$, $D$, and~$\cR$ in Definition~\ref{def:sym-sheaf} are \emph{not} determined by the sheaf~$\cC$,
see Remark~\ref{rem:non-unique-ss}.

The goal of this subsection is to relate symmetric sheaves to cokernel sheaves of quadratic forms.
Most of these results are well-known and not really necessary for the rest of the paper, but useful for the context.

\begin{lemma}
\label{lemma:cf-sd}
If $\cC$ is a $(d,\delta)$-symmetric coherent sheaf on~$\P^n$ there is a self-dual isomorphism
\begin{equation*}
\cC^\vee \cong \cC(m)[-1], 
\end{equation*}
where $m = d - \delta$.
In particular, the sheaf~$\cC$ has projective dimension~$1$.
\end{lemma}
\begin{proof}
Let~$\cC = i_*\cR$.
Using the definitions and Grothendieck duality we deduce
\begin{multline*}
\cC^\vee =
\cRHom(i_*\cR,\cO_{\P^n}) \cong
i_*\cRHom(\cR,i^!\cO_{\P^n}) \cong
i_*\cRHom(\cR,\cO_D(d)[-1]) \cong \\ \cong
i_*\cR^\vee(d)[-1] \cong
i_*\cR(d - \delta)[-1] = 
\cC(m)[-1].
\end{multline*}
This proves the required isomorphism.
Moreover, it follows that this isomorphism is self-dual because so is the isomorphism~$\cR(-\delta) \cong \cR^\vee$.
Finally, it follows that~$\cExt^1(\cC,\cO_{\P^n}) \cong \cC(m)$ and $\cExt^i(\cC,\cO_{\P^n}) = 0$ for~\mbox{$i \ge 2$},
which means that the projective dimension of~$\cC$ is~1.
\end{proof}

The above lemma implies that symmetric sheaves can be understood as quadratic spaces in the derived category~$\Db(\P^n)$
and define classes in the \emph{shifted Witt group} $\rW^1(\Db(\P^n), \cO(-m))$ in the sense of~\cite[\S1.4]{Balmer}.

The following well-known lemma shows that cokernel sheaves of generically non-degenerate quadratic forms are symmetric.
For reader's convenience we provide a proof.

\begin{lemma}
\label{lemma:cokernel-symmetric}
If a non-zero sheaf $\cC$ on~$\P^n$ has a self-dual locally free resolution 
\begin{equation}
\label{eq:ce-m-ce-vee-cf}
0 \to \cE(-m) \xrightarrow{\ q\ } \cE^\vee \to \cC \to 0 
\end{equation}
then~$\cC$ is a~$(d,\delta)$-symmetric sheaf, where $d = 2\rc_1(\cE^\vee) + m\rk(\cE)$ and $\delta = d - m$.
\end{lemma}
\begin{proof}
Applying the functor $\cRHom(-,\cO_{\P^n})$ to the resolution~\eqref{eq:ce-m-ce-vee-cf} of~$\cC$ we obtain a distinguished triangle
\begin{equation*}
\cC^\vee \to \cE \xrightarrow{\ q^\vee\ } \cE^\vee(m)
\end{equation*}
in~$\Db(\P^n)$.
Since $q$ is self-dual, we have $q = q^\vee$. 
In particular, it is generically an isomorphism, hence~$\cC^\vee[1]$ is a pure sheaf and, moreover, $\cC^\vee[1] \cong \cC(m)$.

Let $\det(q)$ be the determinant of~$q$ which we understand as a global section of the line bundle
\begin{equation*}
\det(\cE(-m))^\vee \otimes \det(\cE^\vee) \cong \det(\cE^\vee)^{\otimes 2} \otimes \cO(rm),
\end{equation*}
where $r$ is the rank of~$\cE$.
Let $D \subset \P^n$ be the zero locus of $\det(q)$ and set $d := \deg(D)$, so that the line bundle above is~$\cO(d)$.
Consider also the adjugate morphism $q' := \wedge^{r-1} q \colon \bw{r-1}(\cE(-m)) \to \bw{r-1}\cE^\vee$;
twisting it by~$\det(\cE) \otimes \cO(d - m)$ we obtain a morphism $\cE^\vee \xrightarrow{\ q'\ } \cE(d - m)$.
Note that
\begin{equation*}
q' \circ q = \det(q) \otimes \id_\cE
\qquad \text{and} \qquad
q \circ q' = \det(q) \otimes \id_{\cE^\vee}.
\end{equation*}
It follows that $\cC = \Coker(q)$ is supported on~$D$, i.e., $\cC \cong i_*\cR$, where $i \colon D \hookrightarrow \P^n$ is the embedding.

Inverting the computation of Lemma~\ref{lemma:cf-sd} and using the fact that the functor~$i_*$ is exact and fully faithful on coherent sheaves
we deduce that~$\cR^\vee \cong \cR(-\delta)$.
So, it remains to show that this isomorphism is induced by a symmetric morphism~\mbox{$\cR \otimes \cR \to \cO_D(\delta)$}.
For this we consider the diagram
\begin{equation}
\label{diag:q-prime}
\vcenter{\xymatrix@C=2.2em{
&
\big( \cE(-m) \otimes \cE^\vee \big) \oplus \big( \cE^\vee \otimes \cE(-m) \big) \ar[rr]^-{q \otimes 1 + 1 \otimes q} \ar[d]_{(\Tr,\Tr)} && 
\cE^\vee \otimes \cE^\vee \ar[rr] \ar[d]_{q'} && 
i_*(\cR \otimes \cR) \ar[r] \ar@{-->}[d] & 
0
\\
0 \ar[r] &
\cO(-m) \ar[rr]^-{\det(q)} &&
\cO(\delta) \ar[rr] &&
i_*\cO_D(\delta) \ar[r] &
0,
}}
\end{equation}
where the top row is the tensor square of resolution~\eqref{eq:ce-m-ce-vee-cf} of $\cC = i_*\cR$,
and~$\Tr \colon \cE \otimes \cE^\vee \to \cO$ is the trace map.
It is easy to check that the left square commutes.
Therefore, there exists a unique dashed arrow on the right such that the right square commutes.
Since~$q'$ is symmetric, the dashed arrow is symmetric as well.
Now it is easy to see that it induces the isomorphism~$\cR \to \cR^\vee(\delta)$ constructed above.
\end{proof}

\begin{remark}
\label{rem:non-unique-ss}
If~$\cC$ is a sheaf as in Lemma~\ref{lemma:cokernel-symmetric}, 
it is not in general true that the presentation of~$\cC$ as a symmetric sheaf is unique.
For instance, if~$\cE = \cO \oplus \cO$ and~$q = \diag(f,f)$ for a homogeneous polynomial~$f$, 
we have~$\cC \cong \cO_{D(f)} \oplus \cO_{D(f)}$, where~$D(f) \subset \P^n$ is the divisor of~$f$;
however, the construction of the lemma represents~$\cC$ as a symmetric sheaf on the non-reduced hypersurface~$D = D(f^2)$.
\end{remark}

As it is explained in Theorem~\ref{thm:hacm-existence} (see also Example~\ref{ex:epw})
the converse of Lemma~\ref{lemma:cokernel-symmetric} is not always true.
Below we explain the obstruction.

Let~$n = 2k + 1$.
Recall the graded ring~$\SS$ defined in~\eqref{def:ss}.
Let~$\cC$ be a $(d,\delta)$-symmetric sheaf on~$\P^n$.
A combination of the self-dual isomorphism~$\cC^\vee \cong \cC(m)[-1]$ of Lemma~\ref{lemma:cf-sd} with Serre duality 
endows the $\SS$-module~$\HH^k(\cC)$ with a perfect $\SS$-bilinear pairing 
\begin{equation}
\label{eq:hh-k-cf-pairing}
\HH^k(\cC) \otimes \HH^k(\cC) \to \SS^\vee(n+1-m) \twoheadrightarrow \kk(n+1-m),
\end{equation}
which is symmetric when $k$ is even and skew-symmetric when $k$ is odd.

\begin{lemma}
\label{lemma:hhkce-lagrangian}
Assume $n = 2k + 1$.
If~\eqref{eq:ce-m-ce-vee-cf} 
is a self-dual resolution of a symmetric sheaf~$\cC$ 
then the~$\SS$-submodule $\Ima(\HH^k(\cE^\vee) \to \HH^k(\cC))$ is Lagrangian for the pairing~\eqref{eq:hh-k-cf-pairing}.
In particular, $\dim(\HH^k(\cC))$ is even.
\end{lemma}
\begin{proof}
First, we need to check that the subspace $\Ima(\HH^k(\cE^\vee) \to \HH^k(\cC))$ in~$\HH^k(\cC)$ is isotropic.
For this we note that commutativity of the right square in~\eqref{diag:q-prime} 
implies that the restriction of the pairing~\eqref{eq:hh-k-cf-pairing} to this subspace factors as the composition
\begin{equation*}
\HH^k(\cE^\vee) \otimes \HH^k(\cE^\vee) \xrightarrow{\ q'\ }
\HH^{2k}(\cO(\delta)) \to
\HH^{2k}(\cO_D(\delta)) \to
\HH^{2k+1}(\cO(-m)) = \SS^\vee(n + 1 - m),
\end{equation*}
and it follows that it is zero since the composition of the two middle arrows is.

On the other hand, by Serre duality the maps~$\HH^k(\cE) \to \HH^k(\cE^\vee)$ 
and~$\HH^{k+1}(\cE) \to \HH^{k+1}(\cE^\vee)$
in the cohomology exact sequence
\begin{equation*}
\dots \to \HH^k(\cE) \xrightarrow{\ q\ } \HH^k(\cE^\vee) \to \HH^k(\cC) \to \HH^{k+1}(\cE) \xrightarrow{\ q\ } \HH^{k+1}(\cE^\vee) \to \dots
\end{equation*}
are mutually dual (up to shift of internal grading), so we conclude that
\begin{multline*}
\dim\Ima(\HH^k(\cE^\vee) \to \HH^k(\cC)) =
\dim\Coker(\HH^k(\cE) \to \HH^k(\cE^\vee)) \\ =
\dim\Ker(\HH^{k+1}(\cE) \to \HH^{k+1}(\cE^\vee))=
\dim\Coker(\HH^k(\cE^\vee) \to \HH^k(\cC)), 
\end{multline*}
hence~$\dim(\Ima(\HH^k(\cE^\vee) \to \HH^k(\cC))) = \frac12 \dim(\HH^k(\cC))$, and hence~$\Ima(\HH^k(\cE^\vee) \to \HH^k(\cC))$ is Lagrangian.
\end{proof}

\begin{remark}
\label{rem:ra-k}
Lemma~\ref{lemma:hhkce-lagrangian} gives an important obstruction 
to the existence of a self-dual resolution for a symmetric sheaf~$\cC$:
if $n = 2k + 1$ and $k$ is even the class of the quadratic space $\HH^k(\cC)$ in the Witt group~$\rW(\kk)$ must be trivial
(in particular, the dimension of the space $\HH^k(\cC)$ must be even).
Note also that the latter condition is sufficient for the existence of a Lagrangian $\SS$-submodule $\rA^k \subset \HH^k(\cC)$. 
Indeed, taking into account the twist in~\eqref{eq:hh-k-cf-pairing} we see that the subspace
\begin{equation}
\label{eq:ra-k-typical}
\rA^k = 
\begin{cases}
\moplus_{t > (m - n - 1)/2} H^k(\P^n, \cC(t)) \hphantom{{} \oplus \rA^k_{(m-n-1)/2}} \subset \HH^k(\cC), & \text{if $m - n - 1$ is odd},\\
\moplus_{t > (m - n - 1)/2} H^k(\P^n, \cC(t)) \oplus \rA^k_{(m-n-1)/2} \subset \HH^k(\cC), & \text{if $m - n - 1$ is even},
\end{cases}
\end{equation} 
is an $\SS$-submodule of $\HH^k(\cC)$ and it is Lagrangian as soon as 
\begin{equation*}
\rA^k_{(m-n-1)/2} \subset H^k(\P^n, \cC((m - n - 1)/2))
\end{equation*}
is a Lagrangian subspace for the restriction of the pairing~\eqref{eq:hh-k-cf-pairing}
(note that if $m - n - 1$ is odd the Witt class of~$\HH^k(\cC)$ is trivial,
and if $m - n - 1$ is even, this class equals the class of the space~$H^k(\P^n, \cC((m - n - 1)/2))$,
hence the latter has a Lagrangian subspace as soon as the Witt class of the former is trivial).
Note also that the submodule~$\rA^k \subset \HH^k(\cC)$ defined by~\eqref{eq:ra-k-typical} 
is shadowless in the sense of Definition~\ref{def:shadow}.
\end{remark}

The obstruction of Lemma~\ref{lemma:hhkce-lagrangian} is well-known to be non-trivial.

\begin{example}
\label{ex:epw}
Let $i \colon D \hookrightarrow \P^5$ be the so-called EPW sextic (see~\cite[Example~9.3]{EPW01}).
Then there is a sheaf~$\cR$ on~$D$ such that $\cC = i_*\cR$ is symmetric, but $\dim(\HH^2(\cC)) = 1$.
Consequently, $\cC$ does not admit a self-dual resolution.
\end{example}

The following fundamental result has been proved by Casnati and Catanese.

\begin{theorem}[{\cite[Theorem~0.3]{CC97}, \cite[Theorem~9.1]{EPW01}}]
\label{thm:hacm-existence}
Let $\cC$ be a $(d,\delta)$-symmetric sheaf on~$\P^n$.
If $n = 2k + 1$ and~$k$ and~$m = d - \delta$ are even 
assume that the class of the space~$H^k(\P^n,\cC((m - n - 1)/2))$ endowed with the quadratic form~\eqref{eq:hh-k-cf-pairing}
is trivial in the Witt group~$\rW(\kk)$.
Let~$\rA^k \subset \HH^k(\cC)$ be any Lagrangian $\SS$-submodule defined as in~\eqref{eq:ra-k-typical}.
Then there is a symmetric resolution~\eqref{eq:ce-m-ce-vee-cf} 
such that $\Ima(\HH^k(\cE^\vee) \to \HH^k(\cC)) = \rA^k$.
\end{theorem}

\begin{remark}
\label{remark:epw}
In fact, Theorem~\ref{thm:hacm-existence} has been proved in~\cite{CC97} for $n = 3$ and over an algebraically closed field, 
but as it is pointed out in~\cite[Remark~2.2]{CC97} the proof applies to any~$n$ 
as soon as a Lagrangian subspace in~$H^k(\P^n,\cC((m - n - 1)/2))$ exists.
For $n = 3$ this condition is automatically satisfied because the pairing~\eqref{eq:hh-k-cf-pairing} is skew-symmetric,
but for $n = 2k + 1$ with $k$ even this becomes a non-trivial obstruction.
\end{remark}

\subsection{Elementary modifications}
\label{subsec:elementary-modifications}

Throughout this section we fix a generically non-degenerate quadratic form~$(\cE,q)$ 
with its associated self-dual morphism $q \colon \cE(-m) \to \cE^\vee$.

We will need the following auxiliary result.
Recall the exact sequence~\eqref{eq:koszul-t} and note that it can be considered as concatenation of short exact sequences
\begin{equation}
\label{eq:koszul}
0 \to \bw{p-1}\cT \to \bw{p}V \otimes \cO(p) \to \bw{p}\cT \to 0.
\end{equation} 
We denote by $\tau_p \in \Ext^1(\bw{p}\cT, \bw{p-1}\cT) = \Ext^1(\Omega^{p-1},\Omega^p)$ the extension class of~\eqref{eq:koszul}.

The following observation is used to translate higher cohomology of~$\cE$ to hyperbolic extension classes.
Recall from~\eqref{eq:shadow} the definition of the shadow~$\Sha(p,t)$.

\begin{proposition}
\label{prop:eps-p}
Let~$\cE$ be a vector bundle on~$\P^n$.
Let~$1 \le p \le n-1$ and let~$0 \ne \eps_p \in H^p(\P^n,\cE(t))$ be a cohomology class such that 
\begin{equation}
\label{eq:cohomology-assumption}
H^{p'}(\P^n, \cE(t')) = 0
\qquad\text{for any~$(p',t') \in \Sha(p,t)$}.
\end{equation} 
Then for~$0 \le i \le p$ there exists a sequence of classes~$\eps_i \in H^i(\P^n, \cE(t) \otimes \bw{p-i}\cT) = \Hom(\Omega^{p-i}[-i],\cE(t))$
that fit into a commutative diagram
\begin{equation*}
\vcenter{\xymatrix{
\cO[-p] \ar[r]^-{\tau_1} 		\ar[d]_{\eps_{p}} &
\Omega^{1}[1-p] \ar[r]^-{\tau_2} 	\ar[dl]_(.6){\eps_{p-1}} &
\dots \ar[r]^-{\tau_{p-1}} 		\ar@{}[dll]_(.6){\cdots} &
\Omega^{p-1}[-1] \ar[r]^-{\tau_{p}} 	\ar[dlll]_(.6){\eps_{1}} & 
\Omega^{p} \ar@/^/[dllll]^(.3){\eps_{0}} \ar@{^{(}->}[r] &
\bw{p}V \otimes \cO(-p) \ar@{..>}@/^1.5em/[dlllll]
\\
\cE(t),
}}
\end{equation*}
\textup(where $\tau_i$ are the extension classes of the complexes~\eqref{eq:koszul}\textup);
in other words
\begin{equation}
\label{eq:eps-tau}
\eps_p = \eps_i \circ \tau_{p-i} \circ \dots \circ \tau_1
\end{equation} 
for each~$0 \le i \le p$.
Moreover, for~$i \ge 1$ such~$\eps_i$ are unique, while~$\eps_0$ is unique up to a composition
\begin{equation*}
\Omega^p \hookrightarrow \bw{p}V \otimes \cO(-p) \to \cE(t),
\end{equation*}
where the first arrow is the canonical embedding.

Finally, if one of the following conditions is satisfied
\begin{align}
\label{item:eps-2n}
2p &\le n, &\text{or}\\
\label{item:eps-2n+1}
2p &= n + 1
\quad\text{and}\quad
2t + m + n + 1 \ge 0, 
\end{align} 
then $q(\eps_p,\eps_p) = 0$ implies $q(\eps_{p-i},\eps_{p-i}) = 0$ for each $1 \le i \le p -1$,
where 
\begin{equation*}
q(\eps_{p-i},\eps_{p-i}) \in \Ext^{2(p-i)}(\Omega^{i}(-t-m), \bw{i}\cT(t))
\end{equation*}
is defined as 
the composition~$\Omega^{i}(-t-m)[i - p] \xrightarrow{\ \eps_{p-i}\ } \cE(-m) \xrightarrow{\ q\ } \cE^\vee \xrightarrow{\ \eps_{p-i}\ } \bw{i}\cT(t)[p-i]$.
\end{proposition}

\begin{proof}
The existence of~$\eps_i$ satisfying~\eqref{eq:eps-tau} and their uniqueness 
follow by descending induction from the cohomology exact sequences 
of complexes~\eqref{eq:koszul} tensored with~$\cE(t)$ in view of the vanishing~\eqref{eq:cohomology-assumption}.

For the second assertion we also induct on~$i$. 
Assume $1 \le i \le p - 1$.
We have 
\begin{equation*}
q(\eps_{p-i},\eps_{p-i}) \in 
\Ext^{2(p-i)}(\Omega^{i}(-t-m), \bw{i}\cT(t)) =
H^{2(p-i)}(\P^n, \bw{i}\cT \otimes \bw{i}\cT(2t+m)).
\end{equation*}
Consider the tensor square of~\eqref{eq:koszul}:
\begin{equation}
\label{eq:bw-bw-ct}
0 \to 
\bw{i-1}\cT \otimes \bw{i-1}\cT \to 
\Big( \bw{i}V \otimes \bw{i-1}\cT(i) \Big)^{\oplus 2} \to 
\bw{i}V \otimes \bw{i}V \otimes \cO(2i) \to 
\bw{i}\cT \otimes \bw{i}\cT \to 0.
\end{equation}
Note that its extension class is $\tau_i \otimes \tau_i \in \Ext^2(\bw{i}\cT \otimes \bw{i}\cT, \bw{i-1}\cT \otimes \bw{i-1}\cT)$.
Furthermore, we note that
\begin{align*}
H^{2(p-i)+1}(\P^n, \bw{i-1}\cT_{\P^n}(2t + m + i)) &=
H^{2(p-i)+1}(\P^n, \Omega^{n-i+1}(2t + m + i + n + 1)) = 0,\\
H^{2(p-i)}(\P^n, \cO(2t + m + 2i)) &= 0.
\end{align*}
Indeed, if~\eqref{item:eps-2n} holds we use~$1 \le 2(p - i) + 1 < n - i + 1 \le n$ together with~\eqref{eq:h-q-omega-p} for the first vanishing 
and~$1 \le 2(p - i) < n$ for the second.
Similarly, if~\eqref{item:eps-2n+1} holds and~$i \ge 2$ the same arguments prove the vanishings.
Finally, if~\eqref{item:eps-2n+1} holds and~$i = 1$ the same arguments prove the second vanishing,
while the first cohomology space is equal to~$H^n(\P^n, \cO(2t + m + 1))$, hence vanishes since~$2t + m + 1 > -n - 1$.

The cohomology vanishings that we just established imply that the morphism
\begin{equation*}
\tau_i \otimes \tau_i \colon 
H^{2(p-i)}(\P^n, \bw{i}\cT \otimes \bw{i}\cT(2t+m)) \to 
H^{2(p - i + 1)}(\P^n, \bw{i-1}\cT \otimes \bw{i-1}\cT(2t+m))
\end{equation*}
is injective, and hence the condition 
\begin{equation*}
0 = 
q(\eps_{p-i+1},\eps_{p-i+1}) =
q(\eps_{p-i} \circ \tau_i,\eps_{p-i+1} \circ \tau_i) = 
(\tau_i \otimes \tau_i)(q(\eps_{p-i},\eps_{p-i}))
\end{equation*}
implies $q(\eps_{p-i},\eps_{p-i}) = 0$.
\end{proof}

The following \emph{elementary modification} procedure allows us 
to kill an isotropic cohomology class of a quadratic form
by a hyperbolic extension (see~\S\ref{subsec:he}).
Recall that for a cohomology class~$\eps_p \in H^p(\P^n,\cE(t))$ 
we denote by $q(\eps_p) \in H^p(\P^n,\cE^\vee(m+t))$ the image of~$\eps_p$ under the map 
\begin{equation*}
H^p(\P^n,\cE(t)) \xrightarrow{\ q\ } H^p(\P^n,\cE^\vee(m+t)).
\end{equation*}
Using the class~$q(\eps_p)$ we consider the map 
\begin{equation}
\label{eq:q-eps-p}
\moplus_{i=1}^{n-1} \HH^i(\cE) \twoheadrightarrow
H^{n-p}(\P^n,\cE(-m-t-n-1)) \xrightarrow{\ q(\eps_p)\ }
H^n(\P^n,\cO(-n-1)) = \kk,
\end{equation}
where the first arrow is the projection to a direct summand.
We denote by~$q(\eps_p)^\perp \subset \moplus_{i=1}^{n-1} \HH^i(\cE)$ the kernel of~\eqref{eq:q-eps-p}.

\begin{proposition}
\label{prop:elementary-modification}
Let $\eps_p \in H^p(\P^n,\cE(t))$ be a cohomology class such that $q(\eps_p,\eps_p) = 0$ 
and assume that the condition~\eqref{eq:cohomology-assumption} holds and either~\eqref{item:eps-2n} or~\eqref{item:eps-2n+1} 
is satisfied.
Let~$\eps_1 \in \Ext^1(\Omega^{p-1}(-t),\cE)$ be the extension class defined in Proposition~\textup{\ref{prop:eps-p}}.
Then~$\eps_1$ is $q$-isotropic and for any hyperbolic extension~$(\cE_+,q_+)$ of~$(\cE,q)$ with respect to~$\eps_1$ we have
\begin{equation}
\label{eq:hh-ce+}
\moplus_{i=1}^{n-1} \HH^i(\cE_+) = 
\begin{cases}
\left.\left(q(\eps_p)^\perp \cap \moplus_{i=1}^{n-1} \HH^i(\cE) \right) \right/\kk\eps_p, &
\text{if $q(\eps_p) \ne 0$,}
\\
\hspace{1em}
\kk\eps_+ \oplus \left.\left(\moplus_{i=1}^{n-1} \HH^i(\cE) \right) \right/\kk\eps_p, &
\text{if $q(\eps_p) = 0$}
\end{cases}
\end{equation}
where in the second line~$\eps_+ \in H^{n-p+1}(\P^n, \cE_+(t+m+n+1))$ is a nonzero cohomology class 
that depends on the choice of~$(\cE_+,q_+)$.
\end{proposition}

\begin{proof}
By Proposition~\ref{prop:eps-p} we have $q(\eps_1,\eps_1) = 0$, hence the extension class~$\eps_1$ is $q$-isotropic
and a hyperbolic extension~$(\cE_+,q_+)$ exists by Theorem~\ref{thm:he}.
By Lemma~\ref{lemma:eps-q-eps} its underlying bundle~$\cE_+$ has a length~3 filtration with the factors
\begin{equation*}
\bw{p-1}\cT(t + m),\
\cE,\
\Omega^{p-1}(-t)
\end{equation*}
linked by the classes~$q(\eps_1) \in \Ext^1(\cE, \bw{p-1}\cT(t + m))$ and~$\eps_1 \in \Ext^1(\Omega^{p-1}(-t), \cE)$,
respectively.
Recall that~$\tau_i$ denote the extension classes of complexes~\eqref{eq:koszul}.

Consider the specral sequence of a filtered complex that computes the cohomolgoy of (twists of)~$\cE_+$;
the terms of its first page~$\bE_1^{\bullet,\bullet}$ which compute intermediate cohomology look like
\begin{equation*}
\begin{aligned}
\bE_1^{-1,i} &= 
\HH^{i-1}(\Omega^{p-1}(-t)) &&= 
\begin{cases}
\kk(-t), & \text{if $i = p$},\\
0, & \text{otherwise for~$2 \le i \le n$},
\end{cases}
\\
\bE_1^{0,i} &= \HH^i(\cE),
\\
\bE_1^{1,i} &= 
\HH^{i+1}(\bw{p-1}\cT(t + m)) &&= 
\begin{cases}
\kk(t+m+n+1), & \text{if $i = n - p$},\\
0, & \text{otherwise for~$0 \le i \le n-2$},
\end{cases}
\end{aligned}
\end{equation*}
and the first differentials are given by $\eps_1 \colon \bE_1^{-1,i} \to \bE_1^{0,i}$ 
and~$q(\eps_1) \colon \bE_1^{0,i} \to \bE_1^{1,i}$, respectively.
In particular, there are only two possibly non-trivial differentials here:
\begin{align*}
\kk \xrightarrow[\simeq]{\ \tau_{p-1} \circ \dots \circ \tau_1\ }
H^{p-1}(\P^n,\Omega^{p-1}) &\xrightarrow{\hbox to 3em{\hfil$\scriptstyle \eps_1$\hfil}} H^p(\P^n, \cE(t)),
\qquad\text{and}\\
H^{n-p}(\P^n,\cE(-m-t-n-1))  &\xrightarrow{\hbox to 3em{\hfil$\scriptstyle q(\eps_1)$\hfil}} 
H^{n-p+1}(\P^n, \bw{p-1}\cT(-n-1)) \xrightarrow[\simeq]{\ \tau_{p-1} \circ \dots \circ \tau_1\ } \kk.
\end{align*}
Since the spectral sequence is supported in three columns, 
the differential~$\bd_2$ acts as~$\bE_2^{-1,i} \to \bE_2^{1,i-1}$,
and using~\eqref{eq:h-q-omega-p} we see that its source is nonzero only for~$i \in \{1,n+1\}$ (note that~$\eps_1 \ne 0$), 
while its target is nonzero only in~$i \in \{0, n+1-p, n\}$,
hence~$\bd_2 = 0$.
The further differentials a fortiori vanish,
so that~$\bE_\infty^{\bullet,\bullet} = \bE_2^{\bullet,\bullet}$.
On the other hand, by~\eqref{eq:eps-tau} the image of the first map is~$\kk\eps_p$
and the second map coincides with the map~$q(\eps_p)$ defined in~\eqref{eq:q-eps-p}.
Therefore, the totalization of~$\bE_2^{\bullet,\bullet}$ takes the form of the right-hand side of~\eqref{eq:hh-ce+},
where in the case~$q(\eps_p) = 0$ the class~$\eps_+$ comes from~$\bE_2^{1,n-p}$ which survives exactly in this case. 
\end{proof}

As explained in Theorem~\ref{thm:he} the construction of a hyperbolic extension might be ambiguous.
In the situation described in Proposition~\ref{prop:elementary-modification} this happens precisely
when the space~$\Ext^1(\bw2\Omega^{p-1},\cO(2t + m))$ is non-zero.
In the next lemma we determine when this happens.

\begin{lemma}
\label{lemma:ext1-omega}
Assume $2 \le 2p \le n + 1$.
Then the space $\Ext^1(\bw2\Omega^{p-1},\cO(s))$ is nonzero if and only if~$n = 2k + 1$, $p = k + 1$, $k \ge 1$ is odd, and $s = -n - 1$,
in which case~$\dim(\Ext^1(\bw2\Omega^{p-1},\cO(s))) = 1$. 
\end{lemma}

\begin{proof}
Set $k = p - 1$, so that $2k \le n - 1$.
We have $\Ext^1(\bw2\Omega^{p-1},\cO(s)) = H^1(\P^n, \bw2(\bw{k}\cT) \otimes \cO(s))$.
Taking the exterior square of~\eqref{eq:koszul-t} we see that~$\bw2(\bw{k}\cT)$ is quasiisomorphic to the complex of split bundles
of length~$2k$ if~$k$ is odd and~$2k-1$ if~$k$ is even.
Since split bundles have no intermediate cohomology and since~$2k \le n - 1$, the hypercohomology spectral sequence 
shows that~$H^1(\P^n, \bw2(\bw{k}\cT) \otimes \cO(s)) = 0$ unless~$k$ is odd and~$n = 2k + 1$, 
and in the latter case we have
\begin{equation*}
H^1(\P^n, \bw2(\bw{k}\cT) \otimes \cO(s)) = \Ker \Big( H^n(\P^n,\cO(s)) \to H^n(\P^n, V \otimes \cO(s+1)) \Big),
\end{equation*}
where the morphism in the right side is induced by the tautological embedding~$\cO \hookrightarrow V \otimes\cO(1)$.
Now it is easy to see that this space is zero unless~$s = -n-1$, in which case it is 1-dimensional.
\end{proof}

Now let~$\cC$ be the cokernel sheaf of a generically non-degenerate quadratic form~$(\cE,q)$.
The next result shows that in the case 
where the construction of an elementary modification of Proposition~\ref{prop:elementary-modification} is ambiguous,
i.e., $\Ext^1(\bw2\Omega^{p-1}, \cO(2t+m)) \ne 0$,
one can choose one such modification~$(\cE_+,q_+)$, which has an additional nice property,
namely, 
it has a prescribed image of~$\HH^k(\cE_+^\vee)$ in~$\HH^k(\cC)$.
For our purposes it will be enough to consider the case where the bundle~$\cE$ is VLC, see Definition~\ref{def:lacm-uacm}.

So, assume~$n = 2k + 1$ and the bundle~$\cE$ in~\eqref{eq:ce-m-ce-vee-cf} is VLC.
Note that~$\cE^\vee$ is VUC by Lemma~\ref{lemma:hacm-duality}.
By Lemma~\ref{lemma:hhk-ce+} we have an exact sequence of graded~$\SS$-modules
\begin{equation}
\label{eq:lagrangian-cohomology}
0 \to \HH^{k}(\cE^\vee) \to \HH^k(\cC) \to \HH^{k+1}(\cE) \to 0,
\end{equation}
see~\eqref{eq:hhk-ce-pm}. 
Recall also that the space~$\HH^k(\cC)$ is endowed with the perfect pairing~\eqref{eq:hh-k-cf-pairing}
and that the subspace~$\HH^{k}(\cE^\vee) \subset \HH^k(\cC)$ is Lagrangian, see Lemma~\ref{lemma:hhkce-lagrangian}.
In particular, the pairing induces an isomorphism
\begin{equation*}
\HH^{k+1}(\cE) \cong \HH^{k}(\cE^\vee)^\vee
\end{equation*}
Using this isomorphism, any class~$\eps_{k+1} \in H^{k+1}(\P^n,\cE(t))$ 
can be considered as a homogeneous linear function on~$\HH^{k}(\cE^\vee)$;
we denote by~$\eps_{k+1}^\perp \subset \HH^{k}(\cE^\vee)$ its kernel. 
Note that~$\eps_{k+1}^\perp$ is a graded isotropic $\SS$-submodule in~$\HH^{k}(\cE^\vee)$ and hence also in~$\HH^{k}(\cC)$.

\begin{proposition}
\label{prop:em-lagrangian}
Assume~$k \ge 1$, $n = 2k + 1$, and the bundle~$\cE$ in~\eqref{eq:ce-m-ce-vee-cf} is VLC.
Let~$t$ be the maximal integer such that~$H^{k+1}(\P^n,\cE(t)) \ne 0$, 
let~$\eps_{k+1} \in H^{k+1}(\P^n,\cE(t))$ be a nonzero class,
and let~$\eps_{k+1}^\perp \subset \HH^{k}(\cE^\vee)$ be the corresponding graded isotropic $\SS$-submodule.
For each graded Lagrangian $\SS$-submodule~$\rA \subset \HH^{k}(\cC)$ such that~$\rA \ne \HH^k(\cE^\vee)$ and
\begin{equation}
\label{eq:lagrangian-property}
\eps_{k+1}^\perp \subset \rA \cap \HH^k(\cE^\vee)
\end{equation}
there is a unique elementary modification~$(\cE_+,q_+)$ of~$(\cE,q)$ with respect to~$\eps_{k+1}$ such that
\begin{equation}
\label{eq:hhk-ceplusvee}
\HH^k(\cE_+^\vee) = \rA.
\end{equation}
\end{proposition}

\begin{proof}
Let~$\eps_1 \in \Ext^1(\Omega^k(-t),\cE)$ be the extension class constructed from~$\eps_{k+1}$ in Proposition~\ref{prop:eps-p}
and consider the variety~$\HE(\cE,q,\eps_1)$ of all hyperbolic extensions of~$(\cE,q)$ with respect to~$\eps_1$,
i.e., the set of all elementary modifications of~$(\cE,q)$ with respect to~$\eps_{k+1}$.
By~\eqref{eq:hh-ce+} every~$(\cE_+,q_+) \in \HE(\cE,q,\eps_1)$ is a VLC bundle, hence~$\HH^{k}(\cE_+^\vee)$ 
is a graded Lagrangian $\SS$-submodule in~$\HH^k(\cC)$ by Lemma~\ref{lemma:hhkce-lagrangian}.
Moreover, the equality~\eqref{eq:hh-ce+} also implies that~$\eps_{k+1}^\perp \subset \HH^{k}(\cE_+^\vee)$, 
i.e., $\rA = \HH^{k}(\cE_+^\vee)$ satisfies~\eqref{eq:lagrangian-property}.
Therefore, there is a morphism 
\begin{equation}
\label{eq:lagrangian-morphism}
\lambda \colon \HE(\cE,q,\eps_1) \to \LGr_{\eps_{k+1}}(\HH^k(\cC)),
\qquad 
(\cE_+,q_+) \mapsto [\HH^{k}(\cE_+^\vee)],
\end{equation}
where~$\LGr_{\eps_{k+1}}(\HH^k(\cC))$ is the variety of all graded Lagrangian $\SS$-submodules~$\rA \subset \HH^k(\cC)$ satisfying~\eqref{eq:lagrangian-property}.
We will show that~$\lambda$ is an isomorphism 
onto the complement of the point~$[\HH^k(\cE^\vee)]$ in~$\LGr_{\eps_{k+1}}(\HH^k(\cC))$.

First we check 
that the image of~$\lambda$ is contained in the complement of~$[\HH^k(\cE^\vee)]$ in~$\LGr_{\eps_{k+1}}(\HH^k(\cC))$.
Recall that~$\eps_1 \in \Ext^1(\Omega^k(-t),\cE)$ denotes the extension class constructed from~$\eps_{k+1}$ in Proposition~\ref{prop:eps-p}
and let, as usual, \mbox{$q(\eps_1) \in \Ext^1(\cE,\Omega^{k+1}(t+m+n+1))$} be the class 
obtained from it by the application of~$q$.
Let $(\cE_+,q_+)$ be any hyperbolic extension of~$(\cE,q)$ with respect to~$\eps_1$, 
so that $(\cE,q)$ is the hyperbolic reduction of~$(\cE_+,q_+)$ with respect to an embedding~$\Omega^{k+1}(t+m+n+1) \hookrightarrow \cE_+$.
Then we have the following commutative diagram
\begin{equation}
\label{eq:ce-cep-cepp-ce+}
\vcenter{\xymatrix@R=2ex{
& 0 \ar[d] & 0 \ar[d] 
\\
&
\Omega^{k+1}(2t+m+n+1) \ar@{=}[r] \ar[d] &
\Omega^{k+1}(2t+m+n+1) \ar[d] 
\\
0 \ar[r] &
\cE''(t) \ar[r] \ar[d] &
\cE_+(t) \ar[r] \ar[d] &
\Omega^k \ar[r] \ar@{=}[d] &
0
\\
0 \ar[r] &
\cE(t) \ar[r] \ar[d] &
\cE'(t) \ar[r] \ar[d] &
\Omega^k \ar[r] &
0
\\
& 0 & 0
}}
\end{equation}
with the extension class of the bottom row being~$\eps_1$ and that of the left column being~$q(\eps_1)$.
Note that the cohomology exact sequence of the bottom row and the nontriviality of~$\eps_1$ 
imply that~$\cE'$ is VLC, hence~${\cE'}^\vee$ is VUC.
Similarly, the cohomology exact sequence of the left column implies that~$\cE''$ is VLC.
We will use these observations below.

Consider the dual of the diagram~\eqref{eq:ce-cep-cepp-ce+} 
and the induced cohomology exact sequences:
\begin{equation}
\label{eq:hh-ce-cep-cepp-ce+}
\vcenter{\xymatrix@R=4ex@C=5em{
& \kk(-t-m) \ar@{=}[r] & \kk(-t-m)
\\
\kk(t+n+1) &
\HH^k({\cE''}^\vee) \ar[l]_-{\eps''} \ar[u]^{\iota} &
\HH^k(\cE_+^\vee) \ar[l] \ar[u]
\\
\kk(t+n+1) \ar@{=}[u] &
\HH^k(\cE^\vee) \ar[l]_-{\eps_1} \ar[u] &
\HH^k({\cE'}^\vee) \ar[l] \ar[u]
}}
\end{equation}
(the map~$\iota$ is induced by the embedding~$\Omega^{k+1}(2t+m+n+1) \to \cE''(t)$ 
in the left column of~\eqref{eq:ce-cep-cepp-ce+}).
Since~${\cE'}^\vee$ is VUC, the upper arrow in the right column of~\eqref{eq:hh-ce-cep-cepp-ce+} is surjective.
From the commutativity of the diagram we conclude that the composition 
\begin{equation*}
\HH^k(\cE_+^\vee) \xrightarrow{\quad} \HH^k({\cE''}^\vee) \xrightarrow{\ \iota\ } \kk
\end{equation*}
(of the right arrow in the middle row and the upper arrow in the middle column) is nontrivial, while
\begin{equation*}
\HH^k(\cE^\vee) \xrightarrow{\quad} \HH^k({\cE''}^\vee) \xrightarrow{\ \iota\ } \kk,
\end{equation*}
(the composition of arrows in the middle column) is zero.
This proves that the images of~$\HH^k(\cE_+^\vee)$ and~$\HH^k(\cE^\vee)$ in~$\HH^k({\cE''}^\vee)$ are distinct.
On the other hand, we have an obvious commutative diagram
\begin{equation*}
\xymatrix{
0 \ar[r] &
\cE(-m) \ar[r]^-q \ar@{^{(}->}[d] &
\cE^\vee \ar[r] \ar@{^{(}->}[d] & 
\cC \ar[r] \ar@{=}[d] &
0
\\
0 \ar[r] &
\cE'(-m) \ar[r] &
{\cE''}^\vee \ar[r] & 
\cC \ar[r] &
0
\\
0 \ar[r] &
\cE_+(-m) \ar[r]^-{q_+} \ar@{->>}[u] &
\cE_+^\vee \ar[r] \ar@{->>}[u] & 
\cC \ar[r] \ar@{=}[u] &
0
}
\end{equation*}
and since~$\cE'$ is VLC, the map~$\HH^k({\cE''}^\vee) \to \HH^k(\cC)$ is injective, hence
\begin{equation*}
[\HH^k(\cE_+^\vee)] \ne [\HH^k(\cE^\vee)] \in \LGr_{\eps_{k+1}}(\HH^k(\cC)), 
\end{equation*}
hence the image of~$\lambda$ is contained in the complement of~$[\HH^k(\cE^\vee)]$.

Now we separate the following cases:
\begin{enumerate}\renewcommand{\theenumi}{\alph{enumi}}
\item 
\label{item:point-k-even}
$k$ is even;
\item 
\label{item:point-degree-nonzero}
$2t + m + n + 1 \ne 0$;
\item 
\label{item:p1}
$2t + m + n + 1 = 0$ and~$k$ is odd.
\end{enumerate}
First, we describe the variety~$\LGr_{\eps_{k+1}}(\HH^k(\cC)) \setminus [\HH^k(\cE^\vee)]$ in each case. 

In case~\eqref{item:point-k-even} the bilinear form~\eqref{eq:hh-k-cf-pairing} is symmetric, 
hence there exists exactly two Lagrangian subspaces in~$\HH^k(\cC)$ containing~$\eps_{k+1}^\perp$,
hence~$\LGr_{\eps_{k+1}}(\HH^k(\cC)) \setminus [\HH^k(\cE^\vee)]$ is a single point.

In case~\eqref{item:point-degree-nonzero} the 2-dimensional space~$(\eps_{k+1}^\perp)^\perp / \eps_{k+1}^\perp$ 
lives in two distinct degrees, 
hence it has exactly two \emph{graded} Lagrangian subspaces,
hence~$\LGr_{\eps_{k+1}}(\HH^k(\cC)) \setminus [\HH^k(\cE^\vee)]$ is again a single point.

Finally, in case~\eqref{item:p1} the 2-dimensional space~$(\eps_{k+1}^\perp)^\perp / \eps_{k+1}^\perp$ 
is symplectic and lives in a single degree, hence~$\LGr_{\eps_{k+1}}(\HH^k(\cC)) \cong \P^1$
and~$\LGr_{\eps_{k+1}}(\HH^k(\cC)) \setminus [\HH^k(\cE^\vee)] \cong \A^1$.

Now we see that in cases~\eqref{item:point-k-even} and~\eqref{item:point-degree-nonzero} 
the map~$\lambda$ is a map between two one-point sets, hence it is an isomorphism.
Finally, in case~\eqref{item:p1} it is a map~$\A^1 \to \A^1$, 
and to show it is an isomorphism (and thus to complete the proof of the proposition),
it is enough to check its injectivity. 
So, for the rest of the proof we assume~$k$ is odd and~$2t + m + n + 1 = 0$ 
and prove that~$\lambda$ is injective.

First, we note that since the extension class~$q(\eps_1)$ of the left column of~\eqref{eq:ce-cep-cepp-ce+}
does not depend on any choice,
the hyperbolic extension~$(\cE_+,q_+)$ is determined by the class~$\eps'' \in \Ext^1(\Omega^k, \cE''(t))$ of the middle row in~\eqref{eq:ce-cep-cepp-ce+}, 
which is a lift of~$\eps_1 \in \Ext^1(\Omega^k, \cE(t))$ with respect to the exact sequence
\begin{equation}
\label{eq:action-es}
\Hom(\Omega^k, \cE(t)) \to \Ext^1(\Omega^k, \Omega^{k+1}(2t+m+n+1)) \to \Ext^1(\Omega^k, \cE''(t)) \to \Ext^1(\Omega^k, \cE(t)).
\end{equation}
Since the left arrow in the middle row of~\eqref{eq:hh-ce-cep-cepp-ce+} is determined by the class~$\eps''$ 
and since~$\lambda((\cE_+,q_+))$ is its kernel, we conclude that the morphism~$\lambda$ factors as a composition
\begin{equation*}
\HE(\cE,q,\eps_1) \xrightarrow{\ \lambda_1\ } \Ext^1(\Omega^k, \cE''(t)) \xrightarrow{\ \lambda_2\ } \LGr_{\eps_{k+1}}(\HH^k(\cC)),
\end{equation*}
where~$\lambda_1$ takes a hyperbolic extension~$(\cE_+,q_+)$ to the extension class of the middle row in~\eqref{eq:ce-cep-cepp-ce+}
and~$\lambda_2$ takes an extension class~$\eps'' \in \Ext^1(\Omega^k, \cE''(t))$ 
to the kernel of the map~$\eps''$ in~\eqref{eq:hh-ce-cep-cepp-ce+}.
To check the injectivity of~$\lambda$ it is enough to check the injectivity of~$\lambda_1$ and~$\lambda_2$.

To prove the injectivity of~$\lambda_2$ consider the composition
\begin{equation*}
\Ext^1(\Omega^k, \cE''(t)) = H^{k+1}(\P^n, \cE''(t)) \xrightarrow{{\ \raisebox{-0.5ex}[0ex][0ex]{$\sim$}\ }} 
\Hom(H^k(\P^n, {\cE''}^\vee(-t-n-1)), \kk) \subset \Hom(\HH^k({\cE''}^\vee), \kk(t+n+1))).
\end{equation*}
The equality follows from Proposition~\ref{prop:eps-p} applied to~$\cE''$
(recall that the bundle~$\cE''$ is VLC 
and, moreover, we have~$H^{k+1}(\P^n,\cE''(s)) = 0$ for~$s > t$ by definition of~$t$ and the assumption~$2t + m + n + 1 = 0$)
and the middle arrow is an isomorphism by Serre duality.
Therefore,
the composition is injective, and since~$\lambda_2(\eps'')$ 
is determined by the image of~$\eps''$ under this composition,
we conclude that~$\lambda_2$ is injective.

Finally, we note that~$\HE(\cE,q,\eps_1)$ comes with a transitive action of the group 
\begin{equation*}
\Ext^1(\bw2\Omega^k, \cO(2t+m)) \subset \Ext^1(\Omega^k, \Omega^{k+1}(2t+m+n+1)).
\end{equation*}
Therefore, to check the injectivity of~$\lambda_1$, it is enough to check the injectivity of the middle arrow in~\eqref{eq:action-es}.
And for this, it is enough to check that the first arrow in~\eqref{eq:action-es} vanishes.
To prove this vanishing consider the commutative square
\begin{equation*}
\xymatrix@C=9em{
\Hom(\bw{k}V \otimes \cO(-k), \cE(t)) \ar[r] \ar[d] &
\Hom(\Omega^k, \cE(t)) \ar[d]
\\
\Ext^1(\bw{k}V \otimes \cO(-k), \Omega^{k+1}(2t + m + n + 1)) \ar[r] &
\Ext^1(\Omega^k, \Omega^{k+1}(2t + m + n + 1)),
}
\end{equation*}
where the vertical arrows are induced by the extension class~$q(\eps_1)$ of the left column of~\eqref{eq:ce-cep-cepp-ce+},
and the horizontal arrows are induced by the morphisms in the dual of~\eqref{eq:koszul-t} with~$s = k$. 
The space in the lower left corner is zero by~\eqref{eq:h-q-omega-p} (recall that~$k \ge 1$),
hence the compositions of arrows are zero.
On the other hand, the argument of Proposition~\ref{prop:eps-p} shows 
that the top horizontal arrow is surjective.
Therefore the right vertical arrow is zero, and as we explained above, 
this implies the injectivity of~$\lambda_1$, and hence of~$\lambda$,
and completes the proof of the proposition.
\end{proof}

The elementary modification~$(\cE_+,q_+)$ of~$(\cE,q)$ 
satisfying the properties of Proposition~\ref{prop:em-lagrangian} for a given lift~$\teps_{k+1}$ of~$\eps_{k+1}$,
will be referred to as {\sf the refined elementary modification} associated with the class~$\teps_{k+1}$.

\subsection{Modification theorem}
\label{subsec:hacm-modifications}

Recall that a quadratic form~$(\cE,q)$ is called {\sf unimodular} if the corresponding cokernel sheaf~$\cC$ vanishes,
i.e., if~$q \colon \cE(-m) \to \cE^\vee$ is an isomorphism.
Recall the definitions~\eqref{eq:trivial-linear} and~\eqref{eq:trivial-omega} of standard unimodular quadratic forms.
We will say that a standard unimodular quadratic form is {\sf anisotropic} 
if~$W = W^0$ and the form~$q_{W^0}$ is symmetric and anisotropic.

To prove the main result of this section we need the following simple observations.
Recall the notion of linear minimality, see Definition~\ref{def:linearly-minimal}

\begin{lemma}
\label{lemma:linear-summand}
Assume~$(\cE,q)$ is a generically non-degenerate quadratic form 
such that~$q \colon \cE(-m) \to \cE^\vee$ is not linearly minimal.
Then~$(\cE,q)$ is isomorphic to the orthogonal direct sum~$(\cE_0,q_0) \oplus (\cE_1,q_1)$,
where the second summand is a standard unimodular quadratic form~\eqref{eq:trivial-linear} of rank~$1$ or~$2$.
\end{lemma}
\begin{proof}
Since~$q$ is not linearly minimal, it can be written as a direct sum 
of morphisms~\mbox{$f \colon \cE'(-m) \to \cE''$} 
and~\mbox{$\id \colon \cO(t - m) \to \cO(t - m)$} for some~$t \in \ZZ$;
in particular~$\cE \cong \cE' \oplus \cO(t)$, 
and the restriction of~$q$ to the summand~$\cO(t-m)$ of~$\cE(-m)$ is a split monomorphism.
Consider the composition 
\begin{equation*}
\varphi \colon \cO(t-m) \hookrightarrow \cE'(-m) \oplus \cO(t-m) =  
\cE(-m) \xrightarrow{\ q\ } \cE^\vee = {\cE'}^\vee \oplus \cO(-t).
\end{equation*}
Let $\varphi_0 \colon \cO(t-m) \to {\cE'}^\vee$ and $\varphi_1 \colon \cO(t-m) \to \cO(-t)$ be its components.
Since~$\varphi$ is a split monomorphism, 
there is a map $\psi = (\psi_0,\psi_1) \colon {\cE'}^\vee \oplus \cO(-t) \to \cO(t - m)$
such that 
\begin{equation*}
\psi \circ \varphi =
\psi_0 \circ \varphi_0 + \psi_1 \circ \varphi_1 = 1.
\end{equation*}
We consider the summand $\psi_1 \circ \varphi_1 \colon \cO(t - m) \to \cO(t - m)$.

First, assume $\psi_1 \circ \varphi_1 \ne 0$.
Then it is an isomorphism, hence $\varphi_1$ is a split monomorphism, hence an isomorphism, hence $t - m = -t$ and so $m = 2t$.
Furthermore, it follows that the restriction of~$q$ to the subbundle~$\cE_1 = \cO(t)$ of~$\cE$ is unimodular.
Taking~$\cE_0 = \cE_1^\perp$ to be the orthogonal of~$\cE_1$ in~$\cE$, we obtain the required direct sum decomposition.

Next, assume $\psi_1 \circ \varphi_1 = 0$.
Then it follows that $\psi_0 \circ \varphi_0 = 1$, hence $\varphi_0$ is a split monomorphism.
Therefore, we have $\cE' \cong \cE_0 \oplus \cO(m - t)$, so that $\cE = \cE_0 \oplus \cO(m - t) \oplus \cO(t)$.
Furthermore, it follows that the restriction of~$q$ to the subbundle~$\cE_1 = \cO(m - t) \oplus \cO(t)$ of~$\cE$ is unimodular
(the restriction to~$\cO(m-t)$ is zero and the pairing between~$\cO(m-t)$ and~$\cO(t)$ is a non-zero constant).
Taking~$\cE_0 = \cE_1^\perp$ to be the orthogonal of~$\cE_1$ in~$\cE$, we obtain the required direct sum decomposition.
\end{proof}

\begin{corollary}
\label{cor:trivial-hacm}
If $(\cE,q)$ is a unimodular quadratic form and~$\cE$ is VLC then~$(\cE,q)$ is isomorphic 
to a standard unimodular quadratic form~\eqref{eq:trivial-linear};
in particular, $\cE$ is split.
\end{corollary}

\begin{proof}
Since $q$ is unimodular, we have $\cE(-m) \cong \cE^\vee$, so if~$\cE$ is VLC, and hence~$\cE^\vee$ is VUC,
then~$\cE$ is both VLC and VUC, hence it is split by Lemma~\ref{lemma:hacm-duality}.
Furthermore, $q \colon \cE(-m) \to \cE^\vee$ is an isomorphism of split bundles, hence it is not linearly minimal.
Applying Lemma~\ref{lemma:linear-summand} we obtain a direct sum decomposition~\mbox{$\cE = \cE_0 \oplus \cE_1$},
where~$\cE_1$ is standard unimodular of type~\eqref{eq:trivial-linear}
and~$\cE_0$, being a direct summand of a unimodular VLC quadratic form, is itself unimodular and VLC.
Iterating the argument, we conclude that~$\cE_0$ is standard unimodular of type~\eqref{eq:trivial-linear}, hence so is~$\cE$. 
\end{proof}

\begin{lemma}
\label{lemma:standard-anisotropic}
If $(\cE,q)$ is a standard unimodular quadratic form of type~\eqref{eq:trivial-linear} or~\eqref{eq:trivial-omega},
it is hyperbolic equivalent to one of the following
\begin{itemize}
\item 
$(W^0,q_{W^0}) \otimes \cO(m/2)$, if $m$ is even, or
\item 
$(W^0,q_{W^0}) \otimes \Omega^{n/2}((m + n + 1)/2)$, if~$m$ is odd and~$n$ is divisible by~$4$
\end{itemize}
\textup(where in each case~$(W^0,q_{W^0})$ is an anisotropic quadratic space\textup), or to zero, otherwise.
\end{lemma}

\begin{proof}
By definition of standard unimodular quadratic forms for each~$i \ne 0$ the summands
\begin{equation*}
W^i \otimes \cO((m+i)/2) \oplus W^{-i} \otimes \cO((m-i)/2)
\quad\text{or}\quad 
W^i \otimes \Omega^{n/2}((m+m+1+i)/2) \oplus W^{-i} \otimes \Omega^{n/2}((m+n+1-i)/2)
\end{equation*}
are hyperbolic equivalent to zero, hence any standard unimodular quadratic form 
is hyperbolic equivalent to the one with~\mbox{$W = W^0$}.
It remains to note that by the standard Witt theory the bilinear form~$(W^0,q_{W^0})$ is hyperbolic equivalent to an anisotropic form.
Finally, in the case where~$m$ is odd and~$n \equiv 2 \bmod 4$ the form~$q_{W^0}$ is skew-symmetric, 
so if it is anisotropic, it is just zero.
\end{proof}

Now we are ready to prove the main result of this section.
Recall Definition~\ref{def:shadow}.

\begin{theorem}
\label{thm:he-hacm}
Any generically non-degenerate quadratic form~$q \colon \cE(-m) \to \cE^\vee$
over~$\P^n$ is hyperbolic equivalent to an orthogonal direct sum 
\begin{equation}
\label{eq:min-uni}
(\cE_\mn,q_\mn) \oplus (\cE_\un,q_\un),
\end{equation}
where~$\cE_\mn$ is a VLC bundle, $(\cE_\mn,q_\mn)$ has no unimodular direct summands, 
and~$(\cE_\un,q_\un)$ is an anisotropic standard unimodular quadratic form
which has type~\eqref{eq:trivial-linear} if~$m$ is even,
type~\eqref{eq:trivial-omega} if~$m$ is odd and~$n \equiv 0 \bmod 4$,
and is zero otherwise.

Moreover, if~$n = 2k + 1$, $\cC = \cC(q)$ is the cokernel sheaf of~$(\cE,q)$, and~$\rA^k \subset \HH^k(\cC)$ 
is any shadowless subspace which is Lagrangian with respect to the bilinear form~\eqref{eq:hh-k-cf-pairing}  
then the quadratic form~$(\cE_\mn,q_\mn)$ in~\eqref{eq:min-uni} can be chosen in such a way 
that there is an equality~$\HH^k(\cE_\mn^\vee) = \rA^k$ of $\SS$-submodules in~$\HH^k(\cC)$. 
\end{theorem}

\begin{proof}
We split the proof into a number of steps.

{\bf Step 1.}
First we show that $q$ is hyperbolic equivalent to a quadratic form~$(\cE_1,q_1)$ 
such that~\mbox{$\HH^i(\cE_1) = 0$} for each $1 \le i \le \lfloor (n - 1)/2 \rfloor$
(if~$n$ is odd this is equivalent to the VLC property, and if~$n$ is even this is a bit weaker).
For this we use induction on the parameter 
\begin{equation*}
\label{def:ell-ce}
\hl_1(\cE) := \sum_{i = 1}^{\lfloor (n - 1)/2 \rfloor} \dim \HH^i(\cE).
\end{equation*}
Note that~$\hl_1(\cE) < \infty$ for any vector bundle~$\cE$.

Assume $\hl_1(\cE) > 0$.
Let $1 \le p_0 \le \lfloor (n - 1)/2 \rfloor$ be the minimal integer such that $\HH^{p_0}(\cE) \ne 0$ and
let~$t_0$ be the maximal integer such that $H^{p_0}(\P^n,\cE(t_0)) \ne 0$.
Choose a non-zero element $\eps_{p_0} \in H^{p_0}(\P^n,\cE(t_0))$.
Note that the class $q(\eps_{p_0}, \eps_{p_0}) \in H^{2p_0}(\P^n, \cO(2t_0 + m))$ vanishes because $2 \le 2p_0 \le n - 1$.
Note also that the conditions~\eqref{eq:cohomology-assumption} and~\eqref{item:eps-2n} are satisfied for~$\eps_{p_0}$.
Let~$(\cE_+,q_+)$ be the elementary modification of~$(\cE,q)$ with respect to~$\eps_{p_0}$
constructed in Proposition~\ref{prop:elementary-modification}. 
Then~$(\cE_+,q_+)$ is hyperbolic equivalent to~$(\cE,q)$ and the formula~\eqref{eq:hh-ce+} 
implies that~\mbox{$\hl_1(\cE_+) = \hl_1(\cE) - 1$}.
Indeed, $n - p_0 + 1 > \lfloor (n - 1)/2 \rfloor$, 
so even if the extra cohomology class~$\eps_+$ appears in~$\HH(\cE_+)$
it does not contribute to~$\hl_1(\cE_+)$.
By induction hypothesis the quadratic form~$(\cE_+,q_+)$ is hyperbolic equivalent to a quadratic form~$(\cE_1,q_1)$ 
such that~\mbox{$\HH^i(\cE_1) = 0$} for each~\mbox{$1 \le i \le \lfloor (n - 1)/2 \rfloor$}, hence so is~$(\cE,q)$.

From now on we assume that~$\hl_1(\cE) = 0$ and discuss separately the case of even and odd~$n$.

{\bf Step 2.}
Assume that $n = 2k$.
In this case $\lfloor (n-1)/2 \rfloor = k - 1 < k = \lfloor n/2 \rfloor$, 
hence by Step~1 the only non-trivial intermediate cohomology of~$\cE$ preventing it from being VLC is~$\HH^k(\cE)$
and it fits into the exact sequence
\begin{equation*}
0 \to \HH^{k-1}(\cE^\vee) \to \HH^{k-1}(\cC) \to \HH^k(\cE) \xrightarrow{\ \HH^k(q)\ } \HH^k(\cE^\vee) \to \HH^k(\cC) \to \HH^{k+1}(\cE) \to 0,
\end{equation*}
where~$\HH^k(q)$ is the map induced by~$q$. 
Note also that the combination of the morphism~$\HH^k(q)$ with the Serre duality pairing 
is a graded $\SS$-bilinear form
\begin{equation*}
\HH^k(q) \colon \HH^k(\cE) \otimes \HH^k(\cE) \to \kk(m + n + 1),
\end{equation*}
which is symmetric if~$k$ is even and skew-symmetric if~$k$ is odd.

First, we show that $(\cE,q)$ is hyperbolic equivalent to a quadratic form~$(\cE',q')$ such that~$\hl_1(\cE') = 0$,
the form~$\HH^k(q')$ is non-degenerate, and~$H^k(\P^n, \cE(t)) = 0$ unless~$t = -(m + n + 1)/2$.

If~$\Ker(\HH^k(q)) \ne 0$ let~$t$ be the maximal integer such that~$\Ker(\HH^k(q)) \cap H^k(\P^n,\cE(t)) \ne 0$
and let~$\eps_k$ be any non-zero class in this space (note that~$q(\eps_k,\eps_k) = 0$);
otherwise let~$t$ be the maximal integer such that~$H^k(\P^n,\cE(t)) \ne 0$ 
and let~$\eps_k$ be any non-zero class in this space such that~$q(\eps_k,\eps_k) = 0$ (if it exists). 
As before conditions~\eqref{eq:cohomology-assumption} and~\eqref{item:eps-2n} are satisfied for~$\eps_k$.
Applying the elementary modification of Proposition~\ref{prop:elementary-modification} 
we obtain a quadratic form~$(\cE_+,q_+)$ hyperbolic equivalent to~$(\cE,q)$ 
and such that 
\begin{equation*}
\HH^k(\cE_+) = \HH^k(\cE)/\kk\eps_k,
\qquad\text{and}\qquad 
\HH^p(\cE_+) = \HH^p(\cE) = 0
\quad\text{for}\quad 
1 \le p < k,
\end{equation*}
in particular~$\hl_1(\cE_+) = 0$ and~$\dim(\HH^k(\cE_+)) < \dim(\HH^k(\cE))$.
Iterating this argument we eventually obtain a quadratic form~$(\cE',q')$ 
such that~$\hl_1(\cE') = 0$, the form~$\HH^k(q')$ is non-degenerate, 
and if~$t$ is the maximal integer such that~$H^k(\P^n, \cE'(t)) \ne 0$
then~$q'(\eps_k,\eps_k) \ne 0$ for any ~$0 \ne \eps_k \in H^k(\P^n, \cE'(t))$.

If~$\HH^k(\cE') = 0$ there is nothing to prove anymore.
Otherwise, the condition~$q'(\eps_k,\eps_k) \ne 0$ implies that
\begin{equation*}
2t + m + n + 1 = 0.
\end{equation*}
It remains to note that~$H^k(\P^n, \cE'(s)) = 0$ for~$s \ne t$.
Indeed, for~$s > t$ the vanishing holds by definition of~$t$.
On the other hand, we have 
\begin{equation*}
H^k(\P^n, \cE'(s)) = H^k(\P^n, {\cE'}^\vee(m + s)) = H^k(\P^n, \cE'(-s - m - n - 1))^\vee
\end{equation*}
(the first equality follows from non-degeneracy of~$\HH^k(q')$ and the second from Serre duality),
and as the right-hand side vanishes for~$-s - m - n - 1 > t$, 
the left-hand side vanishes for~$s < - t - m - n - 1 = t$.

Now, replacing~$(\cE,q)$ by~$(\cE',q')$, we may assume that~$\HH^k(q)$ is non-degenerate 
and~$H^k(\P^n, \cE(t)) = 0$ unless~$t = -(m + n + 1)/2$.
So, we set~$t :=  -(m + n + 1)/2$ and let
\begin{equation*}
\beps_k \colon H^k(\P^n,\cE(t)) \otimes \cO(-t)[-k] 
= \Ext^k(\cO(-t), \cE) \otimes \cO(-t)[-k] 
\to \cE
\end{equation*}
be the evaluation morphism in the derived category.
Let 
\begin{equation*}
\beps_0 \colon H^k(\P^n,\cE(t)) \otimes \Omega^k(-t) 
= \Hom(\Omega^k(-t), \cE) \otimes \Omega^k(-t) 
\to \cE
\end{equation*}
be the morphism constructed from~$\beps_k$ in Proposition~\ref{prop:eps-p}.
Consider the composition 
\begin{equation}
\label{eq:omega-ce-cevee-ct}
H^k(\P^n,\cE(t)) \otimes \Omega^k(-t) \xrightarrow{\ \beps_0\ }
\cE \xrightarrow{\ q\ } \cE^\vee(m) \xrightarrow{\ \beps_0^\vee\ }
H^k(\P^n,\cE(t))^\vee \otimes \bw{k}\cT(t+m).
\end{equation}
By Proposition~\ref{prop:eps-p} and Serre duality the first and last arrows in the composition
\begin{equation*}
\HH^k\left( H^k(\P^n,\cE(t)) \otimes \Omega^k(-t) \right) \xrightarrow{ \beps_0 } 
\HH^k(\cE) \xrightarrow{ \HH^k(q) }
\HH^k(\cE^\vee) \xrightarrow{ \beps_0^\vee }
\HH^k\left( H^k(\P^n,\cE(t))^\vee \otimes \bw{k}\cT(t+m) \right)
\end{equation*}
are isomorphisms, while the middle arrow is an isomorphism by non-degeneracy of~$\HH^k(q)$.
It follows that the composition~\eqref{eq:omega-ce-cevee-ct} is an isomorphism 
(note that~$\bw{k}\cT(t+m) \cong  \Omega^k(t+m+n+1) \cong \Omega^k(-t)$ since~$n = 2k$ and~$2t + m + n + 1 = 0$).
This means that~$\beps_0$ is a split monomorphism, i.e., 
\begin{equation*}
\cE \cong \cE_0 \oplus \cE_1,
\qquad 
\cE_1 = H^k(\P^n,\cE(t)) \otimes \Omega^{k}(-t) ,
\end{equation*}
so that~$\cE_1$ is a standard unimodular bundle of type~\eqref{eq:trivial-omega}
and~$\cE_0$ is the orthogonal of~$\cE_1$ with respect to the quadratic form~$q$.
From the direct sum decomposition it easily follows that~$\cE_0$ is VLC.

{\bf Step 3.}
Assume $n = 2k + 1$.
In this case $\lfloor (n-1)/2 \rfloor = k = \lfloor n/2 \rfloor$, 
hence by Step~1 the bundle~$\cE$ is already~VLC.
It remains to find a VLC quadratic form~$(\cE_0,q_0)$ 
hyperbolic equivalent to~$(\cE,q)$ with~\mbox{$\HH^k(\cE_0^\vee) = \rA^k$},
where recall that~$\rA^k \subset \HH^k(\cC)$ is a given shadowless Lagrangian subspace.  
To construct~$\cE_0$ we induct on the parameter
\begin{equation}
\label{def:ell-ra-ce}
\hl_\rA(\cE) 
:= \dim (\Ima(\rA^k \to \HH^{k+1}(\cE))) 
 = \codim_{\rA^k}(\HH^k(\cE^\vee) \cap \rA^k),
\end{equation} 
where the equality follows from the exact sequence~\eqref{eq:lagrangian-cohomology}.

Assume $\hl_\rA(\cE) > 0$,
choose a non-zero homogeneous element~$\eps_{k+1} \in \Ima(\rA^k \to \HH^{k+1}(\cE))$ of maximal degree 
and let~$\teps_{k+1}$ be its arbitrary homogeneous lift to $\rA^k \subset \HH^k(\cC)$.
Recall the definition of the hyperplane~$\eps_{k+1}^\perp \subset \HH^k(\cE^\vee)$ 
that was given before Proposition~\ref{prop:em-lagrangian}.
Note that
\begin{equation*}
\eps_{k+1}^\perp = \HH^k(\cE^\vee) \cap \teps_{k+1}^\perp,
\end{equation*}
where~$\teps_{k+1}^\perp \subset \HH^k(\cC)$ is the hyperplane orthogonal of~$\teps_{k+1}$ 
with respect to the perfect pairing~\eqref{eq:hh-k-cf-pairing};
in particular~$\eps_{k+1}^\perp$ is isotropic and orthogonal to~$\teps_{k+1}$,
hence the subspace
\begin{equation*}
\rA_+^k := \eps_{k+1}^\perp \oplus \kk\teps_{k+1} \subset \HH^k(\cC)
\end{equation*}
is Lagrangian. 
Applying Proposition~\ref{prop:em-lagrangian} we conclude that there exists 
a refined elementary modification~$(\cE_+,q_+)$ of~$(\cE,q)$ 
which is VLC and has the property~$\HH^k(\cE_+^\vee) = \rA_+^k$.
Now it is easy to see that
\begin{equation*}
\HH^k(\cE^\vee) \cap \rA^k \subset \HH^k(\cE^\vee) \cap \teps_{k+1}^\perp = \eps_{k+1}^\perp \subset \rA_+^k
\end{equation*}
(the first inclusion follows from~$\teps_{k+1} \in \rA^k$ since~$\rA^k$ is Lagrangian) and 
\begin{equation*}
\teps_{k+1} \in (\rA_+^k \cap \rA^k) \setminus (\HH^k(\cE^\vee) \cap \rA^k),
\end{equation*}
hence~$\dim(\rA_+^k \cap \rA^k) > \dim(\HH^k(\cE^\vee) \cap \rA^k)$ and so~$\hl_\rA(\cE_+) < \hl_\rA(\cE)$.
By the induction hypothesis the quadratic form~$(\cE_+,q_+)$ is hyperbolic equivalent to~$(\cE_0,q_0)$ 
such that~$\HH^k(\cE_0^\vee) = \rA^k$, hence so is~$(\cE,q)$.

{\bf Step 4.} 
We already have proved that the quadratic form $(\cE,q)$ is hyperbolic equivalent 
to an orthogonal direct sum $(\cE_0,q_0) \oplus (\cE_1,q_1)$, 
where~$\cE_0$ is VLC (with prescribed Lagrangian subspace~$\HH^k(\cE_0^\vee) \subset\HH^k(\cC)$ if $n = 2k + 1$) 
and~$(\cE_1,q_1)$ is standard unimodular of type~\eqref{eq:trivial-omega} (if $n = 2k$).

Obviously we can write $(\cE_0,q_0) \cong (\cE_\mn,q_\mn) \oplus (\cE_2,q_2)$, 
where $(\cE_\mn,q_\mn)$ has no unimodular direct summands and~$(\cE_2,q_2)$ is unimodular.
Then, defining
\begin{equation*}
(\cE_\un,q_\un) := (\cE_1,q_1) \oplus (\cE_2,q_2)
\end{equation*}
we obtain a decomposition of type~\eqref{eq:min-uni} and it remains to modify it slightly.

First, note that since~$\cE_2$ is a direct summand of~$\cE_0$, it is VLC, 
hence~$(\cE_2,q_2)$ is standard of type~\eqref{eq:trivial-linear} by Corollary~\ref{cor:trivial-hacm}.
Second, by Lemma~\ref{lemma:standard-anisotropic} we can replace the summands~$(\cE_1,q_1)$ and~$(\cE_2,q_2)$ above 
by summands of the same type with~$W^i = 0$ for~$i \ne 0$ and~$q_{W^0}$ anisotropic.
It remains to note that we have~\mbox{$0 = i \equiv m \bmod 2$} for the summand of type~\eqref{eq:trivial-linear} 
and~\mbox{$0 = i \equiv m + n + 1 \equiv m + 1 \bmod 2$} and~$n$ is even for the summand of type~\eqref{eq:trivial-omega}.
Moreover, in the latter case if~$n$ is not divisible by~4, the form~$q_{W^0}$ is skew-symmetric, 
and since it is also anisotropic, $W^0 = 0$.
Thus, we obtain the required description of~$(\cE_\un,q_\un)$.
\end{proof}

In the following corollary we deduce from Theorem~\ref{thm:he-hacm} a generalization of the result of Arason~\cite{Ara} 
about the untwisted unimodular Witt group~$\rW(\P^n,\cO)$ of a projective space
to the case of the twisted unimodular group~$\rW(\P^n,\cO(m))$;
thus reproving a result of Walter~\cite{Wa} (see also~\cite[Theorem~1.5.28]{Balmer}) 
in the special case of trivial base.

\begin{corollary}
If~$m$ is even, or if~$m$ is odd and~$n$ is divisible by~$4$, one has~$\rW(\P^n,\cO(m)) \cong \rW(\kk)$.
Otherwise~$\rW(\P^n,\cO(m)) = 0$.
\end{corollary}

\begin{proof}
By Theorem~\ref{thm:he-hacm} a unimodular quadratic form~$(\cE,q)$ on~$\P^n$ 
is hyperbolic equivalent to a sum~\eqref{eq:min-uni}.
The summand~$(\cE_\mn,q_\mn)$ is unimodular (as a direct summand of a unimodular quadratic form) 
and has no unimodular summands by assumption, 
hence~$(\cE_\mn,q_\mn) = 0$ and~$(\cE,q) = (\cE_\un,q_\un)$ 
is an anisotropic standard unimodular quadratic form of type~\eqref{eq:trivial-linear} or~\eqref{eq:trivial-omega}.

If $m$ is even, $(\cE_\un,q_\un) \cong (W^0,q_{W^0}) \otimes \cO(m/2)$ and~$\rw_x(\cE,q) = \rw_x(\cE_\un,q_\un) = [(W^0,q_{W^0})]$
(the first equality follows from Lemma~\ref{lemma:he-rank},
and for the second to be true one has to choose the trivialization of~$\cO(m)_x$ 
to be induced by a trivialization of~$\cO(m/2)_x$)
for any $\kk$-point~$x \in \P^n$, hence the group homomorphism
\begin{equation*}
\rw_x \colon \rW(\P^n,\cO(m)) \to \rW(\kk)
\end{equation*}
is injective. 
On the other hand, it is obviously surjective, hence it is an isomorphism.

Next, assume~$m$ is odd and~$n$ is divisible by~$4$.
Then~$(\cE_\un,q_\un) = (W^0,q_{W^0}) \otimes \Omega^{n/2}((m + n + 1)/2)$,
and~$\hw(\cE,q) = \hw(\cE_\un,q_\un) = [(W^0,q_{W^0})]$
(the first equality follows from Lemma~\ref{lemma:he-rank-cl-prime}), hence the group homomorphism
\begin{equation*}
\hw \colon \rW(\P^n,\cO(m)) \to \rW(\kk)
\end{equation*}
is injective. 
On the other hand, it is obviously surjective, hence it is an isomorphism.

In the remaining cases~$\cE_\un = 0$ by Theorem~\ref{thm:he-hacm}, hence~$\rW(\P^n,\cO(m)) = 0$. 
\end{proof}

\subsection{Proof of Theorem~\ref{thm:he-intro} and Corollary~\ref{corollary:invariants}}
\label{subsec:proofs}

In this final subsection we prove Theorem~\ref{thm:he-intro} and Corollary~\ref{corollary:invariants} from the Introduction.
The next proposition provides the crucial step 

\begin{proposition}
\label{prop:qb-isomorphism}
Assume~$(\cE_1,q_1)$ and~$(\cE_2,q_2)$ are generically non-degenerate quadratic forms 
which have no unimodular direct summands and such that the bundles~$\cE_i$ are VLC.
Let~$\varphi \colon \cC(q_1) \xrightarrow\simeq \cC(q_2)$ be an isomorphism of their cokernel sheaves 
compatible with their induced shifted quadratic forms~\eqref{eq:cokernel-form}.
If~$n = 2k + 1$ assume also that~$\HH^k(\varphi)$ identifies the Lagrangian subspaces~$\HH^k(\cE_i^\vee) \subset \HH^k(\cC(q_i))$.
Then~$\varphi$ is induced by a unique isomorphism~$(\cE_1,q_1) \cong (\cE_2,q_2)$ of quadratic forms.
\end{proposition}

\begin{proof}
By Lemma~\ref{lemma:linear-summand} the VHC morphisms $\cE_{i}(-m) \xrightarrow{\ q_{i}\ } \cE_{i}^\vee$
are linearly minimal resolutions of the sheaves~$\cC(q_i)$, hence by Theorem~\ref{prop:hacm-uniqueness} they are isomorphic, 
i.e., there is a commutative diagram
\begin{equation*}
\xymatrix{
0 \ar[r] &
\cE_1(-m) \ar[r]^-{q_1} \ar[d]_{\varphi_\la} &
\cE_1^\vee \ar[r] \ar[d]_{\varphi_\ua} \ar@{..>}[dl] &
\cC(q_1) \ar[r] \ar[d]_\varphi &
0
\\
0 \ar[r] &
\cE_2(-m) \ar[r]^-{q_2} &
\cE_2^\vee \ar[r] &
\cC(q_2) \ar[r] &
0,
}
\end{equation*}
where~$\varphi_\la$ and~$\varphi_\ua$ are isomorphisms.
Moreover, such diagram is unique up to a homotopy represented by the dotted arrow.
From now on we identify~$\cE_2$ with~$\cE_1$ by means of~$\varphi_\la$, so we assume~$\cE_1 = \cE_2$ and~$\varphi_\la = \id$.
Now consider the dual diagram, and then invert its vertical arrows:
\begin{equation*}
\xymatrix{
0 \ar[r] &
\cE_1(-m) \ar[r]^-{q_1} \ar[d]_{(\varphi_\ua^\vee)^{-1}} &
\cE_1^\vee \ar[r] \ar[d]_{\id} \ar@{..>}[dl]_h &
\cC(q_1) \ar[r] \ar[d]_{(\varphi^\vee)^{-1}} &
0
\\
0 \ar[r] &
\cE_2(-m) \ar[r]^-{q_2} &
\cE_2^\vee \ar[r] &
\cC(q_2) \ar[r] &
0,
}
\end{equation*}
Since~$\varphi$ is compatible with the shifted quadratic forms on~$\cC(q_i)$, we have~$(\varphi^\vee)^{-1} = \varphi$, 
hence by the uniqueness property of the diagram, there is a homotopy~$h$ such that
\begin{equation*}
\varphi_\ua = \id + q_2 \circ h,
\qquad
\id = h \circ q_1 + (\varphi_\ua^\vee)^{-1}.
\end{equation*}
Now note that the endomorphism~$q_2 \circ h$ of~$\cE_1^\vee = \cE_2^\vee$ is nilpotent, again by Theorem~\ref{prop:hacm-uniqueness}.
Therefore, $\varphi_\ua$ is unipotent.
On the other hand, commutativity of the first diagram (with the convention~$\varphi_\la = \id$ taken into account) means that
\begin{equation*}
q_2 = \varphi_\ua \circ q_1,
\end{equation*}
and since~$q_2$ is self-dual, it follows that~$\varphi_\ua$ is self-adjoint with respect to~$q_1$.

Now note that if a unipotent operator over a field is self-adjoint with respect to a non-degenerate quadratic form, 
it is the identity.
Indeed, to prove this we can pass to an algebraic closure of the field,
then the operator can be diagonalized, and a diagonal operator is unipotent only if it is the identity.

The above argument thus shows that~$\varphi_\ua$ restricted to the generic point of~$\P^n$ is the identity.
Finally, since the bundle~$\cE_2 = \cE_1$ is torsion free and~$\varphi_\ua$ is an automorphism, 
it follows that~$\varphi_\ua$ is the identity.
Thus, $q_2 = q_1$, i.e., the quadratic forms~$(\cE_i,q_i)$ are isomorphic.
\end{proof}

Now we can deduce the theorem.

\begin{proof}[Proof of Theorem~\textup{\ref{thm:he-intro}}]
Let $\cC = \cC(q_1) = \cC(q_2)$.
If~\mbox{$n = 2k + 1$} define~$\rA^k \subset \HH^k(\cC)$ by the formula~\eqref{eq:ra-k-typical}, where 
\begin{equation*}
\rA^k_{(m-n-1)/2} := H^k(\P^n,\cE_2^\vee((m-n-1)/2)) \subset H^k(\P^n,\cC((m-n-1)/2);
\end{equation*}
this is a shadowless Lagrangian~$\SS$-submodule as explained in Remark~\ref{rem:ra-k}.
By Theorem~\ref{thm:he-hacm} the quadratic forms~$(\cE_i,q_i)$ are hyperbolic equivalent to orthogonal direct sums 
\begin{equation*}
(\cE_{i,\mn},q_{i,\mn}) \oplus (\cE_{i,\un},q_{i,\un}),
\end{equation*}
where~$\cE_{i,\mn}$ are VLC bundles, $(\cE_{i,\mn},q_{i,\mn})$ have no unimodular direct summands,
$\HH^k(\cE_{i,\mn}^\vee) = \rA^k$ if~$n = 2k + 1$,
and~$(\cE_{i,\un},q_{i,\un})$ are anisotropic standard unimodular quadratic forms~\eqref{eq:trivial-linear} 
or~\eqref{eq:trivial-omega}.

Since the summands~$(\cE_{i,\un},q_{i,\un})$ are unimodular, i.e., $\cC(q_{i,\un}) = 0$, we have
\begin{equation*}
\cC(q_{i,\mn}) = 
\cC(q_{i,\mn} \oplus q_{i,\un}) \cong
\cC(q_i) = 
\cC,
\end{equation*}
where the isomorphism is induced by the hyperbolic equivalence, see Proposition~\ref{prop:he-invariants}\eqref{item:he-cokernel},
hence it is compatible with the shifted quadratic forms~\eqref{eq:cokernel-form}.
Moreover, we have~$\HH^k(\cE_{1,\mn}^\vee) = \HH^k(\cE_{2,\mn}^\vee)$ if~\mbox{$n = 2k + 1$}.
Applying Proposition~\ref{prop:qb-isomorphism} we conclude that
\begin{equation}
\label{eq:min-isomorphic}
(\cE_{1,\mn},q_{1,\mn}) \cong (\cE_{2,\mn},q_{2,\mn}).
\end{equation} 
Next, we identify the unimodular summands~$(\cE_{1,\un},{q}_{1,\un})$ and~$(\cE_{2,\un},{q}_{2,\un})$.

If $m$ is even, the unimodular summands have type~\eqref{eq:trivial-linear}, i.e., they can be written as
\begin{equation*}
(\cE_{i,\un},{q}_{i,\un}) \cong (W_i^0,q_{W_i^0}) \otimes \cO(m/2),
\end{equation*}
where~$(W_i^0,q_{W_i^0})$ are anisotropic.
Moreover, we have 
\begin{equation*}
\rw_x(\cE_i,q_i) 
= \rw_x(\cE_{i,\mn},q_{i,\mn}) + \rw_x(\cE_{i,\un},q_{i,\un})
= \rw_x(\cE_{i,\mn},q_{i,\mn}) + [(W_i^0,q_{W_i^0})],
\end{equation*}
hence~\eqref{eq:min-isomorphic} and the equality~$\rw_x(\cE_1,q_1) = \rw_x(\cE_2,q_2)$ imply the equality~$[(W_1^0,q_{W_1^0})] = [(W_2^0,q_{W_2^0})]$ 
of the Witt classes of the quadratic spaces~$(W_i^0,q_{W_i^0})$.
But since these quadratic spaces are anisotropic, they are isomorphic,
hence we have~$(\cE_{1,\un},q_{1,\un}) \cong (\cE_{2,\un},q_{2,\un})$.

If $m$ is odd and $n$ is divisible by~4, the unimodular summands have type~\eqref{eq:trivial-omega}, i.e., they can be written as
\begin{equation*}
(\cE_{i,\un},q_{i,\un}) \cong (W_i^0,q_{W_i^0}) \otimes \Omega^{n/2}((m + n + 1)/2),
\end{equation*}
where again~$(W_i^0,q_{W_i^0})$ are anisotropic.
Moreover, we have 
\begin{equation*}
\hw(\cE_i,q_i) 
= \hw(\cE_{i,\mn},q_{i,\mn}) + \hw(\cE_{i,\un},q_{i,\un})
= \hw(\cE_{i,\mn},q_{i,\mn}) + [(W_i^0,q_{W_i^0})],
\end{equation*}
hence~\eqref{eq:min-isomorphic} and the equality~$\hw(\cE_1,q_1) = \hw(\cE_2,q_2)$ imply the equality~$[(W_1^0,q_{W_1^0})] = [(W_2^0,q_{W_2^0})]$ 
of the Witt classes of the quadratic spaces~$(W_i^0,q_{W_i^0})$.
Again, since these quadratic spaces are anisotropic, they are isomorphic,
hence we have~$(\cE_{1,\un},q_{1,\un}) \cong (\cE_{2,\un},q_{2,\un})$.

Finally, if $m$ is odd and~$n$ is not divisible by~$4$, we have~$(\cE_{i,\mn},q_{i,\mn}) = 0$ for~$i = 1,2$.

Thus, in all the cases we have~$(\cE_{1,\un},{q}_{1,\un}) \cong (\cE_{2,\un},{q}_{2,\un})$.
Combining this isomorphism with~\eqref{eq:min-isomorphic}, we obtain a chain of hyperbolic equivalences and isomorphisms
\begin{equation*}
(\cE_1,q_1) 
\stackrel{\mathrm{he}}\sim (\cE_{1,\mn},q_{1,\mn}) \oplus (\cE_{1,\un},q_{1,\un}) 
\cong (\cE_{2,\mn},q_{2,\mn}) \oplus (\cE_{2,\un},q_{2,\un}) 
\stackrel{\mathrm{he}}\sim (\cE_2,q_2)
\end{equation*}
and conclude that~$(\cE_1,q_1)$ is hyperbolic equivalent to~$(\cE_2,q_2)$.
\end{proof}

Before proving Corollary~\ref{corollary:invariants} let us recall the definitions of the discriminant double cover and root stack
associated with a generically non-degenerate quadric bundle~$Q \subset \P_X(\cE)$,
and of the Brauer classes on their corank~$\le 1$ loci.

First, assume that $\dim(Q/X) \equiv 0 \bmod 2$.
The determinant of the morphism $q \colon \cE \otimes \cL \to \cE^\vee$ is a non-zero global section~$\det(q)$
of the line bundle $\big((\det\cE)^{\otimes 2} \otimes \cL^{\otimes \rk(\cE)}\big)^\vee$;
it defines a $\ZZ/2$-graded commutative $\cO_X$-algebra structure on the sheaf~$\cO_X \oplus \det\cE \otimes \cL^{\otimes \rk(\cE)/2}$, 
and the determinant double cover is defined as its relative spectrum
\begin{equation*}
S = \Spec_X \left( \cO_X \oplus \det\cE \otimes \cL^{\otimes \rk(\cE)/2} \right).
\end{equation*}
On the other hand, by~\cite[\S3.5]{K08} the algebra~$\cO_X \oplus \det\cE \otimes \cL^{\otimes \rk(\cE)/2}$ 
is identified with the center of the even part of the Clifford algebra~$\Cl_0(\cE,q)$,
which therefore can be written as the pushforward of an~$\cO_S$-algebra~$\cB_0$ from~$S$.
Moreover, the restriction of~$\cB_0$ to the open subset~$S_{\le 1} \subset S$, 
the preimage of the open subset~$X_{\le 1} \subset X$ parameterizing quadrics of corank at most~1 
in the quadric bundle~$Q \to X$,
is an Azumaya algebra.
We denote by~$\upbeta_S \in \Br(S_{\le 1})$ the Brauer class of~$\cB_0\vert_{S_{\le 1}}$.

Similarly, assume $\dim(Q/X) \equiv 1 \bmod 2$.
Locally over~$X$ we can trivialize the line bundle~$\cL$; 
then using~$\det(q)$ in the same way as above we can define local double covers of~$X$, 
which do not necessarily glue into a global double cover, 
but whose quotient stacks by the covering involutions glue into a global stack $S \to X$.
This is, in fact, the root stack 
\begin{equation*}
S = \sqrt[2]{\left.\left( \big((\det\cE)^{\otimes 2} \otimes \cL^{\otimes \rk(\cE)}\big)^\vee, \det(q)\right) \right/ X}
\end{equation*}
as defined in~\cite[\S B.2]{AGV}.
By~\cite[\S3.6]{K08} the algebra~$\Cl_0(\cE,q)$ can be written as the pushforward of an~$\cO_S$-algebra~$\cB_0$ from~$S$.
Moreover, the restriction of~$\cB_0$ to the open substack~$S_{\le 1} \subset S$, 
the preimage of the open subset~$X_{\le 1} \subset X$,
is an Azumaya algebra.
We denote by~$\upbeta_S \in \Br(S_{\le 1})$ the Brauer class of~$\cB_0\vert_{S_{\le 1}}$.

\begin{proof}[Proof of Corollary~\textup{\ref{corollary:invariants}}]
Since the field~$\kk$ is algebraically closed, we have~$\rW(\kk) = \ZZ/2$ via the dimension parity homomorphism,
hence the assumptions~$\rw_x(\cE,q) = \rw_x(\cE',q')$ and~$\hw(\cE,q) = \hw(\cE',q')$ of Theorem~\ref{thm:he-intro}
reduce to the equality~$\rk(\cE) \equiv \rk(\cE') \bmod 2$, which holds true in each part of the corollary,
and to~\mbox{$\rk(H^{n/2}(q)) \equiv \rk(H^{n/2}(q')) \bmod 2$}, which is one of the assumptions of the corollary, respectively.
Therefore, by Theorem~\ref{thm:he-intro}, the quadric bundles~$Q$ and~$Q'$ are hyperbolic equivalent.
Now parts~\eqref{item:cor-brauer-even} and~\eqref{item:cor-brauer-odd} of the corollary follow 
from the Morita equivalence of~$\Cl_0(\cE,q)$ and~$\Cl_0(\cE',q')$ proved in Proposition~\ref{prop:he-invariants}\eqref{item:he-clifford},
part~\eqref{item:cor-l-equivalence} follows from Proposition~\ref{prop:he-invariants}\eqref{item:l-equivalence},
and part~\eqref{item:cor-birational-equivalence} from Proposition~\ref{prop:he-invariants}\eqref{item:he-witt}.
\end{proof}

\end{document}